\documentclass[a4paper,twoside,10pt]{article}
\usepackage[a4paper,left=3cm,right=3cm, top=3cm, bottom=3cm]{geometry}
\usepackage{soul}
\usepackage{mathrsfs}
\usepackage{amsmath}
\usepackage{amsthm}
\usepackage{amssymb}
\usepackage{cite}
\usepackage{color}
\usepackage{cancel}
\usepackage{appendix}
\usepackage{tikz}
\usepackage{soul}
\usepackage{dsfont}
\usepackage[author={Lorenzo}]{pdfcomment}
\usepackage{caption}
\usepackage{caption}
\usepackage{subcaption}
\usepackage[normalem]{ulem}
\theoremstyle{plain}
\newtheorem{theorem}{Theorem}[section]
\newtheorem{corollary}[theorem]{Corollary}
\newtheorem{lemma}[theorem]{Lemma}
\newtheorem{proposition}[theorem]{Proposition}
\theoremstyle{remark}
\newtheorem{remark}{Remark}

\newcommand{\eremk}{\hbox{}\hfill\rule{0.8ex}{0.8ex}}

\DeclareTextCompositeCommand{\u}{T1}{i}{\u\imath}

\newcommand{\Norm}[1]{{\left\|{#1} \right\|}}
\newcommand{\SemiNorm}[1]{{\left|{#1} \right|}}
\newcommand{\jump}[1]{\left[\!\left[#1\right]\!\right]}

\newcommand{\Normth}[1]{{\left\vert\kern-0.25ex\left\vert\kern-0.25ex\left\vert #1 
\right\vert\kern-0.25ex\right\vert\kern-0.25ex\right\vert}}
\newcommand{\Normf}[1]{{\left\vert\kern-0.25ex\left\vert\kern-0.25ex\left\vert\kern-0.25ex\left\vert #1 
\right\vert\kern-0.25ex\right\vert\kern-0.25ex\right\vert\kern-0.25ex\right\vert}}

\DeclareMathOperator{\nablabold}{\boldsymbol \nabla}
\DeclareMathOperator{\nablaboldh}{\nablabold_{\h}}
\DeclareMathOperator{\h}{h}
\newcommand{\ubf}{\mathbf{u}}
\DeclareMathOperator{\qh}{q_{\h}}
\newcommand{\vbf}{\mathbf{v}}
\newcommand{\vbfh}{\vbf_{\h}}
\newcommand{\wbfh}{\mathbf{w}_{\h}}

\newcommand{\E}{K}
\newcommand{\Tauh}{\mathcal T_{\h}}
\newcommand{\p}{p}
\newcommand{\pstar}{\p^*}
\newcommand{\psharp}{\p^\sharp}
\newcommand{\pprime}{\p'}
\DeclareMathOperator{\pop}{p}
\DeclareMathOperator{\qop}{q}

\newcommand{\boldalpha}{\boldsymbol{\alpha}}
\DeclareMathOperator{\W}{W}
\let\H\relax
\DeclareMathOperator{\H}{H}
\newcommand{\Rfrak}{\mathfrak R}
\newcommand{\BRfrak}{B_{\Rfrak}}
\newcommand{\hOmega}{\h_\Omega}
\newcommand{\CPS}{C_{\rm PS}}
\newcommand{\GammaN}{\Gamma_N}
\newcommand{\GammaD}{\Gamma_D}
\newcommand{\Gammatilde}{\widetilde\Gamma}

\newcommand{\CSob}{C_{\rm{Sob}}}
\newcommand{\CTR}{C_{\rm TR}}
\newcommand{\CT}{C_{\rm Trc}}
\newcommand{\CTRsharp}{\CTR^\sharp}

\newcommand{\hE}{\h_\E}
\newcommand{\Fcalh}{\mathcal F_{\h}}
\newcommand{\FcalE}{\mathcal F^\E}
\newcommand{\FcalEI}{\mathcal F^{\E I}}
\newcommand{\FcalED}{\mathcal F^{\E D}}
\newcommand{\FcalhI}{\mathcal F_{\h}^I}
\newcommand{\FcalhB}{\mathcal F_{\h}^B}
\newcommand{\FcalhID}{\mathcal F_{\h}^{ID}}
\newcommand{\FcalhIN}{\mathcal F_{\h}^{IN}}
\newcommand{\F}{F}
\newcommand{\hF}{\h_\F}

\newcommand{\nbf}{\mathbf n}
\newcommand{\nbfOmega}{\nbf_\Omega}
\newcommand{\nbfE}{\nbf_\E}
\newcommand{\nbfF}{\mathbf n_\F}
\newcommand{\q}{q}
\newcommand{\qprime}{\q'}

\let\L\relax
\DeclareMathOperator{\L}{L}
\newcommand{\Lbf}{\mathbf L}
\newcommand{\Wbf}{\mathbf W}

\newcommand{\Rbb}{\mathbb R}
\newcommand{\Nbb}{\mathbf{\underline N}}

\newcommand{\Kbb}{\underline{{\mathbf K}}}

\newcommand{\CSP}{C_{\rm SP}}

\newcommand{\CST}{C_{\rm ST}}

\newcommand{\CSTsharp}{\CST^\sharp}

\newcommand{\CSPstar}{C_{\rm SP}^*}

\newcommand{\CtildeSPstar}{\widetilde C_{\rm SP}^*}
\newcommand{\CtildeSTsharp}{\widetilde C_{\rm ST}^\sharp}

\newcommand{\dive}{\nablabold\cdot}
\newcommand{\gN}{g_N}
\newcommand{\gD}{{g}_D}
\newcommand{\dxbf}{\,{\rm{d}}\mathbf{x}}
\newcommand{\ds}{\,{\rm{d}}s}
\newcommand{\dint}[1]{\,{\rm{d}}{#1}}

\newcommand{\EucNorm}[1]{{\left|{#1} \right|}}
\let\P\relax
\DeclareMathOperator{\P}{P}
\DeclareMathOperator{\Pbf}{\bf P}
\newcommand{\Pk}{\P_k}
\newcommand{\qk}{q_{k}}
\newcommand{\qkmoF}{q_{k-1}^\F}
\newcommand{\vD}{v_D}
\DeclareMathOperator{\CR}{CR}
\DeclareMathOperator{\ph}{p_{\h}}

\newcommand{\DGkmo}{\Vbfcal_{\h}^{k-1}}

\newcommand{\ubfh}{\ubf_{\h}}
\newcommand{\ubfhl}{\ubf_{\h,\ell}}
\newcommand{\alphapt}{\alpha+2}

\newcommand{\alphaprime}{{\alpha'}}

\newcommand{\Kbbmo}{\underline{\mathbf{K}}^{-1}}
\newcommand{\Acal}{\boldsymbol{\mathcal{A}}}
\newcommand{\wbf}{\mathbf{w}}
\newcommand{\lambdamin}{\lambda_{\text{min}}}
\newcommand{\ubfl}{\mathbf{u}_{\ell}}

\newcommand{\Ical}{\mathcal{I}}

\newcommand{\Jcal}{\mathcal{J}}
\newcommand{\zbf}{\mathbf{z}}
\newcommand{\ubftilde}{\widetilde{\ubf}}
\newcommand{\vbftilde}{\widetilde{\vbf}}
\newcommand{\wbftilde}{\widetilde{\wbf}}
\newcommand{\zerobf}{\mathbf{0}}
\newcommand{\ubfzero}{\mathbf{u}_0}
\newcommand{\ubfhzero}{\mathbf{u}_{\h,0}}
\newcommand{\cbfh}{\mathbf{c}_{\h}}

\newcommand{\ubfhn}{\ubf_{\h}^{\rm{n}}}
\newcommand{\ubfhntilde}{\widetilde{\ubf}_{\h}^{\rm{n}}}
\newcommand{\ubfhnmotilde}{\widetilde{\ubf}_{\h}^{\rm{n-1}}}
\newcommand{\ubfhzerotilde}{\widetilde{\ubf}_{\h}^{0}}
\newcommand{\ubfhnmo}{\ubf_{\h}^{\rm{n-1}}}

\newcommand{\phn}{\pop_{\h}^{\rm{n}}}
\newcommand{\phntilde}{\widetilde\pop_{\h}^{\rm{n}}}
\newcommand{\phnmo}{\pop_{\h}^{\rm{n-1}}}
\newcommand{\phnmotilde}{\widetilde\pop_{\h}^{\rm{n-1}}}
\newcommand{\phzero}{\pop_{\h}^0}
\newcommand{\phzerotilde}{\widetilde{\pop}_{\h}^0}
\newcommand{\PikmthreeE}{\Pi_{k-3}^{0,\Tauh}}
\newcommand{\Icalk}{\Ical_k}
\newcommand{\Ltwo}{\L^2}
\newcommand{\CP}{C_P}
\newcommand{\CinvK}{C_{\text{inv}}^\E}
\newcommand{\hK}{\h_\E}

\newcommand{\Vbfcal}{\boldsymbol{\mathcal{V}}}
\newcommand{\Qcal}{\mathcal{Q}}
\newcommand{\Hbfcal}{\boldsymbol{\mathcal{H}}}
\newcommand{\Bcal}{\mathcal{B}}
\newcommand{\GcalN}{\mathcal{G}_N}
\newcommand{\GcalD}{\mathcal{G}_D}
\newcommand{\Dbfcal}{\boldsymbol{\mathcal{D}}}
\newcommand{\Dbfcalhkmo}{\boldsymbol{\mathcal{D}}_{\h}^{k-1}}

\newcommand{\Dbfcalhzero}{\boldsymbol{\mathcal{D}}_{\h}^{0}}
\newcommand{\Dbfcalhperp}{\boldsymbol{\mathcal{D}}_{\h}^{k-1,\perp}}
\DeclareMathOperator{\Pbfh}{\mathbf P_{\h}}
\newcommand{\diveh}{\nablaboldh\cdot}
\newcommand{\LtwoOmegad}{\Lbf^2(\Omega)}

\newcommand{\LalphaOmegad}{\Lbf^{\alpha}(\Omega)}

\newcommand{\LafrakOmegad}{\Lbf^{\afrak}(\Omega)}
\newcommand{\LafrakprimeOmegad}{\Lbf^{\afrak'}(\Omega)}
\newcommand{\LalphaprimeOmegad}{\Lbf^{\alphaprime}(\Omega)}

\newcommand{\RfrakE}{\Rfrak_\E}
\newcommand{\Nfrak}{\mathfrak N}
\newcommand{\PizFcalhz}{\Pi^{0,\Fcalh}_0}
\newcommand{\PizTauhz}{\Pi^{0,\Tauh}_0}

\newcommand{\ybf}{\mathbf{y}}
\newcommand{\LcalFE}{\mathcal{L}^\E_\F}
\newcommand{\LcalFEF}{\mathcal{L}^{\E_\F}_\F}

\newcommand{\alphamt}{{\alpha-2}}
\newcommand{\alphamo}{{\alpha-1}}

\newcommand{\alphafracalphamt}{\frac{\alpha}{\alphamt}}
\newcommand{\Qcalg}{\Qcal_g}
\newcommand{\Qcalzero}{\Qcal_0}
\newcommand{\QcalgD}{\Qcal_{\gD}}
\newcommand{\Qcalhgk}{\Qcal_{h,g}^k}
\newcommand{\QcalhgDk}{\Qcal_{h,\gD}^k}
\newcommand{\Qcalhzerok}{\Qcal_{h,0}^k}
\newcommand{\Vbfcalhkmo}{\Vbfcal_{\h}^{k-1}}
\newcommand{\Lbu}{\underline{\mathbf{L}}}
\newcommand{\Cell}{C_{\ell}}
\newcommand{\boldbeta}{\boldsymbol{\beta}}

\DeclareMathOperator{\Span}{span}
\newcommand{\bkF}{b_k^\F}
\newcommand{\VcalE}{\mathcal V^\E}
\newcommand{\Vcalh}{\mathcal V_{\h}}
\newcommand{\lambdaEF}{\lambda_{\E,\F}}

\newcommand{\Dbfcalhkmoperp}{\Dbfcal_{\h}^{k-1,\perp}}
\newcommand{\nfrak}{\mathfrak{n}}

\newcommand{\PizFcalhkmo}{\Pi^{0,\Fcalh}_{k-1}}

\newcommand{\PizFcalhs}{\Pi^{0,\Fcalh}_s}
\newcommand{\PiboldzFcalhs}{\boldsymbol\Pi^{0,\Fcalh}_s}
\newcommand{\qs}{q_s}

\newcommand{\PizTauhs}{\Pi^{0,\Tauh}_{s}}
\newcommand{\PiboldizTauhs}{\boldsymbol\Pi^{0,\Tauh}_{s}}
\newcommand{\PiboldzTauhkmo}{\boldsymbol\Pi^{0,\Tauh}_{k-1}}
\DeclareMathOperator{\RT}{\mathbf{RT}}
\newcommand{\PiboldRTTauhkmo}{\boldsymbol\Pi^{\RT,\Tauh}_{k-1}}

\newcommand{\PizTauhkmo}{\Pi^{0,\Tauh}_{k-1}}
\newcommand{\omegaF}{\omega_\F}

\newcommand{\FcalhD}{\mathcal F_{\h}^{D}}
\newcommand{\FcalhN}{\mathcal F_{\h}^{N}}
\newcommand{\CLcal}{C_{\mathcal{L}}}

\newcommand{\Shat}{\widehat S}
\newcommand{\Shatk}{\Shat_k}

\newcommand{\SFj}{S_j^\F}

\newcommand{\Ctildeell}{\widetilde{C}_\ell}
\newcommand{\bfrakh}{\mathfrak{b}_{\h}}

\newcommand{\varphibold}{\boldsymbol{\varphi}}
\newcommand{\ubfzerotilde}{\widetilde{\ubf}_0}
\newcommand{\ubfltilde}{\widetilde{\ubf}_{\ell}}
\newcommand{\qbfk}{\mathbf{q}_k}

\newcommand{\qbfkmtE}{\mathbf{q}_{k-2}^\E}
\newcommand{\afrak}{\mathfrak{a}}
\newcommand{\afrakprime}{\mathfrak{a}'}

\newcommand{\tzero}{t_0}
\newcommand{\ubfhzeroup}{\ubfh^0}
\newcommand{\sfrak}{\mathfrak{s}}
\newcommand{\Ibb}{\mathbf{\underline I}}
\newcommand{\obbF}{\mathds{1}_\F}

\DeclareMathOperator{\CBL}{C_{BL}}
\newcommand{\nrm}{{\rm{n}}}
\newcommand{\Abbn}{\underline{{\mathbf{A}}}^{\nrm}}
\newcommand{\Abb}{\underline{{\mathbf{A}}}}

\newcommand{\Mbb}{\underline{{\mathbf{M}}}}
\renewcommand{\Bbb}{\underline{{\mathbf{B}}}}
\newcommand{\Nbbnmo}{\underline{{\mathbf{N}}}^{{\rm{n}}-1}}

\newcommand{\pbfhn}{\mathbf{p}_{\h}^{\rm{n}}}
\newcommand{\sbfhn}{\mathbf{s}_{\h}^{\rm{n}}}
\newcommand{\fbf}{\mathbf{f}}
\newcommand{\bbf}{\mathbf{b}}
\newcommand{\gbf}{\mathbf{g}}
\newcommand{\rbf}{\mathbf{r}}
\newcommand{\nmax}{n_{\max}}

\title{\scriptsize A nonconforming method for a generalized Darcy--Forchheimer
model}
\author{\scriptsize{Michele Botti\thanks{MOX, Department of Mathematics, Politecnico di Milano, 20133 Milano, Italy (michele.botti@polimi.it)},   
Lorenzo Mascotto\thanks{Department of Mathematics and Applications, University of Milano-Bicocca, 20125 Milan, Italy (lorenzo.mascotto@unimib.it); 
Faculty of Mathematics, University of Vienna, 1090 Vienna, Austria;
IMATI-CNR, 27100 Pavia, Italy},
Marialetizia Mosconi \thanks{Department of Mathematics and Applications, University of Milano-Bicocca, 20125 Milan, Italy (m.mosconi@campus.unimib.it)}}}
\date{}

\begin{document}
\maketitle

\begin{abstract}
\noindent We analyze a dual mixed nonconforming discretization
of a generalized Darcy--Forchheimer model.
Compared to the analogous scheme proposed in~\cite{Girault-Wheeler:2008},
we consider general, i.e., nonquadratic, Forchheimer nonlinearities;
we admit mixed, inhomogeneous boundary conditions;
we allow for more general, i.e.,
with lower Lebesgue regularity, permeability tensors;
we construct general-order schemes;
we prove convergence to the exact solution under
low regularity assumptions, based on novel
Sobolev-trace inequalities for broken spaces;
we derive error estimates of general-order
assuming extra regularity of the exact solution
and data;
we present numerical results assessing the performance
of the proposed schemes for different types
of nonlinearity and nonlinear solvers.



\end{abstract}

\section{Introduction}

\paragraph*{State-of-the-art.}

The Darcy--Forchheimer's law models the fluid flow
in porous media~\cite{Forchheimer:1901}.
It is a nonlinear extension of Darcy's law
and provides a more accurate description
of high-velocity flow behaviours.
Numerical discretizations of 
Darcy--Forchheimer equations pose extra theoretical and computational
challenges compared to their linear counterpart;
as such, they have received increasing attention
over the last three decades.
In~\cite{Douglas-PaesLeme-Giorgi:1993}, existence and uniqueness
are established for a semi-discrete formulation
of the standard (quadratic) Darcy--Forchheimer model.
Well-posedness of the fully discrete scheme and corresponding error estimates
are provided in~\cite{Park:2005}; see also \cite{Rui-Pan:2012};
more general nonlinearities are considered in \cite{Kim-Park:1999}.
In these works, the problem is formulated in the standard mixed setting,
and Raviart-Thomas/piecewise-polynomial pairs are employed.
A different approach is undertaken in~\cite{Girault-Wheeler:2008},
where a lowest order mixed dual formulation
based on piecewise-constant/lowest order Crouzeix-Raviart (CR in what follows)
pairs is introduced:
convergence to the exact solution
of discrete solution sequences
under mesh refinements is proved
and the analysis of a nonlinear solver is also provided.
Numerical results for this scheme
are presented in~\cite{Lopez-Molina-Salas:2008-2009};
see also~\cite{Salas-Lopez-Molina:2013}
for results concerning the standard mixed formulation.
We also mention~\cite{Wang-Rui:2015}, where, compared to~\cite{Girault-Wheeler:2008},
additional stabilization in the normal direction for the fluxes are considered,
following ideas from~\cite{Burman-Hansbo:2005} for the linear case.
More recent contributions include
\cite{Zhang-Rui:2017,
Rui-Pan:2017,
Almonacid-Diaz-Gatica-Marquez:2020,
Caucao-Gatica-Sandoval:2021,
Caucao-Gatica-Gatica:2023,
Triki-Sayah-Semaan:2024,
Sayah-Semaan-Triki:2021,
Zhao-Chung-Park-Zhou:2021,
Badia-Carstensen-Martin-RuizBaier-VillaFuentes:2025,
Amigo-Lepe-Otarola-Rivera:2025,
Antonietti-Bonetti-Botti:2025}.

\paragraph{Goals of the paper.}
The main objective of this work is to design a general-order
dual mixed formulation for a generalized Darcy--Forchheimer problem.
The starting point is the work~\cite{Girault-Wheeler:2008}
of Girault and Wheeler,
which we extend along several directions.
\begin{itemize}
\item We consider more general nonlinearities, replacing
the quadratic Darcy--Forchheimer nonlinearity  with an $(\alpha-2)$-type law,
which boils down to the classical one when $\alpha=3$;
see problem~\eqref{Darcy--Forchheimer-strong} below.
\item We admit inhomogeneous, mixed boundary conditions.
\item We allow for permeability tensors~$\Kbb$
with lower Lebesgue regularity; cf. \eqref{data-spaces} below.
\item We introduce a general-order scheme
employing piecewise-polynomial spaces of order~$p-1$ for the fluxes
and Crouzeix-Raviart spaces of order~$p$ for the potentials.
\item We prove convergence under low-regularity assumptions
on the solution and data,
based on novel Sobolev-trace inequalities in broken spaces,
with constants independent of the type of discretization spaces
(in particular, independent of the polynomial degree in our case).
\item We derive general-order error estimates
under additional regularity assumptions on the solution and data.
\item We present numerical experiments assessing the performance of the scheme for different values of the nonlinearity parameter~$\alpha$
and for different nonlinear solvers.
\end{itemize}

The use of high-order schemes for nonlinear problems
may appear counterintuitive, as solutions
can exhibit singular behavior even for regular data and smooth domains;
our motivation is the prospective development of $hp$-adaptive strategies,
which are particularly well suited for
the efficient discretization of such problems.
As in~\cite{Girault-Wheeler:2008}, we adopt a dual mixed formulation
and discretize the potential with nonconforming
(Crouzeix-Raviart) elements.
Using nonconforming methods in the discretization of more complex
nonlinear problems is important, as standard conforming discretization
might not converge to the exact solution; cf.~\cite{Balci-Ortner-Storn}.

\paragraph{Outline of the remainder of the paper.}
While in Section~\ref{subsection:functional-setting}
we set some standard spaces, their properties,
and assumptions on the problem domain,
Section~\ref{section:continuous-setting} is devoted to introduce
the general nonlinear Darcy--Forchheimer model problem
and prove its well-posedness.
In Section~\ref{section:FE-spaces}, we introduce admissible sequences of meshes, broken Sobolev spaces,
and finite element (piecewise-polynomial
and Crouzeix-Raviart) spaces;
we also discuss known and novel Sobolev-Poincar\'e
and Sobolev-trace inequalities
for functions in piecewise Sobolev spaces.
The method is introduced in Section~\ref{section:standard-method};
its convergence for low-regularity solution and data
is detailed in Section~\ref{section:convergence},
while error estimates under extra regularity assumptions
are derived in Section~\ref{section:error-estimates}.
Numerical experiments are given in Section~\ref{section:numerical-experiments},
while some conclusions are drawn in Section~\ref{section:conclusions}.
Appendix~\ref{appendix:properties-Acal}
is concerned with the proof
of technical results needed for the well-posedness
of the continuous problem.

\subsection{Functional setting} \label{subsection:functional-setting}
We consider a domain~$\Omega$ in $\Rbb^d$, $d=2,3$,
with boundary $\Gamma$, which we split into
$\Gamma=\Gamma_{\rm D}\cup\Gamma_{\rm N}$.
We assume that either
\begin{equation} \label{Omega-ass-1}
\text{$\Omega$ is star-shaped with respect
to a ball $\BRfrak$ of radius~$\Rfrak$}
\end{equation}
or
\begin{equation} \label{Omega-ass-2}
\text{$\Omega$ admits a shape-regular decomposition into $\Nfrak$ simplices.}
\end{equation}
These two options are relevant to establish some inequalities
in Section~\ref{subsection:semi-norms-technical} below,
with constants that are explicit
with respect to the geometric parameters
in~\eqref{Omega-ass-1} and~\eqref{Omega-ass-2}.
The simplicial subdivision in~\eqref{Omega-ass-2}
is not related to the discretization meshes
in Section~\ref{subsection:meshes-broken-finite} below.

The diameter of~$\Omega$ is~$\hOmega$;
$\nbfOmega$ is the outward unit normal vector to~$\Gamma$.
We denote scalars in standard font;
vector fields in boldface font;
tensors in underlined boldface font.
Greek letters in boldface font
denote multi-indices in $\mathbb{N}^d$;
the lengths are $|\boldbeta| :=\sum_{i=1}^d\boldbeta_i$
for $\boldbeta$ in $\mathbb N^d$.
For a generic subset~$X$ of~$\Omega$ with diameter~$\h_X$,
and for all~$\p$ in~$[1,\infty)$
and~$k$ in~$\mathbb N$,
we consider the Lebesgue and Sobolev spaces~$\L^\p(X)$ and~$\W^{k,\p}(X)$,
which we endow with norm, and seminorm and norm, respectively,
\[
\Norm{v}_{\L^\p(X)}^\p
:= \int_X \vert v \vert^\p ,
\]
and
\[
\SemiNorm{v}_{\W^{k,\p}(X)}
:= \big(\sum_{\vert \boldalpha \vert=k} 
        \Norm{D^{\boldalpha} v}_{\L^\p(X)}\big)^{\frac1\p},
\qquad\qquad
\Norm{v}_{\W^{k,\p}(X)}
:= \Big( \sum_{\ell=0}^k \big(\h_X^{\ell-k}
            \SemiNorm{v}_{\W^{\ell,\p}(X)}\big)^\p \Big)^\frac1\p.
\]
Sobolev spaces of noninteger order are constructed
by interpolation.
The subspace of functions in~$\L^\p(X)$
with zero average over $X$ is~$\L^\p_0(X)$.
Sobolev norms and seminorms applied to vector fields
are denoted by
$\Norm{\cdot}_{\Wbf^{k,\p}(X)}$
and~$\SemiNorm{\cdot}_{\Wbf^{k,\p}(X)}$;
Lebesgue norms applied to tensors are denoted by
$\Norm{\cdot}_{\Lbu^{\p}(X)}$.
We recall the following Sobolev embeddings \cite[Sect.~2.3]{Ern-Guermond:2021}:
\begin{itemize}
\item if $k \p <d$,
$\W^{k,\p}(X) \hookrightarrow \L^{q}(X)$
for all $q$ in $[\p,\frac{\p d}{d-k\p}]$;
\item if $k \p = d$,
$\W^{k,\p}(X) \hookrightarrow \L^{q}(X)$
for all $q$ in $[\p,\infty)$.
\end{itemize}
We spell out the generic Sobolev embedding bound:
for $k$, $p$, and~$q$ as above,
there exists a positive constant~$\CSob$
depending on $q$, $k$, $\p$, and $X$, such that
\begin{equation} \label{Sobolev-embedding}
\Norm{v}_{\L^{q}(X)}
\le \CSob \h_X^{\frac{d}{q}- \frac{d}{\p} + k} \Norm{v}_{\W^{k,\p}(X)}
\qquad\qquad\qquad
\forall v \in \W^{k,\p}(X).
\end{equation}
Given an index~$\p$ in $[1,\infty)$, we define
(with $1\le\p\le\psharp\le\pstar$)
\begin{equation} \label{indices}
\begin{split}
& \pprime :=  \frac{\p}{\p-1};
\qquad\qquad
\psharp := 
\begin{cases}
\frac{\p(d-1)}{d-\p}    & \text{if } \p < d \\
\infty                  & \text{otherwise};
\end{cases}
\qquad\qquad
\pstar := 
\begin{cases}
\frac{\p d}{d-\p}   & \text{if } \p < d  \\
\infty              & \text{otherwise}.
\end{cases}
\end{split}
\end{equation}
We have the following trace theorem \cite[Ch. 3.2]{Ern-Guermond:2021}:
given $\widetilde\Gamma$ a subset of the boundary of~$X$,
there exists a positive constant~$\CT$
such that
\begin{equation} \label{global-trace}
\Norm{v}_{\L^{\psharp}(\widetilde\Gamma)}
\le \Norm{v}_{\L^{\psharp}(\partial X)}
\overset{\eqref{Sobolev-embedding}}{\le}
\CSob \Norm{v}_{\W^{\frac1{\pprime},\p}(\partial X)}
\le \CT \Norm{v}_{\W^{1,\p}(X)}
\qquad\qquad\forall v\in\W^{1,\p}(X).
\end{equation}
The trace operator is surjective
from~$\W^{1,\p}(X)$ to~$\W^{\frac1{\pprime},\p}(\Gamma)$.
We consider the subspace of functions in~$\W^{1,\p}(X)$
with trace~$g$
in $\W^{\frac{1}{\pprime},\p}(\widetilde\Gamma)$
\begin{equation*}
\W^{1,\p}_{\Gammatilde,g}(X)
:= \{ v \in \W^{1,\p}(X) \mid v_{|\widetilde\Gamma} = g \}.
\end{equation*}
A Poincar\'e-Steklov inequality holds true \cite[Ch.~3.3]{Ern-Guermond:2021}:
there exists a positive constant $\CPS$
depending only on~$\p$ and~$X$ such that
\begin{equation} \label{Poincare-Steklov}
\Norm{v}_{\L^\p(X)}
\le \CPS \h_X \SemiNorm{v}_{\W^{1,\p}(X)}
\qquad\qquad
\forall v \in
\begin{cases}
\W^{1,\p}_{\Gammatilde}(X)  & \Gammatilde \subseteq \partial X \\
\W^{1,\p}(X)\cap\L^\p_0(\Omega).
\end{cases}
\end{equation}
Inequality~\eqref{Poincare-Steklov} holds true
also for functions with zero average
over a subset of~$X$ and/or $\partial X$
with nonzero measure.
We denote the Euclidean vector norm
by $\SemiNorm{\vbf} = \sqrt{\vbf^T\cdot\vbf}$.
For generic Banach spaces $X$ and $Y$, $Y \subset X$,
we consider the quotient space $X / Y$ 
endowed with quotient norm
\begin{equation*}
    \Norm{x}_{X/Y} := \inf_{y \in Y} \Norm{x-y}_X.
\end{equation*}
Moreover, $X^*$ denotes the dual space of~$X$,
which we endow with the norm
$\Norm{z}_{X^*} = \sup_{x\in X} \frac{\langle z, x \rangle}{\Norm{x}_X}$,
where $\langle z, x \rangle$ denotes the action of the functional~$z$
on the elements~$x$ in~$X$.

\section{The model problem} \label{section:continuous-setting}
Given $\alpha > 2$,
$b : \Omega \to \Rbb$,
$\gN: \GammaN \to \Rbb$,
and $\gD: \GammaD \to \Rbb$,
we consider the generalized Darcy--Forchheimer problem:
find a flux $\ubf : \Omega \to \Rbb^d$
and a potential $\pop: \Omega \to \Rbb$
such that
\begin{equation} \label{Darcy--Forchheimer-strong}
\begin{cases}
\nablabold \pop + \frac{\mu}{\rho}\Kbbmo \ubf
            + \frac{\beta}{\rho} \SemiNorm{\ubf}^{\alphamt} \ubf 
            = 0                                                 &\text{ in } \Omega \\
\dive \ubf = b                                                  &\text{ in } \Omega \\
\ubf\cdot \nbf = \gN                                            &\text{ on } \GammaN \\
\pop  = \gD                                                      &\text{ on } \GammaD.
\end{cases}
\end{equation}
If $\GammaN = \Gamma$,
then~$b$ and~$\gN$ must satisfy
the compatibility conditions
\begin{equation} \label{compatibility-condition}
    \int_{\Omega}b\dxbf = \int_{\Gamma}\gN \ds.
\end{equation}
The permeability tensor $\Kbb^{-1}$ is positive definite;
let $\lambdamin > 0$ be its smallest eigenvalue
over~$\Omega$. The density~$\rho$,
the viscosity~$\mu$,
and the Darcy--Forchheimer coefficient~$\beta$
are positive constants.
With the notation as in~\eqref{indices},
we assume that
\begin{equation} \label{data-spaces}
\begin{split}
&\Kbbmo \in \Lbu^{\frac{\alpha}{\alphamt}}(\Omega),
\qquad\qquad\qquad
b \in \Bcal := \L^{((\alphaprime)^*)'}(\Omega)
    = \L^{\frac{\alpha d}{d+\alpha}}(\Omega),\\
&\gN \in \GcalN := \L^{((\alphaprime)^\sharp)'}(\GammaN) 
    =  \L^{\alpha\frac{d-1}{d}}(\GammaN),
    \qquad\qquad\qquad
\gD \in \GcalD := \W^{\frac{1}{\alpha},\alpha'}(\GammaD),
\end{split}
\end{equation}
and, for a generic~$g$ in~$\GcalD$,
we define the spaces
\begin{equation*}
    \Vbfcal :=  \LalphaOmegad,
    \qquad\qquad
    \Qcalg :=
        \begin{cases}
            \W^{1,\alphaprime}(\Omega) \cap \L^2_0(\Omega) &\text{ if } \GammaN = \Gamma \\
            \W^{1,\alphaprime}_{\GammaD,g}(\Omega)            &\text{ if } \GammaN \neq \Gamma.
        \end{cases}
\end{equation*}
On the continuous level, we may also take
$\GcalN$ equal to $[\W^{\frac1\alpha , \alphaprime}(\GammaN)]^*$;
however, this choice would not be fine on the discrete level,
since we shall be using
nonconforming spaces.

We consider the following variational formulation:
find $(\ubf,\pop) \in \Vbfcal \times \QcalgD$ such that
\begin{subequations} \label{variational-formulation}
\begin{align}
    \frac{\mu}{\rho} \int_{\Omega}(\Kbbmo) \ubf \cdot \vbf \dxbf
        + \frac{\beta}{\rho} \int_{\Omega}\EucNorm{\ubf}^{\alphamt} \ubf \cdot \vbf \dxbf
        + \int_{\Omega}\nablabold \pop \cdot \vbf \dxbf
        = 0& \label{variational-formulation:first-eq}
        \qquad\qquad \forall\, \vbf \in \Vbfcal  \\
    \int_{\Omega} \ubf \cdot \nablabold \qop \dxbf
        = - \int_{\Omega} b\, \qop \dxbf
        + \int_{\GammaN} \gN \, \qop \ds&
        \qquad\qquad \forall\, \qop \in \Qcalzero. \label{continuous-constraint}
\end{align}
\end{subequations}
Introduce
\begin{equation*} 
    \Hbfcal := \{ \vbf \in \Vbfcal \, | \, \dive \vbf \in \L^{((\alphaprime)^*)'}(\Omega)\}
\end{equation*}
and
\begin{equation} \label{Dbfcal}
\Dbfcal := \{ \vbf \in \Hbfcal \, | \,
                \vbf \cdot \nbf_{|\GammaN} = 0,
                \, \dive \vbf = 0 \text{ in } \Omega \}.
\end{equation}
An integration by parts reveals that $\vbf$ in $\Vbfcal$
belongs to $\Dbfcal$ if and only if
$\vbf$ satisfies
\begin{equation} \label{Dbfcal:orth-property}
    \int_{\Omega}\nablabold \qop \cdot \vbf \dxbf = 0
    \qquad\qquad \forall \qop \in \Qcalzero.
\end{equation}

\begin{lemma}
If~$\GammaN\ne\Gamma$,
for any~$\afrak$ in $(1,\infty)$
and $g$ in $\GcalD$,
we have the inf-sup condition
\begin{equation} \label{inf-sup}
    \inf_{\qop\in \W^{1,\afrakprime}_{\GammaD,g}(\Omega)} \sup_{\vbf \in \Lbf^{\afrak}(\Omega)}
        \frac{\int_{\Omega} \vbf \cdot \nablabold \qop \dxbf}
            {\Norm{\nablabold \qop}_{\LafrakprimeOmegad} \Norm{\vbf}_{\LafrakOmegad}}
            = 1.
\end{equation}
If~$\GammaN=\Gamma$, then
the inf-sup condition holds true by replacing
$\W^{1,\afrakprime}_{\GammaD,g}(\Omega)$
with $\W^{1,\afrakprime}(\Omega)_{}\cap\L^2_0(\Omega)$.
\end{lemma}
\begin{proof}
A duality argument reveals
\begin{equation*}
    1
        = \frac{\Norm{\nablabold \qop}_{\LafrakprimeOmegad}}{\Norm{\nablabold \qop}_{\LafrakprimeOmegad}}
        =
        \sup_{\vbf \in \Lbf^{\afrak}(\Omega)}
        \frac{\int_\Omega \vbf \cdot \nablabold \qop \dxbf}{\Norm{\nablabold \qop}_{\LafrakprimeOmegad}\Norm{\vbf}_{\LafrakOmegad}}
        \qquad\qquad \forall \qop \in \W^{1,\afrakprime}_{\GammaD,g}(\Omega).
\end{equation*}
Taking the infimum over all~$\qop$ on both sides
yields \eqref{inf-sup}.
\end{proof}

Next,
we show that
there exists solution to \eqref{continuous-constraint}
for given $b$ and $\gN$.
\begin{proposition} \label{prop:existence-ubfl}
Given~$b$ and~$\gN$ as in~\eqref{data-spaces},
there exist~$\ubfl$ in $\Vbfcal / \Dbfcal$
satisfying
\begin{equation} \label{ubfl-problem}
 \int_{\Omega} \ubfl \cdot \nablabold \qop \dxbf
        = - \int_{\Omega} b\,\qop\dxbf
        + \int_{\GammaN} \gN \, \qop \ds
        \qquad\qquad \forall\, \qop \in \Qcalzero
\end{equation}
and a positive constant $\Cell$
only depending on~$\alpha$ and~$\Omega$ such that
\begin{equation} \label{continuity-data}
\Norm{\ubfl}_{\LalphaOmegad}
\leq \Cell \big(\Norm{b}_{\L^{((\alphaprime)^*)'}(\Omega)} 
    + \Norm{\gN}_{\L^{((\alphaprime)^\sharp)'}(\GammaN)} \big).
\end{equation}
\end{proposition}
\begin{proof}
For $\alpha = 3$ and pure Neumann
boundary conditions,
the proof
can be found in \cite[Prop.~1]{Girault-Wheeler:2008}.
Here,
we consider general $\alpha > 2$
and mixed boundary conditions.
We consider two cases.

\paragraph*{Case 1: mixed/Dirichlet boundary conditions.}
Triangle's inequality,
Cauchy-Schwarz' inequality,
the continuous Sobolev embedding theorem \eqref{Sobolev-embedding},
and the continuous trace theorem \eqref{global-trace}
yield,
for a hidden constant only depending on~$\alpha$ and~$\Omega$,
\begin{equation*}
\begin{split}
   \left| -\int_{\Omega}b\,\qop\dxbf + \int_{\GammaN} g\,\qop\ds \right|
     \lesssim \left( \Norm{b}_{\L^{((\alphaprime)^*)'}(\Omega)} + \Norm{\gN}_{\L^{((\alphaprime)^\sharp)'}(\GammaN)}\right)\SemiNorm{\qop}_{\W^{1,\alphaprime}(\Omega)}
    \qquad \forall \qop \in \Qcalzero,
\end{split}
\end{equation*}
i.e., the mapping $\qop \to -\int_{\Omega}b\,\qop\dxbf + \int_{\GammaN} g\,\qop\ds$
belongs to $(\Qcalzero)^*$.
The assertion follows from the inf-sup condition \eqref{inf-sup}.

\paragraph*{Case 2: pure Neumann boundary conditions.}
The compatibility condition~\eqref{compatibility-condition},
triangle's inequality,
Cauchy-Schwarz' inequality,
the continuous Sobolev embedding theorem \eqref{Sobolev-embedding},
the continuous trace theorem \eqref{global-trace},
and Poincar\'e-Steklov's inequality yield,
for a hidden constant only depending on~$\alpha$ and~$\Omega$,
\small
\begin{equation*}
\begin{split}
& \left| -\int_{\Omega}b\,\qop\dxbf + \int_{\GammaN} g\,\qop\ds \right|
\leq \inf_{c\in\Rbb} \Big(\left|\int_{\Omega}b\,(\qop-c)\dxbf \right|
    + \left|\int_{\GammaN} g\,(\qop-c)\ds \right| \Big) \\
& \leq \left( \Norm{b}_{\L^{((\alphaprime)^*)'}(\Omega)} 
    + \Norm{\gN}_{\L^{((\alphaprime)^\sharp)'}(\GammaN)}\right)
    \inf_{c\in\Rbb}
    \left( \Norm{\qop-c}_{\L^{(\alphaprime)^*}(\Omega)}
    + \Norm{\qop-c}_{\L^{(\alphaprime)^\sharp}(\GammaN)} \right) \\
& \lesssim \left( \Norm{b}_{\L^{((\alphaprime)^*)'}(\Omega)} + \Norm{\gN}_{\L^{((\alphaprime)^\sharp)'}(\GammaN)}\right)\SemiNorm{\qop}_{\W^{1,\alphaprime}(\Omega)}.
\end{split}
\end{equation*}
\normalsize
The mapping $\qop \to -\int_{\Omega}b\,\qop\dxbf + \int_{\GammaN} g\,\qop\ds$
belongs to $(\Qcalzero)^*$.
The assertion follows from the inf-sup condition \eqref{inf-sup}.
\end{proof}

We consider the reduced problem on the kernel $\Dbfcal$:
given $\ubfl$ as in \eqref{ubfl-problem},
find $\ubfzero$ in $\Dbfcal$ such that
\begin{equation} \label{splitted-variational-formulation}
    \frac{\mu}{\rho}\int_{\Omega}(\Kbbmo)(\ubfzero + \ubfl)\cdot \vbf \dxbf
        + \frac{\beta}{\rho}\int_{\Omega}|\ubfzero + \ubfl|^{\alphamt}((\ubfzero + \ubfl)\cdot \vbf) \dxbf
        = 0
        \qquad\qquad \forall \vbf \in \Dbfcal.
\end{equation}
We denote by
$\Acal(\cdot) : \LalphaOmegad \to \Lbf^{\alphaprime}(\Omega)$
the mapping
\begin{equation} \label{Acal}
    \Acal(\vbf) := \frac{\mu}{\rho}\Kbbmo\vbf + \frac{\beta}{\rho}\EucNorm{\vbf}^\alphamt\vbf.
\end{equation}
Since
\begin{equation*} 
\begin{split}
    \Norm{\EucNorm{\vbf}^{\alphamt}\vbf}_{\Lbf^{\alphaprime}(\Omega)}
        &= \Big( \int_{\Omega} |\EucNorm{\vbf}^{\alpha-2}\vbf|^{\alphaprime}\dxbf\Big)^{\frac{1}{\alphaprime}} \\
        &=  \Big( \int_{\Omega} |\vbf|^{(\alphamo)\alphaprime}\dxbf\Big)^{\frac{1}{\alphaprime}}
        =  \Big( \int_{\Omega} |\vbf|^{\alpha}\dxbf\Big)^{\frac{\alphamo}{\alpha}}
        =  \Norm{\vbf}_{\LalphaOmegad}^{\alphamo},
\end{split}
\end{equation*}
we have
\begin{equation} \label{bound-Acal-Lalphaptprime}
    \Norm{\Acal(\vbf)}_{\Lbf^{\alphaprime}(\Omega)}
        \leq \left[ \frac{\mu}{\rho} \Norm{\Kbbmo}_{\Lbu^{\frac{\alpha}{\alphamt}}(\Omega)}
             + \frac{\beta}{\rho}\right]
                \Norm{\vbf}_{\LalphaOmegad}^{\alphamo}
\end{equation}
The following result extends
\cite[Prop.~2]{Girault-Wheeler:2008}
from $\alpha=3$ to $\alpha>2$
and mixed boundary conditions.
The proof is omitted for brevity
and is essentially based on the inf-sup condition~\eqref{inf-sup}.
\begin{proposition} \label{prop:equivalence-formulations}
The solution~$\ubf = \ubfzero + \ubfl$
to~\eqref{splitted-variational-formulation}
is also solution to~\eqref{variational-formulation}
and vice-versa. Moreover,
given~$\ubf$ solution to~\eqref{splitted-variational-formulation},
there exists a unique~$\pop$ in~$\QcalgD$
solution to~\eqref{continuous-constraint}, i.e.,
\begin{equation} \label{nablap-Acal}
    \Acal(\ubf) = \nablabold \pop.
\end{equation}
\end{proposition}

We now show an algebraic inequality
involving $\Acal(\cdot)$. 
\begin{lemma}
For all $\vbf$ and $\wbf$ in $\Vbfcal$,
the following inequality holds true:
there exists a positive constant $\CBL$
depending on $\alpha$ such that
\begin{equation} \label{estimate-Acal}
|\Acal(\vbf)-\Acal(\wbf)|
    \leq \frac{\mu}{\rho}| \Kbbmo| \ |\vbf-\wbf|
         + \CBL \frac{\beta}{\rho}
         |\vbf-\wbf|(|\vbf|^{\alphamt}+|\wbf|^{\alphamt})
\qquad \forall \alpha > 2.
\end{equation}
\end{lemma}
\begin{proof}
Triangle's inequality yields
\[
|\Acal(\vbf)-\Acal(\wbf)|
\leq \frac{\mu}{\rho}| \Kbbmo| \ |\vbf-\wbf|
    + \frac{\beta}{\rho}\EucNorm{|\vbf|^\alphamt \vbf - |\wbf|^\alphamt \wbf}.
\]
Using \cite[Lemma~2.1]{Barrett-Liu:1993},
there exists a positive constant $\CBL$ such that
\[
\left| |\vbf|^\alphamt \vbf - |\wbf|^\alphamt \wbf\right|
\le \CBL \left| \vbf-\wbf \right|
        \left( \vert\vbf\vert^{\alpha-2} + \vert \wbf \vert^{\alpha-2} \right).
\]
The assertion follows combining the two bounds above.
\end{proof}
Based on \cite[Sect.~2]{Girault-Wheeler:2008},
in order to show the well-posedness
of~\eqref{splitted-variational-formulation},
it suffices to prove the following three properties
of~$\Acal(\cdot)$:
\begin{subequations}
\begin{enumerate}
    \item monotonicity, i.e.,
    recalling that $\lambdamin$ is the smallest eigenvalue of $\Kbb^{-1}$,
    we have
    \begin{equation} \label{monotonicity}
    \frac{\mu}{\rho}\lambdamin\Norm{\wbf-\vbf}_{\LtwoOmegad}^2
        \leq \int_{\Omega}(\Acal(\wbf+\ybf)-\Acal(\vbf+\ybf))\cdot(\wbf-\vbf)\dxbf
        \qquad
        \forall\ \ybf,\vbf,\wbf \in \Vbfcal;
    \end{equation}
    \item coercivity, i.e.,
    we have
    \begin{equation} \label{coercivity}
    \lim_{\Norm{\vbf}_{\LalphaOmegad}\to +\infty}
        \Big( \frac{1}{\Norm{\vbf}_{\LalphaOmegad}} \int_{\Omega} \Acal(\vbf + \ybf) \cdot \vbf \Big)
        = + \infty
    \qquad \forall\ybf\in\Vbfcal;
    \end{equation}
    \item hemi-continuity, i.e.,
    given $\ybf$, $\vbf$, and $\wbf$ in $\Vbfcal$,
    the mapping
    \begin{equation} \label{hemi-continuity}
	       t \to \int_\Omega \Acal(\ybf + \vbf + t\,\wbf) \cdot \wbf \dxbf
    \end{equation}
    is continuous from $\Rbb$ into $\Rbb$.  
\end{enumerate}
The proofs of these properties are postponed to Appendix \ref{appendix:properties-Acal}.
\end{subequations}
We are in a position to show the well-posedness of \eqref{splitted-variational-formulation}
and hence,
due to Proposition~\ref{prop:equivalence-formulations},
of~\eqref{variational-formulation}.
\begin{proposition}  \label{proposition:well-posedness}
Given $b$ and $\gN$
as in \eqref{data-spaces},
problem \eqref{variational-formulation}
admits a unique solution $(\ubf,\pop)$ in $\Vbfcal \times \QcalgD$.
Moreover, given~$\Cell$ as in~\eqref{continuity-data},
we have
\begin{subequations}
\small\begin{equation} \label{bound-u-Lalphapt}
\begin{split}
\Norm{\ubf}_{\LalphaOmegad}
&\leq \Big( \frac{2^{\alphaprime}(\alpha-1)\mu^{\alphaprime}}{\beta^{\alphaprime}}
    \Norm{\Kbb^{-1}}_{\Lbu^{\frac{\alpha}{\alpha-2}}}
    ^{\alphaprime-1} +2 \Big)
    ^{\frac1\alpha}
    \Cell^{\alphaprime-1}
    \big( \Norm{b}_{\L^{((\alphaprime)^*)'}(\Omega)}
    + \Norm{\gN}_{\L^{((\alphaprime)^\sharp)'}(\GammaN)}
    \big)^{\alphaprime-1}
\end{split}
\end{equation}\normalsize
and
\small\begin{equation} \label{bound-p-Lalphaptprime}
\begin{split}
\Norm{\nablabold \pop}_{\LalphaprimeOmegad}
&   \leq \Big(\frac{\mu}{\rho}
    \Norm{\Kbbmo}_{\Lbu^\frac{\alpha}{\alphamt}(\Omega)}
    + \frac{\beta}{\rho} \Big)
    \Big( \frac{2^{\alphaprime}(\alpha-1)\mu^{\alphaprime}}{\beta^{\alphaprime}}
    \Norm{\Kbb^{-1}}_{\Lbu^{\frac{\alpha}{\alpha-2}}}
    ^{\alphaprime-1} +2 \Big)
    ^{\frac1{\alphaprime}} \cdot \\
&   \qquad\qquad
    \cdot \Cell
    \big(\Norm{b}_{\L^{((\alphaprime)^*)'}(\Omega)}
    + \Norm{\gN}_{\L^{((\alphaprime)^\sharp)'}(\GammaN)}\big).
\end{split}
\end{equation}\normalsize
\end{subequations}
\end{proposition}
\begin{proof}
Based on \cite[Chapt. I]{Showalter:1997},
properties~\eqref{monotonicity},
\eqref{coercivity}, and~\eqref{hemi-continuity}
yield the existence and uniqueness
of a solution~$\ubf$
to~\eqref{splitted-variational-formulation}.
Proposition~\ref{prop:equivalence-formulations}
yields the existence and uniqueness
of~$\pop$ solution to~\eqref{variational-formulation}.
As in Proposition \ref{prop:equivalence-formulations},
we split $\ubf$ as $\ubfzero + \ubfl$.
Since $\ubfzero$ belongs to $\Dbfcal$,
we have
\begin{equation*}
    0 \overset{\eqref{Dbfcal:orth-property}}{=} \int_{\Omega} \nablabold \pop \cdot \ubfzero
      \overset{\eqref{nablap-Acal}}{=}\int_{\Omega}  \Acal(\ubf)\cdot\ubfzero \dxbf.
\end{equation*}
Estimates~\eqref{Jcalprime-utilde}
with~$\vbftilde$ equal to~$\ubf$,
and~\eqref{Acalutilde-u}
with~$\vbftilde$ and~$\ybf$
equal to~$\ubf$ and~$\ubfl$ entail
\begin{equation*}
\begin{split}
0  = \int_{\Omega} \Acal(\ubf)\cdot\ubf \dxbf
       & - \int_{\Omega} \Acal(\ubf)\cdot\ubfl \dxbf 
    \geq \frac{\mu}{\rho}\lambdamin\Norm{\ubf}^2_{\LtwoOmegad}
        + \frac{\beta}{\rho}\Norm{\ubf}^\alpha_{\LalphaOmegad}\\
        &-\frac{\mu}{\rho}   \Norm{\Kbbmo}_{\Lbu^{\alphafracalphamt}(\Omega)}
                            \Norm{\ubf}_{\LalphaOmegad}
                            \Norm{\ubfl}_{\LalphaOmegad} 
        - \frac{\beta}{\rho}\Norm{\ubf}^{\alphamo}_{\LalphaOmegad}
                            \Norm{\ubfl}_{\LalphaOmegad}.
\end{split}
\end{equation*}
Applying Young's inequality twice
gives
\begin{equation*}
\begin{split}
\frac{\mu}{\rho}\lambdamin\Norm{\ubf}^2_{\LtwoOmegad}
    + \frac{\beta}{\rho}\Norm{\ubf}^\alpha_{\LalphaOmegad} 
&\leq \frac{\mu}{\rho}\Norm{\Kbbmo}_{\Lbu^{\frac{\alpha}{\alphamt}}(\Omega)}
            \Big( \frac{\varepsilon}{\alpha}\Norm{\ubf}_{\LalphaOmegad}^{\alpha}
        + \frac{1}{\varepsilon^{\alphaprime-1} \alphaprime}\Norm{\ubfl}_{\LalphaOmegad}^{\alphaprime} \Big) \\
&\quad+ \frac{\beta}{\rho}
            \left(\frac{1}{\alphaprime}\Norm{\ubf}^{\alpha}_{\LalphaOmegad}
                + \frac{1}{\alpha}\Norm{\ubfl}_{\LalphaOmegad}^{\alpha} \right),
\end{split}
\end{equation*}
whence
\begin{equation*}
\begin{split}
\Big(\frac{\beta}{\rho}-\frac{\beta}{\alphaprime\rho}
        -\frac{\varepsilon\mu}{\alpha \rho}
&\Norm{\Kbbmo}_{\Lbu^{\frac{\alpha}{\alphamt}}}\Big)\Norm{\ubf}^\alpha_{\LalphaOmegad}
    =\frac{1}{\alpha}\Big(\frac{\beta}{\rho}
        -\frac{\varepsilon\mu}{\rho}\Norm{\Kbbmo}_{\Lbu^{\frac{\alpha}{\alphamt}}}\Big)\Norm{\ubf}^\alpha_{\LalphaOmegad} \\
    &\leq -\frac{\mu}{\rho}\lambdamin\Norm{\ubf}^2_{\LtwoOmegad} 
         +  \frac{\mu}{\varepsilon^{\alphaprime-1}}\alphaprime\rho
         \Norm{\Kbbmo}_{\Lbu^{\frac{\alpha}{\alphamt}}(\Omega)}\Norm{\ubfl}_{\LalphaOmegad}^{\alphaprime} 
         + \frac{\beta}{\alpha\rho}
            \Norm{\ubfl}_{\LalphaOmegad}^{\alpha}.
\end{split}
\end{equation*}
Estimate \eqref{bound-u-Lalphapt}
follows from applying \eqref{continuity-data} and taking
\begin{equation*}
    \varepsilon = \frac{\beta}{2\mu}\Norm{\Kbbmo}^{-1}_{\Lbu^{\frac{\alpha}{\alphamt}}(\Omega)}.
\end{equation*}
Estimate \eqref{bound-p-Lalphaptprime}
follows from combining \eqref{nablap-Acal},
\eqref{bound-Acal-Lalphaptprime},
and \eqref{bound-u-Lalphapt}.

\end{proof}

\section{Finite element spaces and functional inequalities} \label{section:FE-spaces}
This section is devoted to introduce the basic setting
for a nonconforming discretization of problem~\eqref{variational-formulation}.
More precisely, sequences of regular meshes,
and corresponding broken Sobolev and finite element spaces
are introduced in Section~\ref{subsection:meshes-broken-finite};
important inequalities (Poincar\'e- and trace-type)
for broken Sobolev and Crouzeix-Raviart spaces
are exhibited in Section~\ref{subsection:semi-norms-technical}.
While the tools developed in Section~\ref{subsection:semi-norms-technical} below
remain valid in arbitrary dimension,
henceforth we fix $d=2$.
This choice allows us to characterize the Crouzeix-Raviart space in \eqref{CR}
in two equivalent ways; see Section \ref{subsection:meshes-broken-finite} below for more details.

\subsection{Meshes, and broken Sobolev and finite element spaces} \label{subsection:meshes-broken-finite}

\paragraph*{Meshes.}
We consider mesh sequences $\{\Tauh\}$,
where each $\Tauh$ is a finite collection of disjoint, closed,
simplicial elements such that
$\overline{\Omega}=\bigcup_{\E\in\Tauh}\E$.
For each $\E$ in $\Tauh$, $\partial \E$, $\hE$,
and~$\RfrakE$
denote the boundary of, the diameter of, and the radius
of the largest ball contained in~$\E$, respectively.
We assume that the sequence of meshes is shape-regular,
i.e., there exists a positive constant~$\gamma$
such that $\hE \le \gamma \RfrakE$
for all $\E$ in $\Tauh$, for all meshes~$\Tauh$.
The piecewise $\L^2$-projector~$\PizTauhs$
onto the space of polynomials of order~$s$,
$s$ in~$\mathbb N$, is given by
\begin{equation} \label{L2-projector-mesh}
\int_\E\PizTauhs v \ \qs \dxbf
:= \int_\E v \ \qs \dxbf
    \qquad\qquad
    \forall v \in \L^2(\Omega), \quad
    \forall \qs \in \P_{s}(\E),
    \quad \forall \E\in\Tauh;
\end{equation}
the vector version of this projection operator is~$\PiboldizTauhs$.

We associate each~$\Tauh$ with the set~$\Fcalh$
of facets.
With each facet~$\F$ in~$\Fcalh$,
we associate  its diameter~$\hF$
and a unit normal vector~$\nbfF$, and
\begin{itemize}
\item either there exist distinct $\E_{1,F}$ and $\E_{2,F}$ in $\Tauh$
such that $F = \partial \E_{1,F}\cap\partial \E_{2,F}$
and $\F$ is called an \emph{internal facet},
\item or there exists $\E_F$ in $\Tauh$ such that
$F\subseteq\partial \E_F\cap\Gamma$
and $\F$ is called a \emph{boundary facet}.
\end{itemize}
Interfaces and boundary facets are collected
in the subsets~$\FcalhI$ and~$\FcalhB$, respectively;
we assume that it is possible to split
the boundary facets
into Dirichlet~$\FcalhD$ and Neumann~$\FcalhN$
boundary facets, i.e., $\F$ is in~$\FcalhD$
if it belongs to~$\FcalhB$ and $\F$ is contained
in $\GammaD$ (similar for the Neumann case);
we also introduce $\FcalhID:=\FcalhI\cup\FcalhD$
and $\FcalhIN:=\FcalhI\cup\FcalhN$.
The set of facets of an element~$\E$ is~$\FcalE$;
we also define
$\FcalED:= \FcalE\cap\FcalhD$ and $\FcalEI:=\FcalE\cap \FcalhI$.
For~$\F$ in~$\Fcalh$, we define the facet patches
\begin{equation*} \label{facet-patch}
\omegaF :=  \bigcup \{ \E \in \Tauh \mid \F \in \FcalE \}.
\end{equation*}
The piecewise $\L^2$-projector $\PizFcalhs$ onto the space
of facet polynomials of order~$s$, $s$ in~$\mathbb N$, reads
\begin{equation*}
    \int_\F\PizFcalhs v \ \qs \ds := \int_\F v \ \qs \ds
    \qquad\qquad
    \forall v\in\W^{1,1}(\Omega),\quad
    \forall \qs \in \P_{s}(\F)
    \quad\forall\F\in\Fcalh;
\end{equation*}
the vector version of this projection operator is~$\PiboldzFcalhs$.

We associate each~$\Tauh$ with
the set of its vertices~$\Vcalh$;
the set of vertices of a given element~$\E$ is~$\VcalE$.
The vertices of~$\E$ are denoted by~$\nu_{i,K}$, $i=1,2,3$,
and the corresponding barycentric coordinates by $\lambda_{\E,i}$;
when convenient, we shall replace $\lambda_{\E,i}$
with~$\lambdaEF$ where~$\F$ is the facet
opposite to the $i$-th vertex
and we shall omit the subscript~$\E$
when no confusion occurs.

Consider now an element~$\E$ in~$\Tauh$.
A local continuous trace inequality holds true
\cite[Corollary~1.3]{Botti-Mascotto:2025-B}:
if~$\p$ is in~$[1,d)$, there exists a positive constant $\CTRsharp$
depending only on~$\p$ and~$\gamma$ such that
\begin{equation} \label{Sobolev-trace-sharp-elemental}
\Norm{v}_{\L^{\psharp}(\partial \E)} 
\le \CTRsharp \big( \Norm{v}_{\L^{\pstar}(\E)}
        + \Norm{\nablabold v}_{\Lbf^\p(\E)} \big).
\end{equation}

\paragraph*{Broken Sobolev spaces.}
For~$\p$ larger than or equal to~$1$, we define
\[
\W^{1,\p}(\Tauh):=
\left\{u\in \L^\p(\Omega)\,\mid\, {u}_{|\E}\in \W^{1,\p}(\E)
                        \qquad \forall \E\in\Tauh\right\}.
\]
For every~$v$ in~$\W^{1,\p}(\Tauh)$ and $\F$ in $\Fcalh$,
the jump operator on~$\F$ is well-defined
thanks to the trace theorem
and is given by
$$
\jump{v}_F:=
\begin{cases}
v_{|\E_{1,\F}} \nbf_{\E_{1,\F}} \cdot\nbfF
    + v_{|\E_{2,\F}} \nbf_{\E_{2,\F}} \cdot\nbfF 
    &\text{if } F \subset \FcalhI,\; \F = \partial \E_{1,\F}\cap \partial \E_{2,\F}\\
v_{|\F} \nbf_{\E_{\F}} \cdot\nbfOmega
    &\text{if } F\in\FcalhB,\; \F \subset \partial\E_\F \cap \Gamma.
\end{cases}
$$
We further define the normal jump operator on~$\F$
of a vector field~$\vbf$ in $\Wbf^{1,\p}(\Tauh)$ as
\begin{equation*}
\jump{\vbf}_{F,\nfrak}:=
\begin{cases}
\vbf_{|\E_{1,\F}} \cdot \nbf_{\E_{1,\F}}
    + \vbf_{|\E_{2,\F}} \cdot \nbf_{\E_{2,\F}} 
    &\text{if } F \subset \FcalhI,\; \F = \partial \E_{1,\F}\cap \partial \E_{2,\F}\\
\vbf_{|\F} \cdot \nbf_{\E_{\F}}
    &\text{if } F\in\FcalhB,\; \F \subset \partial\E_\F \cap \Gamma .
\end{cases}
\end{equation*}
We omit the subscript $\F$ whenever it is clear
from the context.
For $k$ in $\mathbb N$, we introduce
\begin{equation*}
\W^{1,\p}_k(\Tauh)
:= \big\{ v \in \W^{1,\p}(\Tauh) \, \big| \,
            ( \jump{v},\qkmoF )_{0,\F} = 0 \quad \forall \qkmoF \in \P_{k-1}(\F),
            \, \forall \F \in \FcalhI \big\},
\end{equation*}
and, for~$\vD$ in~$\L^1(\Gammatilde)$,
$\Gammatilde\subseteq\Gamma$
such that each $\F$ in $\FcalhB$
is either fully contained in $\Gammatilde$ or not,
\begin{equation*}
\begin{split}
\W^{1,\p}_{k,\vD} (\Tauh,\Gammatilde)
:= \big\{ v \in \W^{1,\p}_k(\Tauh) \, \big|
    & \, ( \jump{v} - \vD, \qkmoF )_{0,\F} = 0 \\
    & \quad \quad \forall \qkmoF \in \P_{k-1}(\F),
    \, \forall \F \in \FcalhB \text{ with }
        \F \subset\Gammatilde \big\}.
\end{split}
\end{equation*}
Vector valued and tensor valued broken Sobolev spaces
are denoted using the boldface and underlined-boldface
fonts, respectively.

\paragraph*{Finite element spaces.}
We denote the space of polynomials
of order smaller than or equal to a nonnegative integer~$k$
over an element~$\E$ and an edge~$\F$
by~$\P_k(\E)$ and~$\P_k(\F)$, respectively.
If $k$ is a negative integer,
we set $\P_k(\E)=\P_k(\F)=\emptyset$.
We introduce the Lagrangian finite element space of order~$k$:
\begin{equation*} \label{Xh}
\mathcal{L}_k(\Tauh)
:= \{ \qh \in \mathcal C^{0}(\overline\Omega) 
    \mid {\qh}_{|\E} \in \P_{k}(\E) \quad \forall \E \in \Tauh \}.
\end{equation*}
Given an element $\E$ in $\Tauh$,
basis functions of each local space
can be split into
\begin{subequations} \label{lagrangian}
\begin{align}
    &\text{\textbullet}\quad  \text{vertex basis functions
    denoted by } \{\varphi^{\nu_i}\}_{i=1}^{3}; \label{lagrangian:vertices}\\
    &\text{\textbullet}\quad \text{facet modal basis functions associated with each facet~$\F$ of $\FcalE$,
    denoted by } \{\varphi^\F_j\}_{j=1}^{k-1}; \label{lagrangian:facets}\\
    &\text{\textbullet}\quad \text{bulk modal basis functions
    denoted by }\{\varphi^\E_\ell\}_{\ell=1}^{(k-1)(k-2)/2} \label{lagrangian:bulk}.
\end{align}
\end{subequations}

We introduce the piecewise-polynomial space of order~$k$
\[
\Pk(\Tauh) 
:= \big\{ \qk \in \L^2(\Omega) \,\big \vert\,
    \qk{}_{|\E} \in \P_k (\E) \quad \forall K\in\Tauh \big\};
\]
vector valued piecewise-polynomial spaces are denoted by~$\Pbf_k(\Tauh)$.
Given $g$ in $\GcalD$,
we further define the Crouzeix-Raviart (CR) spaces
\cite{Crouzeix-Raviart:1973, Raviart-Thomas:1977}
\small\begin{subequations} \label{CR}
\begin{align}
\CR_k(\Tauh)
& := \{ \qh \in \Pk(\Tauh)  \,|\,
        ( \jump{\qh}, \qkmoF )_{0,\F} = 0 \quad \forall \qkmoF \in \P_{k-1}(\F),
            \, \forall \F \in \FcalhI \}, \label{implicit-CR}\\
\CR_{k,g}(\Tauh, \FcalhD)
& := \{ \qh \in \CR_k(\Tauh) \, \big| \,
    ( \jump{\qh} - g, \qkmoF )_{0,\F} = 0\;
    \forall \qkmoF \in \P_{k-1}(\F),  \,
    \forall \F \in \FcalhD \} \label{implicit-CR-vD}.
\end{align}
\end{subequations} \normalsize
We have the inclusions
\[
\CR_k(\Tauh) \subset   \W^{1,\p}_k(\Tauh),
\qquad\qquad\qquad
\CR_{k,g}(\Tauh, \FcalhD) \subset \W^{1,\p}_{k,g} (\Tauh,\GammaD).
\]
Next, we introduce a basis for the Crouzeix-Raviart spaces.
We denote the univariate Legendre polynomials of degree $k$ over~$\hat{I}:= [-1,1]$
by $\Shatk(\cdot)$.
With each facet~$\F$ in~$\Fcalh$,
we associate the facet nonconforming bubble function~\cite[Def. 3.2]{Carstensen-Sauter:2022}
(which is extended by zero outside~$\omegaF$)
\begin{equation}\label{nc-edge:functions}
\bkF :=
\begin{cases}
      \Shatk \left( 1-2 \lambdaEF\right)
      \qquad\qquad
      &\forall \E\in\Tauh,\; \E \subset \omegaF  \\
      0 
      &\text{otherwise}.
\end{cases}
\end{equation}
We first detail a set of unisolvent degrees of freedom
for the \emph{\bf odd degree} Crouzeix-Raviart spaces
$\CR_{k,g}(\Tauh, \FcalhD)$.
Given an element $\E$,
consider for all $\F$ in $\FcalE$ the mapped Legendre polynomials~$\SFj$ on $\F$,
$j$ non-negative integer,
and a basis
$\{m_{{\boldsymbol{\beta}}}^K\}_{\vert\boldsymbol{\beta}\vert=0}^{k-3}$ of $\P_{k-3}(\E)$
consisting of elements such that $\Norm{m_{{\boldsymbol{\beta}}}^K}_{\L^\infty(\Omega)} = 1$.
We introduce a set of linear functionals:
given $v$ in $\W^{1,1}(\Omega)$,
\begin{subequations}  \label{dof-CRk-odd}
\begin{align}
    &|\F|^{-1} \int_{F} v \, \SFj \ds
    & \forall j = 0, \dots, k-1, \quad\forall \F \in \FcalhIN; \label{dof-moment-edge}\\
    & |\E|^{-1} \int_{\E} v\, m_{\boldsymbol{\beta}}^K \dxbf
    &     \forall \vert \boldsymbol{\beta}\vert = 0, \dots, k-3,\quad \forall\E\in\Tauh \label{dof-moment-bulk}.
\end{align}
\end{subequations}
The functionals in \eqref{dof-CRk-odd}
are a set of unisolvent degrees of freedom (DoFs)
for $\CR_{k,g}(\Tauh, \FcalhD)$;
see, e.g., \cite[Lemma 2.1]{Ainsworth-Rankin:2008}.

We detail next a set of unisolvent degrees of freedom
for \emph{\bf even degree} Crouzeix-Raviart spaces.
With each element $\E$ in $\Tauh$,
we associate the bulk nonconforming bubble function \cite[Rem. 2]{Stoyan-Baran:2006}
(which is extended by zero outside~$\E$)
\begin{equation} \label{nc-elemental-bubble-function}
b^\E_k := 
\begin{cases}
	\frac12 ( -1 + \sum_{i=1}^3 \Shatk(1-2\lambda_i)) & \text{on } \E \\
	0								    & \text{otherwise},
\end{cases}
\end{equation}
and we consider the linear functionals
\begin{equation} \label{dof-CRk-even}
|\E|^{-1} \int_{\E} v \,b^\E_k\ds.
\end{equation}
The functionals in~\eqref{dof-CRk-odd}
and~\eqref{dof-CRk-even}
are a set of unisolvent DoFs
for $\CR_{k,g}(\Tauh, \FcalhD)$
if and only if, cf.~\cite[Lemma 2.2]{Ainsworth-Rankin:2008},
$v$ satisfies the compatibility condition
\[
\sum_{\F \in \partial \E }
    \int_F v \sum_{\substack{ 1\leq j \leq k-1 \\ j \text{ odd }}}
    \frac{{\SFj}}{\Norm{{\SFj}}^2_{\L^2(F)}} \ds = 0
    \qquad\qquad\qquad \forall \E \in \Tauh.
\]
The Crouzeix-Raviart spaces in~\eqref{CR}
are spanned by modal and nonconforming bubble functions~\cite{Baran-Stoyan:2007}.
This characterization depends on the order of the scheme:
\begin{itemize}
    \item \emph{odd order} Crouzeix-Raviart elements
    are spanned by the modal functions~\eqref{lagrangian:facets} and~\eqref{lagrangian:bulk}
    and the facet nonconforming bubbles~\eqref{nc-edge:functions}, i.e.,
    \begin{equation}\label{explicit-CR:odd}
    \CR_k(\Tauh) = \Span\left(                
                \bkF \text{ and }  \{\varphi_j^\F\}_{j=1}^{k-1}\quad
                \forall \F \subset \Fcalh ;\quad
                \{\varphi_{\ell}^\E\}_{\ell=1}^{(k-1)(k-2)/2} 
                \quad \forall \E \in \Tauh   \right);
    \end{equation}
    \item \emph{even degree} Crouzeix-Raviart elements
    are spanned by the modal functions~\eqref{lagrangian}
    and the bulk nonconforming bubbles, i.e.,
    \small
    \begin{equation} \label{explicit-CR:even}
    \CR_k(\Tauh) = \Span\left(
                \varphi^\nu \ \forall \nu \subset \Vcalh ;
                \quad
                \{\varphi_j^\F\}_{j=1}^{k-1}\
                \forall \F \subset \Fcalh;
                \quad
                b_k^\E \text{ and } \{\varphi_{\ell}^\E\}_{\ell=1}^{(k-2)(k-1)/2} 
                \ \forall \E \in \Tauh \right).
    \end{equation}
    \normalsize
\end{itemize}
The spaces in~\eqref{explicit-CR:odd}-\eqref{explicit-CR:even} and~\eqref{CR} coincide;
see, e.g., \cite[Lem. 1.2, 1.3]{Bressan-Mascotto-Mosconi:2025}
and \cite{Stoyan-Baran:2006}.
For even~$k$, the functions in~\eqref{explicit-CR:even} are linearly dependent,
cf. \cite{Fortin-Soulie:1983} for the second order case,
since the piecewise nonconforming bubble function, which
on each element is given by~$b_k^\E$, is globally continuous.
As such, in order to obtain a basis, it is necessary
to remove one bulk nonconforming bubble function.

\paragraph*{Seminorms in broken Sobolev spaces.}
We add a subscript~$\h$ to all differential operators
to denote corresponding operators defined piecewise
over~$\Tauh$; for instance, $\nablaboldh$
is the broken gradient over~$\Tauh$.

Given~$\p$ and~$\q$ larger than or equal to~$1$,
we endow, whenever it makes sense for the above choices of~$\p$ and~$\q$,
the space~$\W^{1,\p}(\Tauh)$ with the seminorms, for $d=2$,
\begin{subequations} \label{functionals-broken}
\begin{align}
\Norm{v}_{\W^{1,\p;\q}_{\GammaD}(\Tauh)}
&   := \Norm{\nablaboldh v}_{\Lbf^\p(\Omega)}
    + \big( \sum_{\F\in\FcalhID} 
    \hF^{-\frac{\p}{\qprime}
    -d \frac{\p}{\pprime}
    +d\frac{\p}{\qprime}}
    \Norm{\jump{v}}^\p_{\L^{\q}(\F)}
    \big)^\frac1\p, \label{norm-broken}\\
\Normth{v}_{\W^{1,\p;\q}_{\GammaD}(\Tauh)}
&   := \Norm{\nablaboldh v}_{\Lbf^\p(\Omega)}
    + \big( \sum_{\F\in\FcalhID} 
    \hF^{-\frac{\p}{\qprime}
    -d \frac{\p}{\pprime}
    +d\frac{\p}{\qprime}}
    \Norm{\PizFcalhz \jump{v}}^\p_{\L^{\q}(\F)}
    \big)^\frac1\p. \label{norm-broken-average}
\end{align}
\end{subequations}
Both seminorms are stronger than the seminorm $\Norm{\nablaboldh v}_{\Lbf^\p(\Omega)}$.

\subsection{Inequalities in broken Sobolev
and Crouzeix-Raviart spaces}
\label{subsection:semi-norms-technical}
We recall the following technical, preliminary result,
cf.~\cite[Theorem~1.7]{Botti-Mascotto:2025-B}
and \cite[Theorem~3.4]{Botti-Mascotto:2025-C}.
Amongst other things, it states that the seminorms
in~\eqref{functionals-broken} are norms.
Recall that the spatial dimension $d$ is~$2$.

\begin{lemma}[Sobolev-Poincar\'e and -trace inequalities in broken Sobolev spaces]
    \label{lemma:trace-broken-W1P}
Let~$\Omega$ satisfy either~\eqref{Omega-ass-1} or~\eqref{Omega-ass-2}.
Let~$\{\Tauh\}$ be a family of meshes
as in Section~\ref{subsection:meshes-broken-finite}
and $\p$ be in $[1,d)$.
There exist positive constants~$\CSPstar$
and $\CSTsharp$ depending on~$\p$, $\GammaD$,
$\Omega$ through either~$\Rfrak$ or~$\Nfrak$
in~\eqref{Omega-ass-1} and~\eqref{Omega-ass-2}, $d$,
and the shape-regularity parameter~$\gamma$ such that,
for all $v \in \W^{1,\p}_{\GammaD}(\Tauh)$,
\begin{subequations} \label{inequalities-broken}
\begin{align}
\Norm{v}_{\L^{\q}(\Omega)}
& \le \CSPstar \Norm{v}_{\W^{1,\p;\psharp}_{\GammaD}(\Tauh)}
& \forall \qop\in [1,\pstar], \label{eq:broken-Sobolev-sharp}\\
\Norm{v}_{\L^{\q}(\Gamma)} 
& \le \CSTsharp \Norm{v}_{\W^{1,\p;\psharp}_{\GammaD}(\Tauh)}
& \forall \q\in[1, \psharp].
     \label{eq:sob-theorem-2b}
\end{align}
\end{subequations}
The explicit dependence of $\CSPstar$ and~$\CSTsharp$
on the parameters highlighted above
is discussed in \cite[Theorem 1.7]{Botti-Mascotto:2025-B}
and \cite[Theorem~3.4]{Botti-Mascotto:2025-C}.
\end{lemma}
Now, we establish a variant of Lemma~\ref{lemma:trace-broken-W1P},
where weaker norms on the right-hand sides
of~\eqref{inequalities-broken} are employed,
notably, we replace $\Norm{v}_{\W^{1,\p;\q}_{\GammaD}(\Tauh)}$
by $\Normth{\cdot}_{\W^{1,\p;\q}_{\GammaD}(\Tauh)}$.


\begin{proposition}[Improved Sobolev-Poincar\'e and -trace inequalities in broken Sobolev spaces] \label{proposition:Sobolev-inequalities}
Let~$\Omega$ satisfy either~\eqref{Omega-ass-1} or~\eqref{Omega-ass-2}.
Let~$\{\Tauh\}$ be a family of meshes
as in Section~\ref{subsection:meshes-broken-finite}
and $\p$ be in $[1,d)$.
There exist positive constants~$\CtildeSPstar$
and~$\CtildeSTsharp$ depending on~$\p$,
$\GammaD$, $\Omega$
through either~$\Rfrak$ or~$\Nfrak$
in~\eqref{Omega-ass-1} and~\eqref{Omega-ass-2}, $d$,
and the shape-regularity parameter~$\gamma$ such that,
for all $v \in \W^{1,\p}_{\GammaD}(\Tauh)$,
\begin{subequations} \label{inequalities-broken-weak}
\begin{align}
\Norm{v}_{\L^{\q}(\Omega)}
& \le \CtildeSPstar \Normth{v}_{\W^{1,\p;\psharp}_{\GammaD}(\Tauh)} 
    & \forall \q\in[1, \pstar], 
    \label{eq:broken-Sobolev-sharp-weak}\\
\Norm{v}_{\L^{\q}(\Gamma)} 
& \le \CtildeSTsharp \Normth{v}_{\W^{1,\p;\psharp}_{\GammaD}(\Tauh)}
    & \forall \q\in[1, \psharp].
     \label{eq:sob-theorem-2b-weak}
\end{align}
\end{subequations}
\end{proposition}
\begin{proof}
A proof of~\eqref{eq:broken-Sobolev-sharp-weak} is given
in \cite[Corollary 1.10]{Botti-Mascotto:2025-B}.
Therefore, we focus on~\eqref{eq:sob-theorem-2b-weak}.
We assume that~$\q=\psharp$,
since the assertion for smaller values of~$\qop$
follows from H\"older's inequality.

Recall that $\PizTauhz$ denotes the piecewise average operator over~$\Tauh$.
Triangle's inequality gives
\begin{equation} \label{eq:cor2-0}
\Norm{v}_{\L^{\psharp}(\Gamma)}
\le \Norm{v-\PizTauhz v}_{\L^{\psharp}(\Gamma)} 
    + \Norm{\PizTauhz v}_{\L^{\psharp}(\Gamma)}.
\end{equation}
As for the first term on the right-hand side,
an elementwise trace inequality on the elements abutting $\Gamma$
(constant~$\CTRsharp$ depending
on the shape-regularity parameter~$\gamma$),
cf.~\eqref{Sobolev-trace-sharp-elemental},
an elementwise Sobolev embedding inequality
(constant~$\CSob$ depending on the
shape-regularity parameter~$\gamma$),
cf.~\eqref{Sobolev-embedding},
and an elementwise Poincar\'e-Steklov inequality
(constant~$\CPS$ depending
on the shape-regularity parameter~$\gamma$),
cf.~\eqref{Poincare-Steklov},
lead to
\[
\Norm{v-\PizTauhz v}_{\L^{\psharp}(\Gamma)} 
\le \CTRsharp (1+\CPS\CSob) \Norm{\nablaboldh v}_{\Lbf^\p(\Omega)}.
\]
Inserting this inequality in~\eqref{eq:cor2-0} yields
\begin{equation} \label{eq:cor1-0}
\Norm{v}_{\L^{\psharp}(\Gamma)}
\le \CTRsharp (1+\CPS\CSob) \Norm{\nablaboldh v}_{\Lbf^\p(\Omega)}
    + \Norm{\PizTauhz v}_{\L^{\psharp}(\Gamma)}.
\end{equation}
The assertion follows by estimating the second term on the right-hand side.
We apply Lemma~\ref{lemma:trace-broken-W1P} and get,
for~$\CSTsharp$ as in~\eqref{eq:sob-theorem-2b},
\begin{equation} \label{eq:cor1-1}
\Norm{\PizTauhz v}_{\L^{\psharp}(\Omega)}
\le \CSTsharp \Big( \sum_{\F\in\FcalhID}
        \Norm{\jump{\PizTauhz v}}_{\L^{\psharp}(\F)}^\p \Big)^\frac1\p .
\end{equation}
Recall that $\PizFcalhz$ is the piecewise average operator over~$\Fcalh$.
We estimate each jump term separately:
\begin{equation} \label{eq:cor1-2}
\Norm{\jump{\PizTauhz v}}_{\L^{\psharp}(\F)}
\le  \Norm{\jump{\PizTauhz v} - \PizFcalhz \jump{v}}_{\L^{\psharp}(\F)}
        +  \Norm{\PizFcalhz \jump{v}}_{\L^{\psharp}(\F)}.
\end{equation}
The second term on the right-hand side is good to go.
As for the first one,
proceeding exactly as in the proof of \cite[Corollary~1.10]{Botti-Mascotto:2025-B},
standard manipulations give the existence
of a positive constant~$C$ depending only on~$\gamma$ and $\p$ such that
\begin{equation*} \label{eq:cor1-3}
\Big( \sum_{\F\in\FcalhID}
        \Norm{\jump{v-\PizTauhz v}}_{\L^{\psharp}(\F)}^\p \Big)^\frac1\p
    \le C  \Norm{\nablaboldh v}_{\Lbf^\p(\Omega)}.
\end{equation*}
Combining the above display,
\eqref{eq:cor1-0}, \eqref{eq:cor1-1}, \eqref{eq:cor1-2}
entails the assertion.
\end{proof}

Due to definitions~\eqref{CR}
and~\eqref{norm-broken-average},
we note that
\begin{equation*}
\begin{split}
\Normth{\qh}_{\W^{1,\p;\q}_{\GammaD}(\Tauh)}  =
    & \Norm{\nablaboldh \qh}_{\Lbf^\p(\Omega)} \\
    &   + \Big( \sum_{\F\in\FcalhD} \hF^{-\frac{\p}{\qprime}
            -d \frac{\p}{\pprime} +d\frac{\p}{\qprime}}
    \Norm{\PizFcalhz g}^\p_{\L^{\q}(\F)} \big)^\frac1\p
\quad
\forall \qh \in \CR_{k,g}(\Tauh, \FcalhD).
\end{split}
\end{equation*}
Thus, we have the following
immediate consequence of Proposition~\ref{proposition:Sobolev-inequalities}.
\begin{corollary}[Improved Sobolev-Poincar\'e and -trace inequalities in Crouzeix-Raviart spaces] \label{corollary:Sobolev-inequalities-CR}
Let~$\Omega$ satisfy either~\eqref{Omega-ass-1} or~\eqref{Omega-ass-2}.
Let~$\{\Tauh\}$ be a family of meshes
as in Section~\ref{subsection:meshes-broken-finite}
and $\p$ be in $[1,d)$.
Given~$\CtildeSPstar$ and $\CtildeSTsharp$
as in~\eqref{inequalities-broken-weak},
for all $\qh$ in $\CR_{k,\gD}(\Tauh, \FcalhD)$,
$\gD$ in~$\GcalD$,
\begin{subequations} \label{inequalities-CR-weak}
\begin{align}
\Norm{\qh}_{\L^{\q}(\Omega)}
& \le \CtildeSPstar \Big[\Norm{\nablaboldh \qh}_{\Lbf^\p(\Omega)}
    + \big( \sum_{\F\in\FcalhD}
    \Norm{\PizFcalhz \gD}^\p_{\L^{\psharp}(\F)} \big)^\frac1\p \Big]
    & \forall \q\in[1, \pstar], 
    \label{eq:broken-CR-sharp-weak}\\
\Norm{\qh}_{\L^{\q}(\Gamma)} 
& \le \CtildeSTsharp \Big[ \Norm{\nablaboldh \qh}_{\Lbf^\p(\Omega)}
    + \Big( \sum_{\F\in\FcalhD}
    \Norm{\PizFcalhz \gD}^\p_{\L^{\psharp}(\F)} \big)^\frac1\p \Big]
    & \forall \q\in[1, \psharp].
     \label{eq:sob-CR-2b-weak}
\end{align}
\end{subequations}
\end{corollary}

\begin{remark} \label{remark:broken-inequalities}
Corollary~\ref{corollary:Sobolev-inequalities-CR}
extends \cite[Propositions~4 and~5]{Girault-Wheeler:2008}
to the general-order case.
The tools employed in the proof
of Corollary~\ref{corollary:Sobolev-inequalities-CR}
do not rely on any averaging operator.
As such, the constants in~\eqref{inequalities-CR-weak}
do not depend on the structure of the discretization space,
with the exception of the zero moments up to order~$\p-1$
coming from the definition of the CR spaces in~\eqref{CR};
this might be helpful in the proof of the convergence of the scheme,
e.g., on a fixed mesh
and increasing the polynomial degree.
\eremk
\end{remark}

\section{A general-order method}
    \label{section:standard-method}
For a generic $g$ in~$\GcalD$, we introduce the finite element spaces,
\begin{equation*} \label{standard:FE-spaces}
    \Vbfcalhkmo:= \Pbf_{k-1}(\Tauh),
        \qquad\qquad
    \Qcalhgk :=
    \begin{cases}
        \CR_k(\Tauh) \cap \L^2_0(\Omega)          & \text{ if } \GammaN = \Gamma \\
        \CR_{k,g}(\Tauh, \FcalhD)                 & \text{ if } \GammaN \neq \Gamma.        
    \end{cases}
\end{equation*}
We consider the
method:
find $(\ubfh,\ph) \in \Vbfcalhkmo \times \QcalhgDk$ such that
\begin{subequations} \label{discrete-formulation}
\begin{align}
    \frac{\mu}{\rho} \int_{\Omega}(\Kbbmo)\ubfh \cdot \vbfh \dxbf
        + \frac{\beta}{\rho} \int_{\Omega}\EucNorm{\ubfh}^{\alphamt} \ubfh \cdot \vbfh \dxbf
        + \int_{\Omega}\nablaboldh \ph \cdot \vbfh \dxbf
        = 0 &
        \qquad \forall\, \vbfh \in \DGkmo \label{discrete-formulation:first-eq} \\
    \int_{\Omega}\ubfh \cdot \nablaboldh \qh \dxbf
        = - \int_{\Omega} b \,\qh\dxbf
        + \int_{\GammaN} \gN \, \qh \ds&
        \qquad \forall\, \qh \in \Qcalhzerok. \label{constraint}
\end{align}
\end{subequations}
The method \eqref{discrete-formulation}
with $k=1$, $\alpha=3$, and pure Neumann boundary conditions
was analyzed in \cite{Girault-Wheeler:2008}.
This section is devoted to investigate the well-posedness of
(Section~\ref{subsection:well-posedness-discrete})
and error estimate for
(Section~\ref{section:error-estimates})
method~\eqref{discrete-formulation}.

\subsection{Well-posedness}\label{subsection:well-posedness-discrete}
We define the discrete counterpart
of the space $\Dbfcal$ in \eqref{Dbfcal} as
\begin{equation*} \label{Dbfcalh}
    \Dbfcalhkmo := \Big\{ \vbfh \in \DGkmo \, \Big| \,
        \int_\Omega \nablaboldh \qh \cdot \vbfh \dxbf = 0
        \quad \forall \qh \in \Qcalhzerok \Big\}
\end{equation*}
and its $\L^2$-orthogonal in~$\DGkmo$ as
\begin{equation*} \label{Dbfcalh-perp}
     \Dbfcalhperp := \Big\{ \wbfh \in \DGkmo \, \Big| \, 
     \int_{\Omega} \wbfh \cdot \vbfh \dxbf = 0
     \quad \forall \vbfh \in \Dbfcalhkmo   \Big\}.
\end{equation*}

We prove a discrete inf-sup condition.
\begin{lemma}
For all $\afrak$ in $(1,\infty)$,
there exists a positive constant~$\bfrakh$
only depending on~$\afrak$,
the shape-regularity parameter~$\gamma$, and the order~$k$,
such that
\begin{equation} \label{discrete-inf-sup}
\inf_{\qh \in \QcalhgDk}\sup_{\vbfh \in \Vbfcalhkmo} \frac{\int_\Omega \vbfh\cdot\nablaboldh \qh}{\Norm{\nablaboldh \qh}_{\LafrakprimeOmegad}\Norm{\vbfh}_{\LafrakOmegad}} \ge \bfrakh.
\end{equation}
\end{lemma}
\begin{proof}
Given $\qh$ in $\QcalhgDk$
and~$\PiboldzTauhkmo$ as in~\eqref{L2-projector-mesh},
an $\Lbf^{\afrak}-\Lbf^2$ polynomial inverse inequality gives
\begin{equation} \label{vh-infsup}
\wbfh
:= \vert \nablaboldh \qh \vert^{\afrakprime-2}    \nablaboldh \qh,
\qquad
\Norm{\PiboldzTauhkmo \wbfh}_{\Lbf^\afrak(\Omega)}
\le \bfrakh^{-1}
\Norm{\wbfh}_{\Lbf^\afrak(\Omega)}
= \bfrakh^{-1}
\Norm{\nablaboldh \qh}_{\Lbf^{\afrakprime}(\Omega)}^{\alphaprime-1}.
\end{equation}
A duality argument reveals
\begin{equation*} \label{discrete-inf-sup:leq-proof}
\begin{split}
 \Norm{\nablaboldh \qh}_{\Lbf^{\afrakprime}(\Omega)}
&= \frac{\Norm{\nablaboldh \qh}_{\Lbf^{\afrakprime}(\Omega)}^{\afrakprime}}
    {\Norm{\nablaboldh \qh}_{\Lbf^{\afrakprime}(\Omega)}^{\afrakprime-1}}
    = \int_{\Omega} \nablaboldh \qh \cdot
    \frac{\vert \nablaboldh \qh \vert^{\afrakprime-2}  \nablaboldh \qh}
    {\Norm{\nablaboldh \qh}_{\Lbf^{\afrakprime}(\Omega)}^{\afrakprime-1}} \\
&\overset{\eqref{vh-infsup}}{\le}
    \bfrakh^{-1} \int_{\Omega} \nablaboldh \qh \cdot
    \frac{\PiboldzTauhkmo \wbfh}
    {\Norm{\PiboldzTauhkmo \wbfh}_{\Lbf^{\afrak}(\Omega)}}
    \leq \bfrakh^{-1}
    \sup_{\vbfh \in \Vbfcalhkmo} \int_{\Omega} \nablaboldh \qh \cdot
    \frac{\vbfh}{\Norm{\vbfh}_{\Lbf^{\afrak}(\Omega)}}.
\end{split}
\end{equation*}
Taking the infimum over~$\qh$ in~$\QcalgD$ on both sides
of the above display
yields~\eqref{discrete-inf-sup}.

\end{proof}

We prove the discrete version of Proposition \ref{prop:existence-ubfl}.
\begin{lemma} \label{lemma:existence-ubfhl}
Given $b$ and  $\gN$ as in \eqref{data-spaces},
there exists~$\ubfhl$ in~$\Dbfcalhperp$ satisfying
\[
 \int_{\Omega} \ubfhl \cdot \diveh \qh \dxbf
        = - \int_{\Omega} b\,\qh\dxbf
        + \int_{\GammaN} \gN \, \qh \ds
        \qquad\qquad \forall\, \qh \in \Qcal^k_{h,0}
\]
and a positive constant~$\Ctildeell$
depending on the order~$k$ through~$\bfrakh$
in~\eqref{discrete-inf-sup}, $\GammaD$,
$\Omega$ through either~$\Rfrak$ or~$\Nfrak$
in~\eqref{Omega-ass-1} and~\eqref{Omega-ass-2},
and the shape-regularity parameter~$\gamma$ such that
\begin{equation} \label{ubfhl-estimate}
\Norm{\ubfhl}_{\LalphaOmegad}
\leq \Ctildeell \Big(\Norm{b}_{\L^{((\alphaprime)^*)'}(\Omega)} 
    + \Norm{\gN}_{\L^{((\alphaprime)^\sharp)'}(\GammaN)}\Big).
\end{equation}
\end{lemma}
\begin{proof}
We consider two cases.

\paragraph*{Case 1: mixed/Dirichlet boundary conditions.}
Triangle's inequality,
Cauchy-Schwarz' inequality,
the broken Sobolev-Poincar\'e inequality \eqref{eq:broken-CR-sharp-weak},
and the broken trace inequality \eqref{eq:sob-CR-2b-weak}
yield,
for a hidden constant only depending on~$k$, $\GammaD$, $\Omega$ through either~$\Rfrak$ or~$\Nfrak$
in~\eqref{Omega-ass-1} and~\eqref{Omega-ass-2},
and~$\gamma$,
\small
\begin{equation*}
\begin{split}
\left| -\int_{\Omega}\right.& \left.b\,\qh\dxbf + \int_{\GammaN} g\,\qh\ds \right|
     \lesssim \left( \Norm{b}_{\L^{((\alphaprime)^*)'}(\Omega)} + \Norm{\gN}_{\L^{((\alphaprime)^\sharp)'}(\GammaN)}\right)
    \Norm{\nablaboldh \qh}_{\Lbf^\alphaprime(\Omega)}
    \qquad \forall \qh \in \Qcalhzerok,
\end{split}
\end{equation*} \normalsize
where the mapping $\qh \to -\int_{\Omega}b\,\qh\dxbf + \int_{\GammaN} g\,\qh\ds$
belongs to $(\Qcalhzerok)^*$.
The assertion follows from the discrete inf-sup condition \eqref{discrete-inf-sup}.

\paragraph*{Case 2: pure Neumann boundary conditions.}
For all $c$ in $\Rbb$,
triangle's inequality,
Cauchy-Schwarz' inequality,
the compatibility condition \eqref{compatibility-condition},
the broken Sobolev-Poincar\'e
inequality~\eqref{eq:broken-CR-sharp-weak},
and the discrete broken trace
inequality~\eqref{eq:sob-CR-2b-weak} yield,
for a hidden constant only depending on~$k$, $\GammaD$, $\Omega$
through either~$\Rfrak$ or~$\Nfrak$
in~\eqref{Omega-ass-1} and~\eqref{Omega-ass-2},
and~$\gamma$,
\begin{equation*}
\begin{split}
&   \left| -\int_{\Omega}b\,\qh\dxbf + \int_{\GammaN} g\,\qh\ds \right|
\leq \left|\int_{\Omega}b\,(\qh-c)\dxbf \right|
    + \left|\int_{\GammaN} g\,(\qh-c)\ds \right| \\
&   \leq \left( \Norm{b}_{\L^{((\alphaprime)^*)'}(\Omega)}
    + \Norm{\gN}_{\L^{((\alphaprime)^\sharp)'}(\GammaN)}\right)
    \left( \Norm{\qh-c}_{\L^{(\alphaprime)^*}(\Omega)}
    + \Norm{\qh-c}_{\L^{(\alphaprime)^\sharp}(\GammaN)} \right) \\
&   \lesssim \left( \Norm{b}_{\L^{((\alphaprime)^*)'}(\Omega)} + \Norm{\gN}_{\L^{((\alphaprime)^\sharp)'}(\GammaN)}\right) \Norm{\nablaboldh \qh}_{\Lbf^\alphaprime(\Omega)}
    \quad \forall \qh \in \Qcalhzerok.
\end{split}
\end{equation*}
The mapping $\qh \to -\int_{\Omega}b\,\qh\dxbf + \int_{\GammaN} g\,\qh\ds$
belongs to $(\Qcalhzerok)^*$.
The assertion follows from the discrete
inf-sup condition \eqref{discrete-inf-sup}.
\end{proof}

We now consider the discrete version of \eqref{splitted-variational-formulation}:
find $\ubfhzero$ in $\Dbfcalhkmo$
such that
\begin{equation} \label{splitted-discrete-formulation}
    \frac{\mu}{\rho}\int_{\Omega}(\Kbbmo)(\ubfhzero + \ubfhl)\cdot \vbfh \dxbf
        + \frac{\beta}{\rho}\int_{\Omega}|\ubfhzero + \ubfhl|^{\alphamt}(\ubfhzero + \ubfhl)\cdot \vbfh \dxbf
        = 0
    \quad \forall \vbfh \in \Dbfcalhkmo.
\end{equation}
A finite-dimensional modification
of the arguments in Proposition \ref{proposition:well-posedness}
and the discrete inf-sup condition~\eqref{discrete-inf-sup}
entail the following result.

\begin{proposition}  \label{proposition:discrete-well-posedness-enriched}
Given $b$ and $\gN$
as in \eqref{data-spaces},
problem \eqref{discrete-formulation}
admits a unique solution $(\ubfh,\ph)$
in $\Vbfcalhkmo \times \QcalhgDk$.
Moreover, given~$\Ctildeell$ as in~\eqref{ubfhl-estimate},
we have
\begin{subequations}
\small\begin{equation} \label{apriori-ubfh}
\begin{split}
\Norm{\ubfh}_{\LalphaOmegad}
&\leq \Big( \frac{2^{\alphaprime}(\alpha-1)\mu^{\alphaprime}}{\beta^{\alphaprime}}
    \Norm{\Kbb^{-1}}_{\Lbu^{\frac{\alpha}{\alpha-2}}}
    ^{\alphaprime-1} +2 \Big)
    ^{\frac1\alpha}
    \Ctildeell^{\alphaprime-1}
    \big( \Norm{b}_{\L^{((\alphaprime)^*)'}(\Omega)}
    + \Norm{\gN}_{\L^{((\alphaprime)^\sharp)'}(\GammaN)}
    \big)^{\alphaprime-1}
\end{split}
\end{equation}\normalsize
and
\begin{equation} \label{apriori-ph}
\begin{split}
\Norm{\nablaboldh\ph}_{\LalphaprimeOmegad}
&\leq \Big(\frac{\mu}{\rho}
    \Norm{\Kbbmo}_{\Lbu^\frac{\alpha}{\alphamt}(\Omega)}
    + \frac{\beta}{\rho} \Big)
    \Big( \frac{2^{\alphaprime}(\alpha-1)\mu^{\alphaprime}}{\beta^{\alphaprime}}
    \Norm{\Kbb^{-1}}_{\Lbu^{\frac{\alpha}{\alpha-2}}}
    ^{\alphaprime-1} +2 \Big)
    ^{\frac1{\alphaprime}} \cdot \\
&   \qquad\qquad
    \cdot \Ctildeell
    \big(\Norm{b}_{\L^{((\alphaprime)^*)'}(\Omega)}
    + \Norm{\gN}_{\L^{((\alphaprime)^\sharp)'}(\GammaN)}\big).
\end{split}
\end{equation}
\end{subequations}
\end{proposition}

\begin{remark} \label{remark:hdiv-conformity}
Given $\vbfh$ in $\Dbfcalhzero$, we have
\begin{equation*} 
\begin{split}
     \int_{\Fcalh} \PizFcalhz \qh \ \jump{\vbfh}_{\nfrak} \dxbf 
        &\overset{( \vbfh \in \Vbfcal_{\h}^0)}{=}
        \int_{\Fcalh} \PizFcalhz \qh \ \jump{\vbfh}_{\nfrak} \dxbf 
            -  \int_\Omega\qh \diveh \vbfh \\
        &\quad \overset{\text{(IBP)}}{=}
            \int_\Omega \nablaboldh \qh \cdot \vbfh \dxbf
         \overset{( \vbfh \in \Dbfcalhzero)}{=} 0
         \qquad\qquad \forall \qh \in \Qcal_{h,0}^1,
\end{split}
\end{equation*}
whence $\jump{\vbfh}_\nfrak = 0$ on each $\F$ in $\Fcalh$.
As such, $\dive \vbfh$ belongs to $\LalphaOmegad$
and in particular we have
\begin{equation*} 
    \Dbfcalhzero = \{\vbfh \in \Vbfcal_{\h}^0 \, | \, \dive \vbfh = 0 \}.
\end{equation*}
The characterization above
is not valid for general $k$ larger than~$1$.
Indeed, the term $(\qh,\diveh\vbfh)$ does not vanish in general;
to see this, consider $\vbfh$ in $\Dbfcalhkmo$,
observe that~$\dive \vbfh{}_{|\E}$ belongs to~$\P_{k-2}(\E)$ on each $\E$ in $\Tauh$ and
the bulk moments of~$\qh$ in~\eqref{dof-moment-bulk} are up to order~$k-3$.
 On the one hand,
this fact does not undermine the well-posedness of the method
as well as the derivation of error estimates of optimal order
since in the dual-mixed formulation \eqref{discrete-formulation}
the inclusion $\Dbfcalhkmo \subset \Dbfcal$ is not necessary,
while it would be essential in the standard mixed counterpart.
\eremk
\end{remark}

\section{Convergence analysis} \label{section:convergence}
In this section,
we discuss the strong convergence in the natural norms
of sequences of solutions
$(\ubfh,\ph)$ to method~\eqref{discrete-formulation}.
To this aim, we proceed by showing
the following intermediate steps:
\begin{itemize}
\item introduce $\cbfh$ in $\Lbf^{\alpha}(\Omega)$,
cf. \eqref{def-cbfh}, which represents
a lifting of the discrete divergence of~$\ubf$,
and show its strong convergence in~$\Lbf^{\alpha}(\Omega)$
to $\zerobf$, see Theorem~\ref{theorem:estimate-ch};
\item show strong convergence of several interpolation, projection, \dots
operators, see Lemma~\ref{lemma:strong-convergence-Ph-Lalpha};
\item show weak convergence in $\Lbf^\alpha(\Omega)$ of the fluxes,
      see Lemma~\ref{lemma:weak-conv-flux-Lalpha};
\item show strong convergence in $\Lbf^2(\Omega)$ of the fluxes,
      see Lemma~\ref{lemma:strong-conv-flux-L2};
\item show weak convergence in $\Lbf^{\alphaprime}(\Omega)$
    of the gradient of the potentials,
      see Proposition~\ref{prop:weak-conv-pot-Lalphaprime};
\item show strong convergence in $\Lbf^{\alpha}$ of the fluxes,
      see Theorem~\ref{theorem:strong-conv-flux-Lalpha};
\item show strong convergence in $\Lbf^{\alphaprime}(\Omega)$
    of the gradient of the potentials,
      see Theorem~\ref{theorem:strong-conv-pot-Lalphaprime}.
\end{itemize}
Several of these results will be proven
under an additional assumption:
\begin{equation} \label{extra-ass-convergence}
\alpha\in(2,4).
\end{equation}
We preliminarily introduce the Crouzeix-Raviart interpolant~$\Icalk$;
see, e.g., \cite{Fortin-Soulie:1983,Ainsworth-Rankin:2008,Bressan-Mascotto-Mosconi:2025};
due to the intrinsic different nature of the spaces
depending on the parity of the order~$k$ of the scheme,
cf. Section~\ref{subsection:meshes-broken-finite},
we distinguish two cases.
For odd $k$ and given $\qop$ in $\W^{1,1}(\Omega)$,
we consider the interpolant $\Icalk$ defined such that
the DoFs in~\eqref{dof-CRk-odd}
of~$\qop$ and~$\Icalk \qop$ coincide.
Instead, for even~$k$ and~$\qop$ in $\W^{1,t}(\Omega)$,
$t>2$,
we consider the modal interpolant~$\Icalk$:
$\Icalk \qop$ is the unique function
in the Crouzeix-Raviart space such that
\begin{itemize}
\item $\Icalk \qop(\nu) := \qop(\nu)$
        for all~$\nu$ in~$\Vcalh$;
\item the facet~\eqref{dof-moment-edge}
        and bulk~\eqref{dof-moment-bulk} moments
        of~$\qop$ and~$\Icalk \qop$ coincide;
\item   the coefficients of the nonconforming bubbles
    in~\eqref{nc-elemental-bubble-function}
    in the expansion of~$\Icalk \qop$
    are set to~$0$.
\end{itemize}
For an arbitrary order (even or odd)~$k$,
the above definitions of~$\Icalk$ imply
\begin{subequations} \label{interpolant-CR}
\begin{align}
& \int_{F} (\qop-\Icalk \qop)  \, \SFj \ds = 0
& \forall j = 0, \dots, k-1, \quad\forall \F\in \FcalhI,
    \label{interpolant-moment-edge}\\
& \int_{\E} (\qop-\Icalk \qop)\, m_{\boldsymbol{\beta}}^K \dxbf = 0
&   \forall \vert \boldsymbol{\beta}\vert = 0, \dots, k-3, \ \forall \E \in \Tauh
    \label{interpolant-moment-bulk}.
\end{align}
\end{subequations}
For $k$, $r$, and $m$ larger than or equal to~$1$,
standard techniques~\cite{Brenner-Scott:2008} entail
\small
\begin{equation} \label{cr-interpolation:estimates}
\h^{-1}\Norm{\qop-\Icalk \qop}_{\L^m(\Omega)}
+ \Norm{\nablaboldh(\qop - \Icalk \qop)}_{\Lbf^m(\Omega)}
\lesssim \h^{s-1} \Norm{\qop}_{\W^{s,m}(\Omega)}
\quad \forall \qop \in \W^{r,m}(\Omega),
\, s  := \min\{r,k+1\}.
\end{equation} \normalsize

\subsection{Weak convergence in the natural norms}
    \label{subsection:weak-convergence}
Given~$\wbf$ in $\Lbf^{\afrak}(\Omega)$
with $\dive\wbf$ in $\L^{((\afrakprime)^*)'}(\Omega)$,
$\afrak$ larger than~$2$,
let $\cbfh = \cbfh(\wbf)$ in~$\Dbfcalhkmoperp$
solve
\begin{equation} \label{def-cbfh}
   \int_\Omega \cbfh \cdot \nablaboldh \qh \dxbf
    = -\int_\Omega \wbf \cdot \nablaboldh\qh \dxbf
      -\int_\Omega \dive \wbf \, \qh \dxbf
      +  \langle \wbf\cdot \nbfOmega, \qh\rangle_{\GammaN}
      \qquad \forall\qh\in \Qcalhzerok.
\end{equation}
The existence of a solution to~\eqref{def-cbfh}
follows from the discrete
inf-sup condition~\eqref{discrete-inf-sup}
together with the fact that the right-hand side is a linear functional on $\Qcalhzerok$,
as can be shown by arguments analogous to those
of the proof of Lemma~\ref{lemma:existence-ubfhl}.
We further set
\begin{equation} \label{def:Ph}
    \Pbfh(\wbf) := \PiboldzTauhkmo \wbf + \cbfh(\wbf).
\end{equation}
The operator~$\PiboldRTTauhkmo$ denotes
the Raviart-Thomas interpolant of order~$k-1$.
Given~$\obbF$ the indicator function of~$\F$ in~$\partial\E$
(for a $\E=\E_\F$ in $\Tauh$ with $\F$ in $\FcalE$),
$\PiboldRTTauhkmo$ in particular satisfies
\begin{align}
&   ( (\wbf-\PiboldRTTauhkmo\wbf), \qbfkmtE )_\E = 0
    & \forall \qbfkmtE \in \Pbf_{k-2}(\E),\;
        \forall\E\in\Tauh, \label{RT-dofs-bulk} \\
&   \dive(\PiboldRTTauhkmo \wbf) = \PizTauhkmo\dive\wbf
&   \text{on each } \E, \label{div-RT} \\
&   ((\PiboldRTTauhkmo \wbf)\cdot\nbfF),\qkmoF )_{0,\F}
    =   \langle \wbf\cdot\nbfF , \qkmoF \obbF \rangle_{\partial\E}
    & \forall \qkmoF\in\P_{k-1}(\F),\;
    \forall\F\in\Fcalh
    \label{trick-RT-indicator}.
\end{align}
The right-hand side in~\eqref{trick-RT-indicator}
is well-posed due to the
regularity of~$\wbf$ and~$\obbF$,
which belongs to
$\W^{\frac{1}{\afrak},\afrakprime}(\partial\E)$,
$\afrak$ larger than~$2$;
cf. \cite[Ch. 17]{Ern-Guermond:2021}.

\begin{theorem} \label{theorem:estimate-ch}
Let~$\afrak$ be larger than~$2$.
Let~$\cbfh$ in~$\Dbfcalhkmoperp$
be solution to~\eqref{def-cbfh}
for~$\wbf$ in $\Lbf^{\afrak}(\Omega)$
with $\dive\wbf$ in $\L^{((\afrakprime)^*)'}(\Omega)$.
Then, for hidden constants
only depending on the degree~$k$ and
the shape-regularity parameter~$\gamma$, we have
\begin{equation} \label{cbfh-bound}
\begin{split}
\Norm{\cbfh}_{\Lbf^{\afrak}(\Omega)}
&\lesssim 
\Norm{\wbf - \PiboldRTTauhkmo \wbf}_{\Lbf^{\afrak}(\Omega)}
+ \Norm{\dive\wbf - \PizTauhkmo \dive\wbf}_{\L^{((\afrakprime)^*)'}(\Omega)}.
\end{split}
\end{equation}
\end{theorem}

\begin{proof}
For all $\qh$ in $\Qcalhzerok$, we have
\small\begin{equation} \label{ch-nablahqh}
\begin{split}
&   \int_\Omega\cbfh(\wbf) \cdot \nablaboldh\qh \dxbf
    \overset{\eqref{def-cbfh}}{=}
        - \int_\Omega \wbf \cdot \nablaboldh\qh \dxbf
          - \int_\Omega(\dive \wbf) \qh \dxbf
          + \int_{\GammaN}  (\wbf\cdot \nbfOmega) \qh \ds \\
    & \overset{\text{(IBP)}}{=}
        - \sum_{\F\in\Fcalh}
        \langle \wbf\cdot\nbfF, \jump{\qh}\obbF\rangle_{\partial\E_\F}\\
    & \overset{(\qh\in\Qcalhzerok),\eqref{trick-RT-indicator}}{=}
        - \sum_{\F\in\Fcalh}
        \Big\langle
        (\wbf - \PiboldRTTauhkmo\wbf)\cdot\nbfF),
        \Big(\jump{\qh} - \jump{\PizTauhz \qh}\Big)
        \obbF
        \Big\rangle_{\partial\E_\F} \\
     & \leq  \sum_{\E \in \Tauh}
     \Big( \sum_{\F\in\FcalE}
        \Norm{(\wbf
        - \PiboldRTTauhkmo\wbf)\cdot\nbfE}_{\big(\W^{\frac{1}{\afrak},\afrakprime}(\F)\big)^*}^{\afrak}\Big)^{\frac{1}{\afrak}}
     \Norm{\qh-\PizTauhz \qh}_{\W^{\frac{1}{\afrak},\afrakprime}(\partial\E)}.
\end{split}
\end{equation}\normalsize
The dual norms on the right-hand side are well-posed
since $\afrak$ is larger than~$2$.
We first derive a bound for the second term on the
right-hand side of~\eqref{ch-nablahqh}.
The continuous trace inequality~\eqref{global-trace}
and Poincar\'e-Steklov's inequality~\eqref{Poincare-Steklov} entail
\begin{equation} \label{estimate-term-qh}
\begin{split}
\Norm{\qh-\PizTauhz \qh}
_{\W^{\frac{1}{\afrak},\afrakprime}(\partial\E)}
\lesssim \Norm{\nablabold \qh}_{\Lbf^{\afrakprime}(\E)}.
\end{split}
\end{equation}
Next, we derive an upper bound
for the first term on the right-hand side
of~\eqref{ch-nablahqh}.
Given~$\F$ in~$\FcalE$, $\E$ in~$\Tauh$,
a duality argument gives
\begin{equation} \label{dual-norm-unb}
\begin{split}
&\Norm{(\wbf - \PiboldRTTauhkmo\wbf) \cdot\nbfE}
    _{(\W^{\frac{1}{\afrak},\afrakprime}(\F))^*}
    := \sup_{\varphi \in \W^{\frac{1}{\afrak},\afrakprime}(\F)}
    \frac{\langle (\wbf - \PiboldRTTauhkmo\wbf) \cdot\nbfE,
    \varphi\rangle}
    {\Norm{\varphi}_{\W^{\frac{1}{\afrak},\afrakprime}(\F)}}.
\end{split}
\end{equation}
There exists a stable facet-to-cell lifting operator
$\LcalFE: \W^{\frac{1}{\afrak},\afrakprime}(\F) \to \W^{1,\afrakprime}(\E)$,
cf. \cite[Lem. 17.1 and Eq. (17.8)]{Ern-Guermond:2021},
which requires~$\afrak$ larger than~$2$
and combines the zero-extension
from an~$\F$ in~$\FcalE$ to $\partial\E$
with the right-inverse of the trace operator over~$\E$.
In particular, it holds true that 
\begin{equation} \label{zero-extension-lifting-stability-lifting}
\LcalFE(\phi) = \phi \obbF
\quad \text{ on } \partial\E,
\qquad\quad
\SemiNorm{\LcalFE(\phi)}_{\W^{1,\afrakprime}(\E)}
\leq \CLcal \Norm{\phi}_{\W^{\frac{1}{\afrak},\afrakprime}(\F)}
\qquad\quad\forall \phi\in \W^{1,\afrakprime}(\E).
\end{equation}
For all~$\F$ and $\E_\F$ such that $\F$ belongs
to~$\FcalE$,
let~$\psi$ denote~$\LcalFEF(\varphi-\PizFcalhkmo\varphi)$.
H\"older's inequality entails
\begin{equation} \label{estimate-unbf-langle}
\begin{split}
&   \langle (\wbf - \PiboldRTTauhkmo\wbf)\cdot\nbfE,
    \varphi \rangle_\F\\
&   \overset{\eqref{trick-RT-indicator}}{=}
    \langle (\wbf - \PiboldRTTauhkmo\wbf)\cdot\nbfE,
    \varphi -\PizFcalhkmo\varphi \rangle_\F
    \overset{\eqref{zero-extension-lifting-stability-lifting}}{=} 
    \langle (\wbf - \PiboldRTTauhkmo \wbf) \cdot\nbfE,
    \psi \rangle_{\partial\E_\F}\\
&   \overset{\text{(IBP)}}{=}
    \int_{\E_\F}(\wbf - \PiboldRTTauhkmo \wbf) \cdot
    \nablabold \psi \dxbf
    + \int_{\E_\F} (\dive(\wbf - \PiboldRTTauhkmo\wbf))
            \psi\dxbf \\
&\overset{\eqref{Sobolev-embedding}, \eqref{Poincare-Steklov}, \eqref{div-RT}}{\lesssim}
    \Big(\Norm{\wbf - \PiboldRTTauhkmo \wbf}_{\Lbf^{\afrak}(\E_\F)}
    + \Norm{\dive\wbf - \PizTauhkmo \dive\wbf}_{\L^{((\afrakprime)^*)'}(\E_\F)}\Big)
    \SemiNorm{\psi}_{\W^{1,\afrakprime}(\E_\F)}\\
&\overset{\eqref{zero-extension-lifting-stability-lifting}}{\lesssim}
    \Big(\Norm{\wbf - \PiboldRTTauhkmo \wbf}_{\Lbf^{\afrak}(\E_\F)}
    + \Norm{\dive\wbf - \PizTauhkmo \dive\wbf}_{\L^{((\afrakprime)^*)'}(\E_\F)}\Big)
    \Norm{\varphi}_{\W^{\frac{1}{\afrak},\afrakprime}(\F)}.
\end{split}
\end{equation}
\normalsize
Combining~\eqref{dual-norm-unb} and~\eqref{estimate-unbf-langle}, we deduce
\begin{equation} \label{minchioni}
\begin{split}
&   \Norm{\wbf\cdot\nbfE - \PizFcalhkmo(\wbf\cdot\nbfE)}
    _{(\W^{\frac{1}{\afrak},\afrakprime}(\F))^*}\\
&   \lesssim 
    \Norm{\wbf - \PiboldRTTauhkmo \wbf}_{\Lbf^{\afrak}(\E_\F)}
    + \Norm{\dive\wbf - \PizTauhkmo \dive\wbf}
    _{\L^{((\afrakprime)^*)'}(\E_\F)}.
\end{split}
\end{equation}
Since
\small\begin{equation} \label{two-queen}
\frac{\afrak}{((\afrakprime)^*)'}
= \frac{d+\afrak}{d} >1,
\end{equation}\normalsize
Jensen's inequality for sequences entails
\[
\begin{split}
&   \big( \sum_{\E\in\Tauh}
    \Norm{\dive\wbf - \PizTauhkmo \dive\wbf}_{\L^{((\afrakprime)^*)'}(\E)}^\afrak   
    \big)^\frac1\afrak
    = \big( \sum_{\E\in\Tauh}
    \Norm{\dive\wbf - \PizTauhkmo \dive\wbf}_{\L^{((\afrakprime)^*)'}(\E)}
    ^{((\afrakprime)^*)' \frac{\afrak}{((\afrakprime)^*)'}}   
    \big)^\frac1\afrak \\
&   \overset{\eqref{two-queen}}{\le}
    \big( \sum_{\E\in\Tauh}
    \Norm{\dive\wbf - \PizTauhkmo \dive\wbf}_{\L^{((\afrakprime)^*)'}(\E)}
    ^{((\afrakprime)^*)'}   
    \big)^\frac{1}{((\afrakprime)^*)'}
    = \Norm{\dive\wbf - \PizTauhkmo \dive\wbf}_{\L^{((\afrakprime)^*)'}(\Omega)}.
\end{split}
\]
Combining the above display,
\eqref{minchioni}, \eqref{ch-nablahqh},
and~\eqref{estimate-term-qh} yields
\small\begin{equation} \label{int-chnablaqh-leq-unablaqh}
\begin{split}
\int_\Omega\cbfh\cdot \nablaboldh\qh \dxbf
&   \lesssim 
    \Big(\Norm{\wbf - \PiboldRTTauhkmo \wbf}_{\Lbf^{\afrak}(\E)}
    + \Norm{\dive\wbf - \PizTauhkmo \dive\wbf}_{\L^{((\afrakprime)^*)'}(\Omega)} \Big)
    \Norm{\nablaboldh \qh}_{\Lbf^{\afrakprime}(\Omega)}.
\end{split}
\end{equation}\normalsize
The assertion follows from
the inf-sup condition~\eqref{discrete-inf-sup}
and the fact that~$\cbfh$ is selected in~$\Dbfcalhperp$.
\end{proof}

We have strong convergence of several interpolation
and projection operators in certain norms.
\begin{lemma} \label{lemma:strong-convergence-Ph-Lalpha}
Let~$\wbf$ be in $\Lbf^{\afrak}(\Omega)$
with $\dive\wbf$ in $\L^{((\afrakprime)^*)'}(\Omega)$,
$\afrak$ larger than~$2$,
and $\qop$ in $\W^{1,s'}(\Omega)$,
$s'$ larger than~$1$.
Given $\PiboldzTauhkmo$ and~$\PizTauhkmo$,
$\PiboldRTTauhkmo$,
$\Pbfh$, and~$\Icalk$
as in~\eqref{L2-projector-mesh},
\eqref{RT-dofs-bulk}--\eqref{div-RT}--\eqref{trick-RT-indicator},
\eqref{def:Ph}, and~\eqref{interpolant-CR},
we have
\begin{align}
&   \lim_{\h\to 0} \PiboldzTauhkmo\wbf= \wbf,
    \qquad
    \lim_{\h\to 0} \PiboldRTTauhkmo\wbf= \wbf,
    \qquad
    \lim_{\h\to 0} \Pbfh(\wbf)= \wbf,
    & \text{ in } \Lbf^{\afrak}(\Omega)
    \label{eq:strong-convergence-Piboldkmo}\\
&   \lim_{\h\to 0} \PizTauhkmo \dive \wbf = \dive\wbf,
    & \text{ in } \Lbf^{((\afrakprime)^*)'}(\Omega),
    \label{eq:strong-convergence-Ph} \\
&   \lim_{\h\to 0} \nablaboldh \Icalk \qop= \nablabold \qop
    & \text{ in } \Lbf^{s'}(\Omega)
    \label{eq:strong-convergence-Icalk}. 
\end{align}
\end{lemma}
\begin{proof}
The validity of~\eqref{eq:strong-convergence-Piboldkmo}
and~\eqref{eq:strong-convergence-Icalk}
follows from a density argument
and the stability of~$\PiboldzTauhkmo$ and~$\Icalk$
in the correct norms.
The proof of~\eqref{eq:strong-convergence-Ph} follows
recalling that~$\Pbfh(\wbf)$ is
$\PiboldzTauhkmo(\wbf)+\cbfh(\wbf)$, and
applying~\eqref{eq:strong-convergence-Piboldkmo} to the first term,
and Theorem~\ref{theorem:estimate-ch}
and standard polynomial approximation properties
to the second term.
\end{proof}

\begin{remark} \label{remark:weak-convergences}
In light of~\eqref{apriori-ubfh} and~\eqref{apriori-ph},
$\ubfh$ and~$\nablaboldh\ph$ are uniformly bounded
in~$\Lbf^{\alpha}(\Omega)$ and~$\Lbf^{\alphaprime}(\Omega)$, respectively; hence,
we deduce the weak convergence
in~$\Lbf^{\alpha}(\Omega)$ of a subsequence of~$\ubfh$
and in~$\Lbf^{\alphaprime}(\Omega)$
of a subsequence of~$\nablaboldh\ph$.
With an abuse of notation,
we still denote these subsequences
by~$\ubfh$ and~$\nablaboldh\ph$.
\eremk
\end{remark}

We show the weak convergence in $\LalphaOmegad$
of $\ubfh$ to $\ubf$.

\begin{lemma} \label{lemma:weak-conv-flux-Lalpha}
Let $\ubf$ and $\ubfh$ for a mesh~$\Tauh$
be the solutions to~\eqref{splitted-variational-formulation}
and~\eqref{splitted-discrete-formulation}, respectively.
Then, we have
\begin{equation*}
    \lim_{\h\to 0} \ubfh= \ubf \quad \text{ weakly in } \LalphaOmegad.
\end{equation*}
\end{lemma}
\begin{proof}
Let $\ubftilde$ denote the weak limit in $\Lbf^\alpha(\Omega)$
of the subsequence~$\{\ubfh\}$
in Remark~\ref{remark:weak-convergences}.
The assertion follows if we prove~$\ubftilde = \ubf$.
Given $\ubfl$ as in~\eqref{ubfl-problem},
we consider the splittings
\begin{equation} \label{ubftilde-ubfh}
\ubftilde =(\ubftilde - \ubfl) + \ubfl
=: \ubfzerotilde + \ubfl.
\qquad\qquad
\ubfh =(\ubfh - \Pbfh(\ubfl)) + \Pbfh(\ubfl)
=: \ubfhzero + \ubfhl.
\end{equation}
We preliminarily observe
\begin{equation} \label{Acal-integral-leq-zero}
\begin{split}
0
&\overset{\eqref{monotonicity}}{\leq}
\int_\Omega \big( \Acal(\ubfhzero + \ubfhl)
                 -\Acal(\vbfh+\ubfhl)\big)
            \cdot (\ubfhzero - \vbfh) \dxbf \\
&\overset{\eqref{splitted-discrete-formulation}}{=}
\int_\Omega -\Acal(\vbfh+\ubfhl) \cdot (\ubfhzero - \vbfh) \dxbf
\qquad \forall \vbfh \in \Vbfcalhkmo.
\end{split}
\end{equation} \normalsize
Now,
let $\vbf$ be in $\Dbfcal$
and $\vbfh := \Pbfh(\vbf)$.
We have
\begin{equation*} 
\begin{split}
&\Norm{\Acal(\Pbfh(\vbf) + \ubfhl) - \Acal(\vbf + \ubfl)}_{\Lbf^{\alphaprime}(\Omega)}
\overset{\eqref{estimate-Acal}}{\leq}
    |\Omega|^{\frac{1}{\alpha}}
    \Big[ \frac{\mu}{\rho}\Norm{\Kbbmo}_{\Lbu^{\frac{\alpha}{\alphapt}}(\Omega)}
    + \CBL\frac{\beta}{\rho} \cdot \\
&   \qquad\qquad\cdot
    \big(\Norm{\Pbfh(\vbf)+\ubfhl}^\alphamt_{\Lbf^\alpha(\Omega)}
    + \Norm{\vbf + \ubfl}^\alphamt_{\Lbf^\alpha(\Omega)}\big)
    \Big] 
    \big(
    \Norm{ (\Pbfh(\vbf)-\vbf) + (\ubfhl-\ubfl)}_{\Lbf^\alpha(\Omega)}
    \big).
\end{split}
\end{equation*}
Using~\eqref{eq:strong-convergence-Ph}
and Remark~\ref{remark:weak-convergences} entails
\begin{equation} \label{strong-convergence-ph-ubfhl}
\lim_{\h\to 0} \Acal(\Pbfh(\vbf) + \ubfhl)
    = \Acal(\vbf + \ubfl)
    \quad \text{ in } \Lbf^{\alphaprime}(\Omega).
\end{equation}
The weak convergence of $\ubfhzero$ to $\ubfzerotilde$,
which is a consequence of the weak convergence
of $\ubfh$ to $\ubftilde$
and the strong convergence of $\ubfhl$ to $\ubfltilde$,
and the strong convergence of $\Pbfh(\vbf)$ to $\vbf$ in $\Lbf^\alpha(\Omega)$
entail
\begin{equation} \label{weak-convergence-uhz-Phv}
\lim_{\h\to 0}
(\ubfhzero - \Pbfh(\vbf))
= \ubfzerotilde - \vbf
\quad \text{ weakly in } \Lbf^{\alpha}(\Omega).
\end{equation}
Then,
\eqref{strong-convergence-ph-ubfhl} and \eqref{weak-convergence-uhz-Phv}
allow for passing to the limit in \eqref{Acal-integral-leq-zero}
and obtain
\begin{equation*}
\int_{\Omega}\Acal(\vbf + \ubfl) \cdot (\ubfzerotilde -\vbf) \dxbf \leq 0
\qquad\qquad
\forall \vbf \in \Dbfcal.
\end{equation*}
Given an arbitrary $\varphibold$ in $\Dbfcal$ and $\varepsilon >0$,
fixing $\vbf$ as
$\ubfzerotilde + \varepsilon\varphibold$
and as $\ubfzerotilde - \varepsilon\varphibold$,
the hemi-continuity~\eqref{hemi-continuity} yields
\begin{equation*}
\int_{\Omega}\Acal(\ubfzerotilde + \ubfl) \cdot \varphibold \dxbf = 0
\qquad\qquad
\forall \varphibold \in \Dbfcal.
\end{equation*}
The assertion follows from
the uniqueness of the solution
to~\eqref{splitted-variational-formulation}.
\end{proof}

We prove an auxiliary result,
which extends \cite[Eq.~(3.36)]{Girault-Wheeler:2008}
to the general-order case.

\begin{lemma} \label{lemma:false-Galerkin}
Let $(\ubf,\pop)$ and $(\ubfh,\ph)$
for a mesh~$\Tauh$ be the solutions
to~\eqref{variational-formulation}
and~\eqref{discrete-formulation},
and~$k>1$. Then, we have
\begin{equation} \label{false-Galerkin}
    \int_{\Omega}( \Acal(\ubf) - \Acal(\ubfh))\cdot \vbfh \dxbf
        = \int_{\Omega} (\pop-\Icalk \pop)
            (\diveh\vbfh-\PikmthreeE\diveh\vbfh)\dxbf
    \qquad \forall \vbfh \in \Dbfcalhkmo.
\end{equation}
If~$k=1$, then we recover \cite[Eq.~(3.36)]{Girault-Wheeler:2008},
i.e.,
\begin{equation} \label{true-Galerkin}
 \int_{\Omega}( \Acal(\ubf) - \Acal(\ubfh))\cdot \vbfh \dxbf
    = 0 \qquad\qquad \forall \vbfh \in \Dbfcalhzero.
\end{equation}
\end{lemma}
\begin{proof}
Given $\vbfh$ in $\Dbfcalhkmo$
and $\E$ in $\Tauh$,
we have
\begin{equation} \label{steps-galerkin}
\begin{split}
& \int_\E (\pop - \Icalk \pop) 
    (\dive\vbfh-\PikmthreeE\dive\vbfh)\dxbf
\overset{\eqref{dof-moment-bulk}}{=} 
    \int_{\E} (\pop - \Icalk \pop) \dive \vbfh \dxbf \\
&\overset{\eqref{dof-moment-edge}}{=}
    \int_{\partial \E} (\Icalk \pop - \pop)(\vbfh \cdot \nbfE) \ds
    -  \int_{\E} (\Icalk \pop - \pop) \dive \vbfh \dxbf 
= \int_\E\nablabold(\Icalk \pop - \pop)\cdot \vbfh \dxbf \\
&\overset{(\vbfh \in \Dbfcalhkmo)}{=} \int_\E\nablabold(\Icalk \pop - \pop)\cdot \vbfh \dxbf + \int_\E \nablabold(\ph -\Icalk \pop) \cdot \vbfh \dxbf
= \int_\E \nablabold (\ph - \pop) \cdot \vbfh \dxbf.
\end{split}
\end{equation}
The assertion follows from
summing over~$\E$ in~$\Tauh$,
subtracting~\eqref{variational-formulation}
to~\eqref{discrete-formulation},
and the above display.
\end{proof}

We now show the strong convergence in $\Lbf^{2}(\Omega)$
of $\ubfh$ to $\ubf$.
\begin{lemma} \label{lemma:strong-conv-flux-L2}
Let $\ubf$ and $\ubfh$ for a mesh~$\Tauh$
be the solutions
to~\eqref{splitted-variational-formulation}
and~\eqref{splitted-discrete-formulation}, respectively.
Then, we have
\begin{equation*}
    \lim_{\h\to 0} \ubfh= \ubf \quad \text{ in } \Lbf^{2}(\Omega).
\end{equation*}
\end{lemma}
\begin{proof}
Since~$\ubfzero$ belongs to~$\Dbfcal$,
we write
\begin{equation*}
0 \overset{\eqref{Dbfcal}, \eqref{def-cbfh}}{=}
    \int_\Omega \cbfh(\ubfzero) \cdot \nablaboldh \qh \dxbf
    \overset{\eqref{def:Ph}}{=}
    \int_\Omega \Pbfh(\ubfzero) \cdot \nablaboldh \qh \dxbf
    -\int_\Omega \PiboldzTauhkmo\ubfzero \cdot \nablaboldh\qh \dxbf
    \qquad \forall\qh\in \Qcalhzerok,
\end{equation*}
i.e., thanks to~\eqref{ubftilde-ubfh},
$\ubfhzero$ belongs to $\Dbfcalhkmo$.
Then, for $k>1$, we write
\begin{equation} \label{Acal-uhz-uhl-Phuhz-uhl}
\begin{split}
&   \frac{\mu}{\rho}\lambdamin \Norm{\ubfh - \Pbfh(\ubf)}^2_{\Lbf^2(\Omega)}
\overset{\eqref{ubftilde-ubfh}}{=}
\frac{\mu}{\rho}\lambdamin \Norm{\ubfhzero - \Pbfh(\ubfzero)}^2_{\Lbf^2(\Omega)} \\
&\qquad\overset{\eqref{monotonicity}}{\leq}
\int_{\Omega}(\Acal(\ubfhzero+\ubfhl) - \Acal(\Pbfh(\ubfzero) + \ubfhl))
\cdot(\ubfhzero - \Pbfh(\ubfzero)) \dxbf \\
&\qquad
    \overset{\eqref{ubftilde-ubfh}, \eqref{false-Galerkin}, \eqref{steps-galerkin}}{=}
\int_{\Omega}(\Acal(\ubfzero + \ubfl) - \Acal(\Pbfh(\ubfzero)+\ubfhl))
\cdot (\ubfhzero - \Pbfh(\ubfzero)) \dxbf \\
&\qquad\qquad\qquad\qquad
    + \int_{\Omega} \nablaboldh(\Icalk \pop - \pop)
    \cdot (\ubfhzero-\Pbfh(\ubfzero) )  \dxbf.
\end{split}
\end{equation}
Proceeding as in the proof
of~\eqref{strong-convergence-ph-ubfhl},
$\Acal(\Pbfh(\ubfzero) + \ubfhl)$ converges strongly
in $\Lbf^{\alphaprime}(\Omega)$
to $\Acal(\ubfzero + \ubfl)$;
$\ubfhzero$ converges weakly
in $\Lbf^{\alpha}(\Omega)$ to~$\ubfzero$,
cf. Lemma~\ref{lemma:weak-conv-flux-Lalpha};
$\Pbfh\ubfzero$ converges strongly
in $\Lbf^{\alpha}(\Omega)$ to~$\ubfzero$,
cf. \eqref{eq:strong-convergence-Piboldkmo}.
Therefore, we can pass to the limit
for $\h$ going to $0$ in the first
term on the right-hand side
of~\eqref{Acal-uhz-uhl-Phuhz-uhl} and get~$0$.
On the other hand,
$\Pbfh(\ubfzero)$ converges strongly
in $\Lbf^{\alpha}(\Omega)$ to~$\ubfzero$,
cf. \eqref{eq:strong-convergence-Piboldkmo};
$\nablaboldh\Icalk \pop$ converges strongly
in $\Lbf^{\alphaprime}(\Omega)$
to $\nablaboldh \pop$,
cf.~\eqref{eq:strong-convergence-Icalk}.
Therefore, we can pass to the limit
for~$\h$ going to~$0$ in the second
term on the right-hand side
of~\eqref{Acal-uhz-uhl-Phuhz-uhl} and get~$0$.
Therefore, we can pass to the limit
in~\eqref{Acal-uhz-uhl-Phuhz-uhl}
and the assertion follows for~$k$ larger than~$1$.

The case~$k=1$ is dealt with analogously
and no pressure term appears;
cf.~\eqref{true-Galerkin} and~\cite{Girault-Wheeler:2008}.
\end{proof}

Define~$\wbfh$ as
\begin{equation} \label{def:wbfh-local}
\wbfh := \nablaboldh \ph.
\end{equation}
We have that
\begin{itemize}
    \item $\wbfh$ is uniformly bounded in $\Lbf^{\alphaprime}(\Omega)$
    due to Proposition \ref{proposition:discrete-well-posedness-enriched},
    whence there exists $\wbf$ in $\Lbf^{\alphaprime}(\Omega)$ such that
    \begin{equation} \label{weak-convergence-whw-alphaprime}
    \lim_{\h\to 0}
    \wbfh
    = \wbf
    \quad \text{ weakly in }\Lbf^\alphaprime(\Omega);
    \end{equation}
    \item $\ph$ is uniformly bounded in $\Lbf^{(\alphaprime)^*}(\Omega)$
    due to Corollary \ref{corollary:Sobolev-inequalities-CR}
    and Proposition \ref{proposition:discrete-well-posedness-enriched},
    whence
    there exists~$\qop$ in $\L^{(\alphaprime)^*}(\Omega)$ such that
    \begin{equation} \label{weak-convergence-ph-alphaprimestar}
    \lim_{\h\to 0}
    \ph
    = \qop
    \quad \text{ weakly in }\L^{(\alphaprime)^*}(\Omega).
    \end{equation}
\end{itemize}
We show the weak convergence in $\Lbf^{\alphaprime}(\Omega)$
of the gradient of the potential variable.
\begin{proposition} \label{prop:weak-conv-pot-Lalphaprime}
Let $(\ubf,\pop)$ and $(\ubfh,\ph)$
for a mesh~$\Tauh$ be the solutions
to~\eqref{splitted-variational-formulation}
and~\eqref{splitted-discrete-formulation}, respectively,
and let the~\eqref{extra-ass-convergence} be valid.
Then, for~$\wbfh$ as in~\eqref{def:wbfh-local},
we have
\begin{equation*}
    \lim_{\h\to 0} \wbfh= \nablabold \pop \quad \text{ weakly in } \LalphaprimeOmegad.
\end{equation*}
\end{proposition}
\begin{proof}
We split the proof into two steps.
\paragraph*{Step 1: proving that the limit is the gradient
of a function~$\qop$.}
An integration by parts yields
\small\begin{equation} \label{duality-wh-varphi}
\begin{split}
\int_\Omega \wbfh \cdot \varphibold \dxbf
    \overset{\eqref{def:wbfh-local}}{:=}
    \sum_{\E \in \Tauh} \int_\E \nablabold \ph \cdot \varphibold \dxbf
= \int_{\FcalhI} \jump{\ph} (\varphibold\cdot \nbfF) \ds
    - \int_\Omega \ph \dive \varphibold \dxbf
    \qquad
    \forall \varphibold \in [\mathcal{C}_0^\infty(\Omega)]^2.
\end{split}
\end{equation}\normalsize
Standard polynomial approximation (PA) results yield
\begin{equation*}
\Big\vert \int_{\FcalhI}\jump{\ph}(\varphibold\cdot\nbfF)\ds \Big\vert
\overset{\text{(IBP)}, \eqref{def-cbfh}}{=}
\Big\vert \int_{\Omega}\cbfh(\varphibold)\cdot\nablaboldh\ph \dxbf \Big\vert
\overset{\text{(PA)}, \eqref{int-chnablaqh-leq-unablaqh}}{\lesssim}
\h \SemiNorm{\varphibold}_{\Wbf^{1,\alpha}(\Omega)}
    \Norm{\nablaboldh \ph}_{\Lbf^{\alphaprime}(\Omega)}.
\end{equation*}
Hence, the left-hand side of the above display
tends to zero when~$\h$ goes to zero,
which combined with~\eqref{weak-convergence-whw-alphaprime}
and~\eqref{weak-convergence-ph-alphaprimestar}
allows for passing to the limit in~\eqref{duality-wh-varphi}
and obtain
\begin{equation*}
\int_\Omega \wbf \cdot \varphibold  \dxbf
= - \int_\Omega \qop \dive \varphibold \dxbf,
\end{equation*}
i.e.,
$\wbf = \nablabold \qop$ and $\qop$ belongs to $\W^{1,\alphaprime}(\Omega)$.
\paragraph*{Step 2: proving that~$\qop$ equals~$\pop$.}
Given~$\vbf$ in~$\Lbf^{\frac{2\alpha}{4-\alpha}}(\Omega)$, 
taking~$\PiboldzTauhkmo\vbf$ as a test function
in~\eqref{discrete-formulation:first-eq} gives
\begin{equation} \label{eq:int-wbfh-varphi}
\begin{split}
\int_\Omega \wbfh \cdot \PiboldzTauhkmo\vbf \dxbf
& \overset{\eqref{def:wbfh-local}}{:=}
    \int_{\Tauh}\nablaboldh\ph\cdot\PiboldzTauhkmo\vbf\dxbf
\overset{\eqref{discrete-formulation:first-eq}}{=}
    -\int_\Omega \Acal(\ubfh)\cdot\PiboldzTauhkmo\vbf\dxbf \\
&= -\int_\Omega (\Acal(\ubfh)-\Acal(\ubf))\cdot\PiboldzTauhkmo\vbf\dxbf
+ \int_\Omega\Acal(\ubf)\cdot\PiboldzTauhkmo\vbf\dxbf \\
&\overset{\eqref{nablap-Acal}}{=} -\int_\Omega (\Acal(\ubfh)-\Acal(\ubf))\cdot\PiboldzTauhkmo\vbf\dxbf
+ \int_\Omega \nablabold \pop\cdot\PiboldzTauhkmo\vbf\dxbf.
\end{split}
\end{equation}
Using \eqref{estimate-Acal} in the above display
together with H\"older's inequality applied twice
with exponents
$(\alpha/(\alpha-2), 2, 2\alpha/(4-\alpha))$,
which is an admissible choice
since~$\alpha$ is in $(2,4)$ due to~\eqref{extra-ass-convergence},
give
\begin{equation} \label{eq:convergence-nablaph-Pikmo}
\begin{split}
&   \left| \int_\Omega (\Acal(\ubfh)-   
    \Acal(\ubf))\cdot\PiboldzTauhkmo\vbf\dxbf \right| 
    \leq \frac{\mu}{\rho}\Norm{\Kbbmo}_{\Lbu^{\frac{\alpha}{\alphamt}}(\Omega)}
    \Norm{\ubf-\ubfh}_{\Lbf^2(\Omega)}
    \Norm{\PiboldzTauhkmo\vbf}_{\Lbf^{\frac{2\alpha}{4-\alpha}}(\Omega)} \\
&   \qquad + \CBL\frac{\beta}{\rho}
    \Norm{\ubf-\ubfh}_{\Lbf^2(\Omega)}(\Norm{\ubfh}_{\Lbf^{\alpha}(\Omega)}^{\alphamt}
    +\Norm{\ubf}_{\Lbf^{\alpha}(\Omega)}^{\alphamt})
    \Norm{\PiboldzTauhkmo\vbf}_{\Lbf^{\frac{2\alpha}{4-\alpha}}(\Omega)}.
\end{split}
\end{equation}
The strong convergence in $\Lbf^2(\Omega)$ of $\ubfh$ to $\ubf$,
cf. Lemma~\ref{lemma:strong-conv-flux-L2},
and the strong convergence
in~$\Lbf^{\frac{2\alpha}{4-\alpha}}(\Omega)$
of~$\PiboldzTauhkmo\vbf$~to $\vbf$,
cf.~\eqref{eq:strong-convergence-Piboldkmo},
imply that the left-hand side of \eqref{eq:convergence-nablaph-Pikmo}
tends to zero when $\h$ goes to zero.
The assertion follows
passing to the limit in \eqref{eq:int-wbfh-varphi}
and using the uniqueness of the solution of~\eqref{splitted-variational-formulation}.
\end{proof}

\subsection{Strong convergence in the natural norms}
    \label{subsection:strong-convergence}
Next,
we show the strong convergence of $\ubfh$ in $\Lbf^{\alpha}(\Omega)$.
\begin{theorem} \label{theorem:strong-conv-flux-Lalpha}
Let $(\ubf,\p)$ and $(\ubfh,\ph)$
for a mesh~$\Tauh$
be the solutions to~\eqref{variational-formulation}
and~\eqref{discrete-formulation}, respectively,
and let~\eqref{extra-ass-convergence} be valid.
Assume that $\Kbb^{-1}$ is in $\Lbu^{\frac{2\alpha}{\alpha-2}}(\Omega)$.
Then, we have
\begin{equation*}
    \lim_{\h\to 0} \ubfh= \ubf \quad \text{ in } \LalphaOmegad.
\end{equation*}
\end{theorem}
\begin{proof}
We have
\begin{equation*}
\begin{split}
&    0 \overset{\eqref{discrete-formulation:first-eq}}{=} 
\frac{\mu}{\rho} \int_{\Omega}(\Kbbmo)\ubfh \cdot \ubfh \dxbf
        + \frac{\beta}{\rho} \int_{\Omega}\EucNorm{\ubfh}^{\alpha} \dxbf
        + \int_{\Omega}\nablaboldh \ph \cdot \ubfh \dxbf \\
&\quad\overset{(\ubfhzero \in \Dbfcalhkmo)}{=}
\frac{\mu}{\rho} \int_{\Omega}(\Kbbmo)\ubfh \cdot \ubfh \dxbf
        + \frac{\beta}{\rho} \int_{\Omega}\EucNorm{\ubfh}^{\alpha} \dxbf
        + \int_{\Omega}\nablaboldh \ph \cdot \ubfhl \dxbf.
\end{split}
\end{equation*}
Using
the fact that $\Kbb^{-1}$ is in
$\Lbu^{\frac{2\alpha}{\alpha-2}}(\Omega)$,
the strong convergence
of $\ubfh$ to~$\ubf$ in $\Lbf^2(\Omega)$,
see Lemma~\ref{lemma:strong-conv-flux-L2},
and the weak convergence
of $\ubfh$ to~$\ubf$ in $\Lbf^\alpha(\Omega)$,
see Lemma~\ref{lemma:weak-conv-flux-Lalpha},
for the first term on the left-hand side,
the weak convergence of $\nablaboldh \ph$
to $\nabla \pop$ in $\Lbf^{\alphaprime}(\Omega)$,
see Proposition~\ref{prop:weak-conv-pot-Lalphaprime},
and the strong convergence of $\ubfhl$
to $\ubfl$
in $\Lbf^\alpha(\Omega)$,
see Lemma~\ref{lemma:strong-convergence-Ph-Lalpha},
for the third term on the left-hand side,
and passing to the limit~$\h$ going to~$0$
entail
\begin{equation} \label{first-equation:with-limit}
\begin{split}
&\frac{\mu}{\rho} \int_{\Omega}(\Kbbmo)\ubf \cdot \ubf \dxbf
        + \lim_{h\to0}\frac{\beta}{\rho} \int_{\Omega}\EucNorm{\ubfh}^{\alpha} \dxbf
        + \int_{\Omega}\nablabold \pop \cdot \ubf \dxbf
= 0.
\end{split}
\end{equation}
Taking $\ubf = \ubfzero + \ubfl$ as test function $\vbf$
in the continuous formulation \eqref{variational-formulation:first-eq}
gives
\begin{equation} \label{first-equation:without-limit}
\frac{\mu}{\rho} \int_{\Omega}(\Kbbmo) \ubf \cdot \ubf \dxbf
        + \frac{\beta}{\rho} \int_{\Omega}\EucNorm{\ubf}^{\alpha} \dxbf
        + \int_{\Omega}\nablabold \pop \cdot \ubf \dxbf
        = 0
\end{equation}
Taking the difference of~\eqref{first-equation:with-limit}
and~\eqref{first-equation:without-limit} yields
\begin{equation*}
\lim_{h\to0}\Norm{\ubfh}_{\Lbf^\alpha(\Omega)}^\alpha
= \Norm{\ubf}_{\Lbf^\alpha(\Omega)}^\alpha.
\end{equation*}
\end{proof}

Next,
we show the strong convergence of $\nablaboldh \ph$ in $\Lbf^{\alphaprime}(\Omega)$.
\begin{theorem} \label{theorem:strong-conv-pot-Lalphaprime}
Let $(\ubf,\pop)$ and $(\ubfh,\ph)$
for a mesh~$\Tauh$
be the solutions to~\eqref{variational-formulation}
and~\eqref{discrete-formulation}, respectively,
and let~\eqref{extra-ass-convergence} be valid.
Assume that $\Kbb^{-1}$ is in
$\Lbu^{\frac{2\alpha}{\alpha-2}}(\Omega)$.
Then, we have
\begin{equation*}
    \lim_{\h\to 0} \nablaboldh \ph= \nablabold \pop \quad \text{ in } \LalphaprimeOmegad.
\end{equation*}
\end{theorem}
\begin{proof}
Subtracting~\eqref{variational-formulation:first-eq}
and~\eqref{discrete-formulation:first-eq},
we have
\begin{equation*}
\begin{split}
\int_\Omega (\Acal(\ubf) - \Acal(\ubfh)) \cdot \vbfh \dxbf
&= \int_{\Omega} \nablaboldh(\ph - \pop) \cdot \vbfh \dxbf \\
&= \int_{\Omega} \nablaboldh(\ph - \Icalk \pop) \cdot \vbfh \dxbf
+ \int_{\Omega} \nablaboldh(\Icalk \pop - \pop) \cdot \vbfh \dxbf
\qquad \forall \vbfh \in \Vbfcalhkmo.
\end{split}
\end{equation*} \normalsize
Using \eqref{estimate-Acal} in the above display
together with
triangle's inequality
and H\"older's inequality
with exponents $(\alpha/(\alpha-2),\alpha,\alpha)$
give
\small
\begin{equation} \label{ph-Icalkph-tilde}
\begin{split}
\left| \int_\Omega \nablaboldh(\ph - \Icalk \pop) \cdot \vbfh \dxbf \right|
    &\leq \left| \int_\Omega \nablaboldh(\pop - \Icalk \pop) \cdot \vbfh \dxbf \right|
            + \left|\int_\Omega (\Acal(\ubf) - \Acal(\ubfh))\cdot \vbfh \dxbf \right|  \\
&   \leq \left| \int_\Omega \nablaboldh(\pop - \Icalk \pop) 
        \cdot \vbfh \dxbf \right|
        + \frac{\mu}{\rho}\Norm{\Kbbmo}_{\Lbu^{\frac{\alpha}{\alpha-2}}(\Omega)}
        \Norm{\ubf-\ubfh}_{\Lbf^\alpha(\Omega)}
        \Norm{\vbfh}_{\Lbf^\alpha(\Omega)} \\
& \qquad + \CBL\frac{\beta}{\rho}
        \Norm{\ubf-\ubfh}_{\Lbf^\alpha(\Omega)}
        (\Norm{\ubfh}_{\Lbf^{\alpha}(\Omega)}^{\alphamt}
        +\Norm{\ubf}_{\Lbf^{\alpha}(\Omega)}^{\alphamt})
        \Norm{\vbfh}_{\Lbf^\alpha(\Omega)}.
\end{split}
\end{equation}
\normalsize
We deduce
\begin{equation*} \label{ph-Icalkph-tilde-norm}
\begin{split}
&   \bfrakh \Norm{\nablaboldh(\ph - \Icalk \pop)}_{\Lbf^\alphaprime(\Omega)}
    \overset{\eqref{discrete-inf-sup}}{\le}
    \sup_{\vbfh \in \Vbfcalhkmo} 
    \frac{\int_\Omega \vbfh \cdot \nablaboldh(\ph - \Icalk \pop) \dxbf}
    {\Norm{\vbfh}_{\Lbf^\alpha(\Omega)}} \\
&\overset{\eqref{ph-Icalkph-tilde},\Vbfcalhkmo \subset \Vbfcal}{\leq}
\sup_{\vbf \in \Vbfcal}
    \frac{\int_\Omega \vbf \cdot \nablaboldh(\pop - \Icalk \pop) \dxbf }
    {\Norm{\vbf}_{\Lbf^\alpha(\Omega)}}
    + \frac{\mu}{\rho}\Norm{\Kbbmo}_{\Lbu^{\frac{\alpha}{\alphamt}}(\Omega)}\Norm{\ubf-\ubfh}_{\Lbf^\alpha(\Omega)} \\
&\qquad\qquad
    + \CBL \frac{\beta}{\rho}
    \Norm{\ubf-\ubfh}_{\Lbf^\alpha(\Omega)}
    (\Norm{\ubfh}_{\Lbf^{\alpha}(\Omega)}^{\alphamt}
    +\Norm{\ubf}_{\Lbf^{\alpha}(\Omega)}^{\alphamt}).
\end{split}
\end{equation*}
\normalsize
Triangle's inequality
and the above estimate yield
\begin{equation*}
\begin{split}
\bfrakh \Norm{\nablaboldh(\pop-\ph)}_{\Lbf^\alphaprime(\Omega)}
&   \leq  \bfrakh
    \Norm{\nablaboldh(\pop - \Icalk \pop)}_{\Lbf^\alphaprime(\Omega)}
    + \bfrakh\Norm{\nablaboldh(\ph -\Icalk \ph)}_{\Lbf^\alphaprime(\Omega)} \\
&   \le (1+\bfrakh)
    \Norm{\nablaboldh(\pop - \Icalk\pop)}_{\Lbf^\alphaprime(\Omega)}
    + \frac{\mu}{\rho}\Norm{\Kbbmo}_{\Lbu^{\frac{\alpha}{\alphamt}}(\Omega)}
    \Norm{\ubf-\ubfh}_{\Lbf^\alpha(\Omega)} \\
&   \qquad 
    + \CBL \frac{\beta}{\rho} \Norm{\ubf-\ubfh}_{\Lbf^\alpha(\Omega)}
    (\Norm{\ubfh}_{\Lbf^{\alpha}(\Omega)}^{\alphamt}
    +\Norm{\ubf}_{\Lbf^{\alpha}(\Omega)}^{\alphamt}).
\end{split}
\end{equation*}
The assertion follows from combining the above estimate
with the strong convergences in~$\Lbf^\alpha(\Omega)$ of~$\ubfh$ to $\ubf$
(here we use that $\Kbb^{-1}$ is in
$\Lbu^{\frac{2\alpha}{\alpha-2}}(\Omega)$),
cf. Theorem~\ref{theorem:strong-conv-flux-Lalpha},
and in~$\Lbf^\alphaprime(\Omega)$
of~$\nablaboldh \Icalk \pop$ to~$\nablabold \pop$,
cf.~\eqref{eq:strong-convergence-Icalk}.
\end{proof}

Due to the uniqueness of the solution $(\ubf,\pop)$,
we deduce the convergence of the whole sequence.

\section{Error estimates} \label{section:error-estimates}
The goal of this section is to establish
general-order error estimates for method~\eqref{discrete-formulation}.

Given an element $\E$ in $\Tauh$,
a polynomial inverse estimate holds true:
there exists a positive constant $\CinvK$
depending only on $k$ and $\gamma$,
such that
\begin{equation} \label{inverse-estimate:divergence}
\Norm{\dive \qbfk}_{\L^{s}(\E)} 
\le \CinvK \hE^{-1} \Norm{\qbfk}_{\L^{s}(\E)}
\qquad \forall \qbfk \in \Pbf_k(\E), \, \forall \E\in\Tauh.
\end{equation} 
We show an upper bound for the error in the $\Ltwo$-norm
of the flux variable.
\begin{proposition} \label{prop:error-estimate-velocity}
Let $(\ubf,\pop)$ and $(\ubfh, \ph)$
for a mesh~$\Tauh$
be solutions to~\eqref{variational-formulation} 
and~\eqref{discrete-formulation}, respectively.
Assume that
\begin{equation*}
\begin{cases}
\ubf \in \Lbf^{4\alpha-8}(\Omega),
    \ \Kbbmo \in \Lbu^\infty(\Omega) 
    &\text{if } \alpha > 2 \\
\ubf \in \Lbf^{\frac{2\alpha}{4-\alpha}}(\Omega), \ 
    \Kbbmo \in \Lbu^{\frac{\alpha}{\alpha-2}}(\Omega)
    &\text{only if } \alpha\in(2,4),
\end{cases}
\qquad\qquad
\pop\in\H^1(\Omega)\cap\QcalgD.
\end{equation*}
Then, for all $\wbfh$ in $\Vbfcalhkmo$,
we have
\begin{equation} \label{ltwo-estimate-flux}
\begin{split}
    \Norm{\ubf-\ubfh}_{\Lbf^2(\Omega)}
        &\leq \mathfrak E(\Kbb^{-1}, \ubf,\wbfh,\alpha) \\
       & \quad +\frac{\beta}{\mu\lambdamin}(\alpha-1)\Norm{\ubf-\wbfh}_{\Lbf^4(\Omega)}
                (\Norm{\ubf}^\alphamt_{\Lbf^{4\alpha-8}(\Omega)} + \Norm{\wbfh}^\alphamt_{\Lbf^{4\alpha-8}(\Omega)}) \\
            &\quad+ \frac{\mu}{\rho\lambdamin} \CSP \max_{\E \in \Tauh}
            \CinvK \Norm{\nablaboldh(\pop-\Icalk \pop)}_{\Lbf^2(\Omega)},
\end{split}
\end{equation}
where
\begin{equation*}
\mathfrak E(\Kbb^{-1}, \ubf,\wbfh,\alpha) = 
\begin{cases}
\Big(1 + \frac{\Norm{\Kbbmo}_{\Lbf^\infty(\Omega)}}{\lambdamin}\Big) 
                \Norm{\ubf-\wbfh}_{\Lbf^2(\Omega)}
    & \text{if } \alpha > 2 \\
\Big(1 + \frac{\Norm{\Kbbmo}_{\Lbf^{\frac{\alpha}{\alpha-2}}(\Omega)}}{\lambdamin}\Big) 
        \Norm{\ubf-\wbfh}_{\Lbf^{\frac{2\alpha}{4-\alpha}}(\Omega)}
    & \text{only if } \alpha\in(2,4).
\end{cases}
\end{equation*}
The last term on the right-hand side
of~\eqref{ltwo-estimate-flux} drops if~$k=1$;
cf. \cite[Theorem~8]{Girault-Wheeler:2008} and~\eqref{true-Galerkin}.
\end{proposition}
\begin{proof}
We prove the assertion for
$\Kbb^{-1}$ in $\Lbu^\infty(\Omega)$
as the other case follows analogously.

For a given $\wbfh$ in $\DGkmo$,
triangle's inequality implies
\begin{equation} \label{uwh+uhwh}
     \Norm{\ubf-\ubfh}_{\Lbf^2(\Omega)}
        \leq \Norm{\ubf-\wbfh}_{\Lbf^2(\Omega)}
            + \Norm{\ubfh-\wbfh}_{\Lbf^2(\Omega)}.
\end{equation}
Choosing $\ubf$ as $\ubfl$ in \eqref{monotonicity}
and $\ubfh - \wbfh$ as $\vbfh$ in \eqref{false-Galerkin},
applying inequality \eqref{estimate-Acal},
Cauchy-Schwarz' inequality
(if $\alpha$ is in $(2,4)$,
then we use H\"older's inequality with
exponents $(\alpha/(\alpha-2),2\alpha/(4-\alpha),2)$),
and H\"older's inequality with exponents $(4,4,2)$, we have
\small
\begin{equation} \label{lambda-error}
\begin{split}
&   \frac{\mu}{\rho}\lambdamin \Norm{\ubfh - \wbfh}^2_{\Lbf^2(\Omega)}
    \overset{\eqref{monotonicity}}{\leq}
        \int_{\Omega}(\Acal(\ubfh) - \Acal(\wbfh)) \cdot (\ubfh - \wbfh) \dxbf \\
& \overset{\eqref{ubftilde-ubfh}, \eqref{false-Galerkin}, \eqref{steps-galerkin}}{=}
    \int_{\Omega}(\Acal(\ubf) - \Acal(\wbfh)) \cdot (\ubfh - \wbfh) \dxbf \\
&   \qquad\qquad\qquad
    +\sum_{\E \in \Tauh} \int_{\E} (\Icalk \pop - \pop) 
    (\dive (\ubfh-\wbfh) - \PikmthreeE \dive (\ubfh-\wbfh)) \dxbf. \\
&   \overset{\eqref{estimate-Acal}}{\leq} 
    \frac{\mu}{\rho}\Norm{\Kbbmo}_{\Lbf^\infty(\Omega)}
    \Norm{\ubf-\wbfh}_{\Lbf^2(\Omega)}
    \Norm{\ubfh-\wbfh}_{\Lbf^2(\Omega)} \\
&\qquad
    + \CBL \frac{\beta}{\rho}
    \Norm{\ubf-\wbfh}_{\Lbf^4(\Omega)}
    \Big(\Norm{\ubf}^\alphamt_{\Lbf^{4\alpha-8}(\Omega)}
        +\Norm{\wbfh}^\alphamt_{\Lbf^{4\alpha-8}(\Omega)}\Big)
        \Norm{\ubfh-\wbfh}_{\Lbf^2(\Omega)} \\
&   \qquad
    + \sum_{\E \in \Tauh}\Norm{\pop-\Icalk \pop}_{\L^{2}(\E)}
        \Norm{\dive(\ubfh - \wbfh)}_{\Lbf^2(\E)}.
\end{split}
\end{equation}
\normalsize
Using that $\pop-\Icalk \pop $
has zero average over $\partial\E$ for all $\E$ in $\Tauh$,
the polynomial inverse estimate \eqref{inverse-estimate:divergence},
Cauchy-Schwarz' inequality,
and Poincar\'e-Steklov's
inequality~\eqref{Poincare-Steklov}
imply
\begin{equation*} 
\begin{split}
    \sum_{\E \in \Tauh} \Norm{\pop-\Icalk \pop}_{\L^{2}(\E)}
                   & \Norm{\dive(\ubfh - \vbfh)}_{\Lbf^2(\E)} 
        \overset{\eqref{inverse-estimate:divergence}}{\leq}
            \sum_{\E \in \Tauh} \Norm{\pop-\Icalk \pop}_{\L^2(\E)}
             \CinvK \hK^{-1} \Norm{\ubfh- \vbfh}_{\Lbf^2(\E)} \\
        &\leq \Big((\max_{\E \in \Tauh} \CinvK)^2\sum_{\E \in \Tauh}\hK^{-2} \Norm{\pop-\Icalk \pop}_{\L^2(\E)}^2\Big)^{\frac12}
            \Big(\sum_{\E \in \Tauh} \Norm{\ubfh- \vbfh}_{\Lbf^2(\E)}^2\Big)^{\frac12} \\
        &\leq \CSP\max_{\E \in \Tauh} \CinvK\ 
            \Norm{\nablaboldh(\pop-\Icalk \pop)}_{\Lbf^2(\Omega)}
            \Norm{\ubfh - \vbfh}_{\Lbf^2(\Omega)}.
\end{split}
\end{equation*}
Combining the display above
and~\eqref{lambda-error} yields
\begin{equation*}
\begin{split}  
\Norm{\ubfh - \wbfh}_{\Lbf^2(\Omega)}
&\leq \frac{\Norm{\Kbbmo}_{\Lbf^\infty(\Omega)}}{\lambdamin}
        \Norm{\ubf-\wbfh}_{\Lbf^2(\Omega)}
        + \frac{\mu}{\rho\lambdamin} \max_{\E \in \Tauh} \CP
    \CinvK \Norm{\nablaboldh(\pop-\Icalk \pop)}_{\Lbf^2(\Omega)}\\
&   \quad
    + \CBL\frac{\beta}{\lambdamin}
        \Norm{\ubf-\wbfh}_{\Lbf^4(\Omega)}
        \Big(\Norm{\ubf}^\alphamt_{\Lbf^{4\alpha-8}(\Omega)}
        +\Norm{\wbfh}^\alphamt_{\Lbf^{4\alpha-8}(\Omega)}\Big) .
\end{split}
\end{equation*}
Inserting the above display in \eqref{uwh+uhwh},
the assertion follows.
\end{proof}

We are now in the position to derive
error estimates.
\begin{theorem}
Let $(\ubf,\pop)$ and $(\ubfh,\ph)$ be the solutions
to~\eqref{variational-formulation}
and~\eqref{discrete-formulation}, respectively.
Given~$s$ in~$\mathbb N$, assume that
\begin{equation*}
\begin{cases}
\ubf \in \Wbf^{s,4}(\Omega), \ 
\dive \ubf \in \W^{s,\frac43}(\Omega), \
\Kbbmo\in\Lbu^\infty(\Omega)
    &\text{if } \alpha>2 \\
\ubf \in \Wbf^{s,\frac{2\alpha}{4-\alpha}}(\Omega), \
\dive \ubf \in \W^{s,\frac\alpha2}(\Omega), \
\Kbbmo\in\Lbu^{\frac{\alpha}{\alpha-2}}(\Omega)
    &\text{only if } \alpha\in(2,4),
\end{cases}
\qquad
\pop \in \W^{s+1,2}(\Omega).
\end{equation*}
Denoting $\sfrak := \min\{k,s\}$,
for hidden constants only depending on~$k$,
$\GammaD$, $\Omega$ through either~$\Rfrak$ or~$\Nfrak$
in~\eqref{Omega-ass-1} and~\eqref{Omega-ass-2},
$\gamma$, $\alpha$, $\beta$, $\rho$,
$\mu$, and~$\Kbbmo$,
the following estimates hold true:
\small\begin{equation} \label{error-estimate:flux-old}
\begin{cases}
\Norm{\ubf - \ubfh}_{\Lbf^2(\Omega)}
\lesssim \h^{\sfrak}
    \big( \Norm{\ubf}_{\Wbf^{\sfrak,4}(\Omega)}
    + \Norm{\dive \ubf}_{\W^{\sfrak,\frac43}(\Omega)}
    + \Norm{\pop}_{\W^{\sfrak+1,2}(\Omega)} \big)     
        &\text{if } \alpha > 2\\
\Norm{\ubf - \ubfh}_{\Lbf^2(\Omega)}
\lesssim \h^{\sfrak}
    \big(\Norm{\ubf}_{\Wbf^{\sfrak,\frac{2\alpha}{4-\alpha}}(\Omega)}
    + \Norm{\dive \ubf}_{\W^{\sfrak,\frac\alpha2}(\Omega)}
    + \Norm{\pop}_{\W^{\sfrak+1,2}(\Omega)}  \big)
        &\text{only if } \alpha \in (2,4)
\end{cases}
\end{equation}\normalsize
and
\small\begin{equation} \label{error-estimate:potential-old}
\begin{cases}
\Norm{\nablaboldh(\pop - \ph)}_{\Lbf^{\alphaprime}(\Omega)}
\lesssim \h^{\sfrak}
    \big(\Norm{\ubf}_{\Wbf^{\sfrak,4}(\Omega)}
    + \Norm{\dive \ubf}_{\W^{\sfrak,\frac43}(\Omega)}
    + \Norm{\pop}_{\W^{\sfrak+1,2}(\Omega)} \big)       
        &\text{if } \alpha > 2\\
\Norm{\nablaboldh(\pop - \ph)}_{\Lbf^{\alphaprime}(\Omega)}
\lesssim \h^{\sfrak}
    \big(\Norm{\ubf}_{\Wbf^{\sfrak,\frac{2\alpha}{4-\alpha}}(\Omega)}
    + \Norm{\dive \ubf}_{\W^{\sfrak,
    \frac\alpha2}(\Omega)}
    + \Norm{\pop}_{\W^{\sfrak+1,2}(\Omega)}  \big)
    &\text{only if } \alpha \in (2,4).
\end{cases}
\end{equation}\normalsize
\end{theorem}
\begin{proof}
We prove the assertion for
$\Kbb^{-1}$ in $\Lbu^\infty(\Omega)$;
the other case follows analogously
as in Proposition~\ref{prop:error-estimate-velocity}
using different Sobolev embeddings
and noting that
$\alpha/2 = (((2\alpha/(4-\alpha))')^*)'$.

\paragraph*{Flux error estimates.}
Given $\cbfh=\cbfh(\ubf)$ solution to~\eqref{def-cbfh}
with~$\ubf$ as in~\eqref{variational-formulation},
we pick~$\wbfh$ equal to
$\PiboldzTauhkmo(\ubf)+\cbfh(\ubf)$
in~\eqref{ltwo-estimate-flux} and get
\[
\begin{split}
&   \Norm{\ubf-\ubfh}_{\Lbf^2(\Omega)}
\lesssim \Big(\Norm{\ubf-\PiboldzTauhkmo \ubf}_{\Lbf^2(\Omega)} 
                + \Norm{\cbfh(\ubf)}_{\Lbf^2(\Omega)}\Big) \\
&   \quad + \Big(\Norm{\ubf-\PiboldzTauhkmo \ubf}_{\Lbf^4(\Omega)}
                    + \Norm{\cbfh(\ubf)}_{\Lbf^4(\Omega)} \Big)
           \Big(\Norm{\ubf}^\alphamt_{\Lbf^{4\alpha-8}(\Omega)} 
        + \Norm{\PiboldzTauhkmo(\ubf)+\cbfh(\ubf)}
            ^\alphamt_{\Lbf^{4\alpha-8}(\Omega)}\Big) \\
&   \quad+ \Norm{\nablaboldh(\pop-\Icalk \pop)}_{\Lbf^2(\Omega)}
    =: T_1 + T_2 \cdot T_3 + T_4.
\end{split}
\]
Standard polynomial approximation properties
and~\eqref{cbfh-bound} with $\afrak=2$, $4$,
and~$4\alpha-8$ imply
\[
\begin{split}
T_1 
& \lesssim \h^{\sfrak} 
    ( \Norm{\ubf}_{\Wbf^{\sfrak,2}(\Omega)}
    + \Norm{\dive \ubf}_{\Wbf^{\sfrak,1}(\Omega)} ),
    \qquad 
    T_2  \lesssim \h^{\sfrak} 
    ( \Norm{\ubf}_{\Wbf^{\sfrak,4}(\Omega)}
    + \Norm{\dive \ubf}_{\Wbf^{\sfrak,\frac43}(\Omega)} ) ,\\
T_3
& \lesssim \Norm{\ubf}^\alphamt_{\Lbf^{4\alpha-8}(\Omega)} 
    + \Norm{\dive\ubf}^\alphamt
    _{\Lbf^{(((4\alpha-8)')^*)'}(\Omega)} .
\end{split}
\]
Using the interpolation estimates~\eqref{cr-interpolation:estimates},
we further obtain
\[
T_4
\lesssim \h^{\sfrak} \Norm{\pop}_{\W^{\sfrak+1,2}(\Omega)} .
\]
Estimate~\eqref{error-estimate:flux-old} follows combining the above bounds.

\paragraph*{Potential error estimates.}
Subtracting \eqref{discrete-formulation:first-eq} from \eqref{variational-formulation:first-eq},
we have
\small
\begin{equation*}
\begin{split}
-\int_\Omega (\Acal(\ubfh) - \Acal(\ubf))\cdot \vbfh \dxbf 
&= \int_\Omega \nablaboldh(\ph - \pop) \cdot \vbfh \dxbf \\
&=  \int_\Omega \nablaboldh(\ph - \Icalk \pop) \cdot \vbfh \dxbf
+  \int_\Omega \nablaboldh(\Icalk \pop - \pop) \cdot \vbfh \dxbf
\qquad \forall \vbfh \in \Vbfcalhkmo.
\end{split}
\end{equation*}
\normalsize
Using~\eqref{estimate-Acal} in the above display
together with 
triangle's inequality,
Cauchy-Schwarz' inequality
(if $\Kbb^{-1}$ is in $\Lbu^{\frac{\alpha}{\alpha-2}}(\Omega)$
we rather use H\"older's inequality with
exponents $(2\alpha/(\alpha-2),2,\alpha)$),
and H\"older's inequality
with exponents $(2,2\alpha/(\alpha-2),\alpha)$
give
\begin{equation} \label{ph-Icalkph-old}
\begin{split}
& \left| \int_\Omega \vbfh \cdot \nablaboldh(\ph - \Icalk \pop) \dxbf \right|
    \leq \left| \int_\Omega  \vbfh \cdot \nablaboldh(\pop - \Icalk \pop) \dxbf \right|
            + \left|\int_\Omega (\Acal(\ubfh) - \Acal(\ubf))\cdot \vbfh \dxbf \right|  \\
& \leq \left| \int_\Omega \vbfh \cdot \nablaboldh(\pop - \Icalk \pop)  \dxbf \right|
        + \frac{\mu}{\rho}\Norm{\Kbbmo}_{\Lbu^\infty(\Omega)}
        \Norm{\ubf-\ubfh}_{\Lbf^2(\Omega)}
        \Norm{\vbfh}_{\Lbf^2(\Omega)} \\
& \qquad 
    + \CBL \frac{\beta}{\rho} \Norm{\ubf-\ubfh}_{\Lbf^2(\Omega)}
    (\Norm{\ubfh}_{\Lbf^{2\alpha}(\Omega)}^{\alphamt}
    +\Norm{\ubf}_{\Lbf^{2\alpha}(\Omega)}^{\alphamt})
    \Norm{\vbfh}_{\Lbf^\alpha(\Omega)}.
\end{split}
\end{equation}
We deduce
\begin{equation*}
\begin{split}
\bfrakh \Norm{\nablaboldh(\ph - \Icalk \pop)}_{\Lbf^\alphaprime(\Omega)}
&    \overset{\eqref{discrete-inf-sup}}{\le}
    \sup_{\vbfh \in \Vbfcalhkmo} 
    \frac{\int_\Omega \vbfh\cdot \nablaboldh(\ph - \Icalk \pop) \dxbf}
    {\Norm{\vbfh}_{\Lbf^\alpha(\Omega)}} \\
&   \overset{\eqref{ph-Icalkph-old},\Vbfcalhkmo \subset \Vbfcal}{\leq}
    \sup_{\vbf \in \Vbfcal}
    \frac{\int_\Omega \vbf \cdot \nablabold(\pop - \Icalk \pop) \dxbf }
    {\Norm{\vbf}_{\Lbf^\alpha(\Omega)}}
    + \frac{\mu}{\rho}\Norm{\Kbbmo}_{\Lbu^{\infty}(\Omega)}
    \Norm{\ubf-\ubfh}_{\Lbf^2(\Omega)} \\
&   \qquad\qquad
    + \CBL\frac{\beta}{\rho}
    \Norm{\ubf-\ubfh}_{\Lbf^2(\Omega)}
    (\Norm{\ubfh}_{\Lbf^{2\alpha}(\Omega)}^{\alphamt}
    +\Norm{\ubf}_{\Lbf^{2\alpha}(\Omega)}^{\alphamt}).
\end{split}
\end{equation*}
Triangle's inequality,
H\"older's inequality,
and the above estimate yield
\begin{equation} \label{pressure-Lalphaprime-old}
\begin{split}
\bfrakh \Norm{\nablaboldh(\pop-\ph)}_{\Lbf^\alphaprime(\Omega)}
&   \leq \bfrakh
    \Norm{\nablaboldh(\pop - \Icalk \pop)}_{\Lbf^\alphaprime(\Omega)}
    + \bfrakh \Norm{\nablaboldh(\ph -\Icalk \pop)}_{\Lbf^\alphaprime(\Omega)} \\
&   \le (1+\bfrakh)
    \Norm{\nablaboldh(\pop - \Icalk\pop)}_{\Lbf^\alphaprime(\Omega)}
    + \frac{\mu}{\rho} \Norm{\Kbbmo}_{\Lbu^\infty(\Omega)}
    \Norm{\ubf-\ubfh}_{\Lbf^2(\Omega)} \\
&\qquad + \CBL \frac{\beta}{\rho}
    \Norm{\ubf-\ubfh}_{\Lbf^2(\Omega)}
    (\Norm{\ubfh}_{\Lbf^{2\alpha}(\Omega)}^{\alphamt}
    +\Norm{\ubf}_{\Lbf^{2\alpha}(\Omega)}^{\alphamt}).
\end{split}
\end{equation}
We now show that~$\ubfh$ 
is bounded in~$\Lbf^{2\alpha}(\Omega)$
for fixed order~$k$.
Triangle's inequality,
an $\Lbf^{2\alpha}-\Lbf^2$ polynomial inverse inequality,
and the stability of the projection imply
\begin{equation*}
\begin{split}
\Norm{\ubfh}_{\Lbf^{2\alpha}(\Omega)}
\leq \Norm{\ubfh - \PiboldzTauhkmo\ubf}_{\Lbf^{2\alpha}(\Omega)}
    + \Norm{\PiboldzTauhkmo\ubf}_{\Lbf^{2\alpha}(\Omega)}
\lesssim \h^{\frac{\alpha-1}{\alpha}}\Norm{\ubf-\ubfh}_{\Lbf^2(\Omega)}
        + \Norm{\ubf}_{\Lbf^{2\alpha}(\Omega)}.
\end{split}
\end{equation*}
Since the Sobolev embedding $\Wbf^{1,4}(\Omega)$
in $\Lbf^{2\alpha}(\Omega)$ holds true,
estimate~\eqref{error-estimate:potential-old} gives
\begin{equation} \label{bound-L2alpha-norm-old}
\Norm{\ubf}_{\Lbf^{2\alpha}(\Omega)}
\lesssim \Norm{\ubf}_{\Wbf^{1,4}(\Omega)}
\qquad\qquad\Longrightarrow\qquad\qquad
\ubfh \in \Lbf^{2\alpha}(\Omega) .
\end{equation}
The assertion follows inserting \eqref{cr-interpolation:estimates},
\eqref{ltwo-estimate-flux},
and~\eqref{bound-L2alpha-norm-old} in~\eqref{pressure-Lalphaprime-old},
and noting that~$\alphaprime$ is smaller than~$2$.
\end{proof}

\begin{remark} \label{remark:galerkin-lowest-order}
In  the lowest-order counterpart of Theorem \ref{prop:error-estimate-velocity};
see \cite[Prop. 8]{Girault-Wheeler:2008},
the error estimate on the fluxes is decoupled
from the potential. Indeed,
for $k=1$ the right-hand side of \eqref{false-Galerkin} vanishes
and a Galerkin orthogonality property for the mapping $\Acal$ holds true;
thus,
in \eqref{lambda-error} the term involving the pressure does not longer appear.
This fact is confirmed on
the numerical level as well;
see Figure~\ref{fig:constant-cubic} below.
\eremk
\end{remark}

\section{Numerical experiments} \label{section:numerical-experiments}
Here, we assess the numerical performance of
method~\eqref{discrete-formulation};
to this aim, we consider two iterative schemes
in order to cope with the nonlinearity.
In what follows, we consider a slight generalization
of problem~\eqref{Darcy--Forchheimer-strong},
where the first equation admits
an inhomogeneous right-hand side~$\fbf$.

\paragraph*{Iterative scheme~1.}
We consider a \textit{standard}
fixed point iterative scheme.
We initialize the scheme constructing
an initial guess~$(\ubfhzeroup,\phzero)$
in $\Vbfcalhkmo \times \QcalhgDk$,
which is solution to the linear Darcy model problem
\begin{subequations} \label{algorithm:Darcy-step}
\begin{align}
\frac{\mu}{\rho} \int_{\Omega} (\Kbbmo)\ubfhzeroup \cdot \vbfh \dxbf
    + \int_{\Omega}\nablaboldh \phzero \cdot \vbfh \dxbf
    = \int_{\Omega} \fbf \cdot \vbfh \dxbf &
    \qquad\qquad \forall\, \vbfh \in \Vbfcalhkmo  \\
\int_{\Omega}\nablabold \qh \cdot \ubfhzeroup \dxbf
    = - \int_{\Omega} b\,\qh\dxbf
    + \int_{\GammaN} \gN \, \qh \ds&
    \qquad\qquad \forall\, \qh \in \Qcalhzerok.
\end{align}
\end{subequations}
Then, for all $\nrm\geq1$,
given the solution $(\ubfhnmo,\phnmo)$
at the ($\nrm-1$)-th step of the iterative scheme,
we solve: find $(\ubfhn,\phn)$
in $\Vbfcalhkmo \times \QcalhgDk$ such that
\small
\begin{subequations} \label{algorithm:Darcy-Forchheimer-step}
\begin{align}
    \frac{\mu}{\rho} \int_{\Omega}(\Kbbmo)\ubfhn \cdot \vbfh \dxbf
        + \frac{\beta}{\rho} \int_{\Omega}\EucNorm{\ubfhnmo}^{\alphamt} \ubfhn \cdot \vbfh \dxbf
        + \int_{\Omega}\nablabold \phn \cdot \vbfh \dxbf
        = \int_{\Omega} \fbf \cdot \vbfh \dxbf &
        \quad \forall\, \vbfh \in \DGkmo  \\
    \int_{\Omega}\nablabold \qh \cdot \ubfhn \dxbf
        = - \int_{\Omega} b\,\qh\dxbf
        + \int_{\GammaN} \gN \, \qh \ds&
        \quad \forall\, \qh \in \Qcalhzerok. 
\end{align}
\end{subequations}
\normalsize
With obvious notation for the matrices~$\Mbb$,
$\Nbbnmo$, and~$\Bbb$,
and for the vectors~$\fbf$, $\bbf$, and $\gbf$,
we set
\begin{equation*}
\Abbn:=
\begin{pmatrix}
    \Mbb + \Nbbnmo & \Bbb^T   \\
    \Bbb            & \zerobf   
\end{pmatrix}
\qquad\qquad\qquad
\rbf :=
\begin{pmatrix}
    \fbf \\
    -\bbf + \gbf
\end{pmatrix}.
\end{equation*}
In matrix form,
\eqref{algorithm:Darcy-Forchheimer-step} reads
\begin{equation*}
    \Abbn
    \begin{pmatrix}
        \ubfhn \\
        \pbfhn 
    \end{pmatrix}
    = \rbf.
\end{equation*}

\paragraph*{Iterative scheme~2.}
We further consider a \textit{relaxed}
(also known in the literature as
\textit{damped} or \textit{Mann}'s)
fixed point iterative scheme;
see, e.g., \cite[Sect.~1.2, Ch.~4]{Berinde:2007}.
The initial guess
$(\ubfhzerotilde,\phzerotilde)=(\ubfhzeroup,\phzero)$
is the solution to~\eqref{algorithm:Darcy-step}.
Then, for all $\nrm\geq1$,
given $(\ubfhnmotilde,\phnmotilde)$,
we proceed as follows:
(\emph{i}) find~$(\ubfhn,\phn)$ solution
to~\eqref{algorithm:Darcy-Forchheimer-step}
where the term $\EucNorm{\ubfhnmo}^{\alphamt}$
is replaced by~$\EucNorm{\ubfhnmotilde}^{\alphamt}$;
(\emph{ii}) given a relaxation parameter~$\omega$ in $(0,1]$, set
\begin{equation} \label{damping}
    \ubfhntilde := \omega \ubfhn + (1-\omega)\ubfhnmotilde,
    \qquad\qquad\qquad
    \phntilde=\phn.
\end{equation}
When no confusion occurs, we shall denote
either~$(\ubfhn,\phn)$
or~$(\ubfhntilde,\phntilde)$ by~$\sbfhn$.

\paragraph*{Meshes.}
We consider sequences of shape-regular, quasi-uniform,
unstructured simplicial meshes with decreasing diameter.

\paragraph*{Parameters.}
Given~$\Nbb$ the matrix associated with
$\frac{\beta}{\rho} \int_{\Omega}\EucNorm{\ubfhn}^{\alphamt} \ubfhn \cdot \vbfh \dxbf$,
we set
\begin{equation*}
\Abb :=
\begin{pmatrix}
    \Mbb + \Nbb     & \Bbb^T   \\
    \Bbb            & \zerobf   
\end{pmatrix} .
\end{equation*}
We consider the following stopping criterion
for the iterative scheme:
\begin{equation} \label{stopping-criterion}
    \Norm{\Abb \sbfhn - \rbf }_{\L^2(\Omega)} \leq \textsc{TOL}.
\end{equation}
In all experiments below,
\textsc{TOL} is set to~$10^{-8}$.
We also consider a maximum number
of iterations~$\nmax$ equal to~$2500$.
We set $\Kbb = \Ibb$, and $\mu = \rho = 1$.
We consider full Neumann boundary conditions,
i.e., set~$\GammaN = \Gamma$,
and impose the zero-mean condition for the potential
as a Lagrange multiplier.

\paragraph*{Test case~1.}
On the domain $\Omega := (-1,1)^2$,
we consider the following exact flux and potential:
\begin{equation} \label{u1}
\ubf(x,y) =
\left(\begin{array}{c}
     \sin(\pi x)  \\
     \cos(\pi y)
\end{array}\right)
\qquad\qquad\qquad
\pop(x,y) = 
\cos\big(\frac{\pi}{2}x\big)
\sin\big(\frac{\pi}{2}y\big).
\end{equation}
The corresponding data are computed accordingly:
the source term is given by
\begin{equation*}
\fbf(x,y) = \left(
\begin{array}{c}
\sin(\pi x)-\frac\pi2\sin\big(\frac{\pi}{2}x\big)\sin\big(\frac{\pi}{2}y\big)+\beta\big(\sin(\pi x)^2+\cos(\pi y)^2\big)^{\frac{\alpha-2}{2}}\sin(\pi x)  \\
\cos(\pi y)-\frac\pi2\cos\big(\frac{\pi}{2}x\big)\cos\big(\frac{\pi}{2}y\big)+\beta\big(\sin(\pi x)^2+\cos(\pi y)^2\big)^{\frac{\alpha-2}{2}}\cos(\pi y)
\end{array} \right);
\end{equation*}
the divergence constraint is given by
$b(x,y) = -2\pi\sin(\pi x)\sin(\pi y)$;
the Neumann boundary datum~$\gN$ is~$0$
on the left and right facets;
$1$ on the bottom facet; $-1$ on the top facet.

\paragraph*{Test case~2.}
On the domain $\Omega := (-1,1)^2$,
we consider the following exact flux and potential:
\begin{equation} \label{u2}
\ubf(x,y) =
\left(\begin{array}{c}
     1  \\
     -1
\end{array}\right)
\qquad\text{ and }\qquad
\pop(x,y) = x^3 + y^3.
\end{equation}
The corresponding data are computed accordingly:
the source term is given by
\begin{equation*}
\fbf(x,y) = \left(
\begin{array}{c}
1 + 2^{\frac{\alpha-2}{2}}\beta + 3x^2  \\
-1 - 2^\frac{\alpha-2}{2}\beta + 3y^2
\end{array} \right);
\end{equation*}
the divergence constraint is given by~$b(x,y) = 0$;
the Neumann boundary datum~$\gN$ is~$1$
on the right and bottom facets;
$-1$ on the left and top facets.

\paragraph*{Error measures.}
In what follows, for a fixed order~$k$,
we compute the following error measures:
\begin{equation}\label{error-measures}
E_{k-1,n}^{\ubf}
:= \frac{\Norm{\ubf-\ubfhn}_{\LtwoOmegad}}{\Norm{\ubf}_{\LtwoOmegad}}
\qquad\text{ and }\qquad
E_{k,n}^{\pop}
:= \frac{\Norm{\nablaboldh(\pop-\phn)}_{\LalphaprimeOmegad}}{\Norm{\nablabold \pop}_{\LalphaprimeOmegad}}.
\end{equation}
The same error measures are considered
when~$\ubfhn$ is replaced by~$\ubfhntilde$.

\paragraph*{Numerical results: varying the order of accuracy; fixed point scheme.} 
In Figure~\ref{fig:sincos-k}, we display
the errors in~\eqref{error-measures}
under mesh refinements
for the iterative scheme~\eqref{algorithm:Darcy-Forchheimer-step}.
We consider orders of accuracy~$k$ in~$\{1,2,3,4\}$,
$\alpha = 3$, and~$\beta=10$.
In Table~\ref{table:sincos-k},
we report the number of iterations
needed to meet the stopping criterion~\eqref{stopping-criterion}.
\begin{figure}[htbp]
\centering
\includegraphics[width=0.49\linewidth]{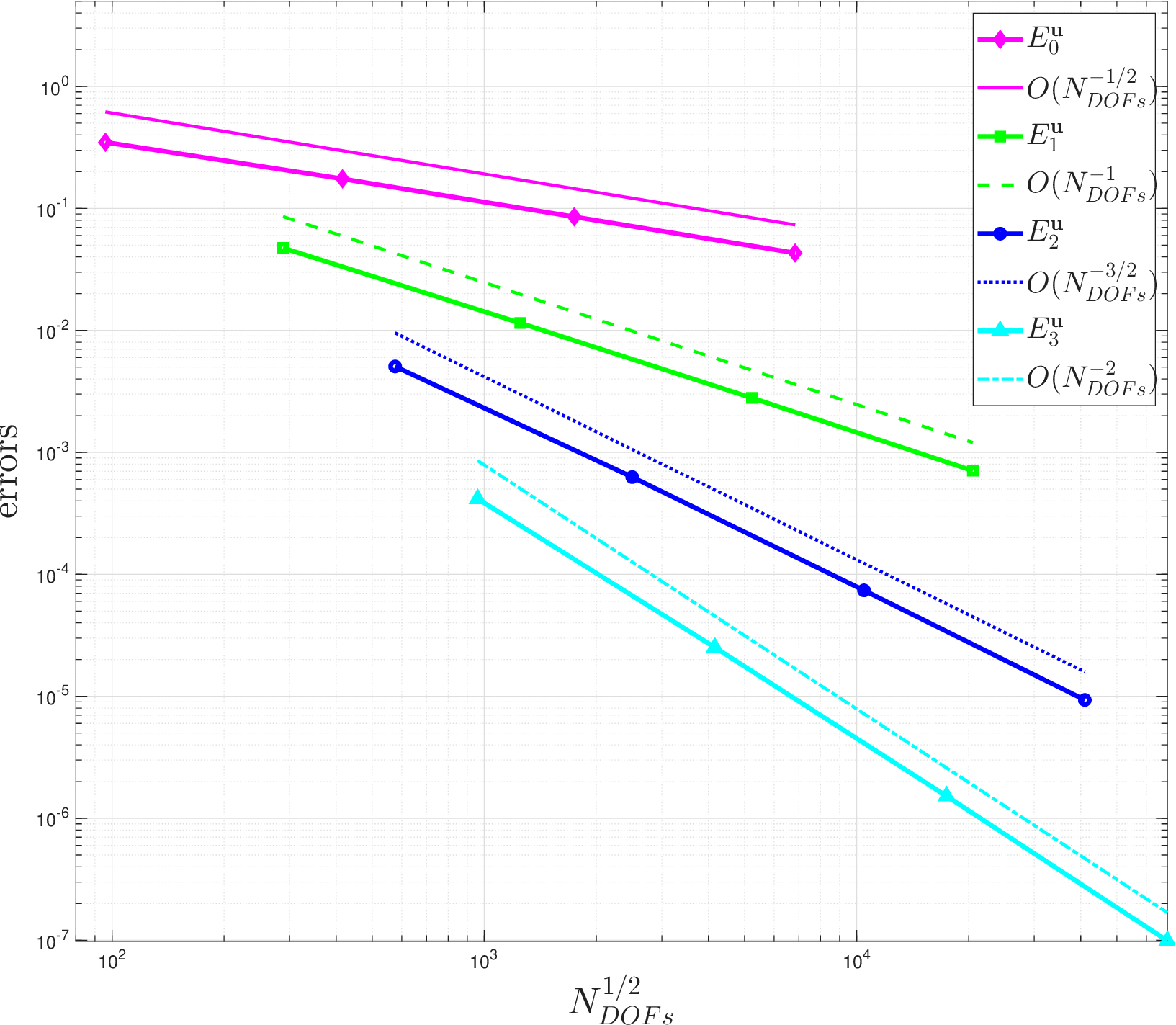}
\includegraphics[width=0.49\linewidth]{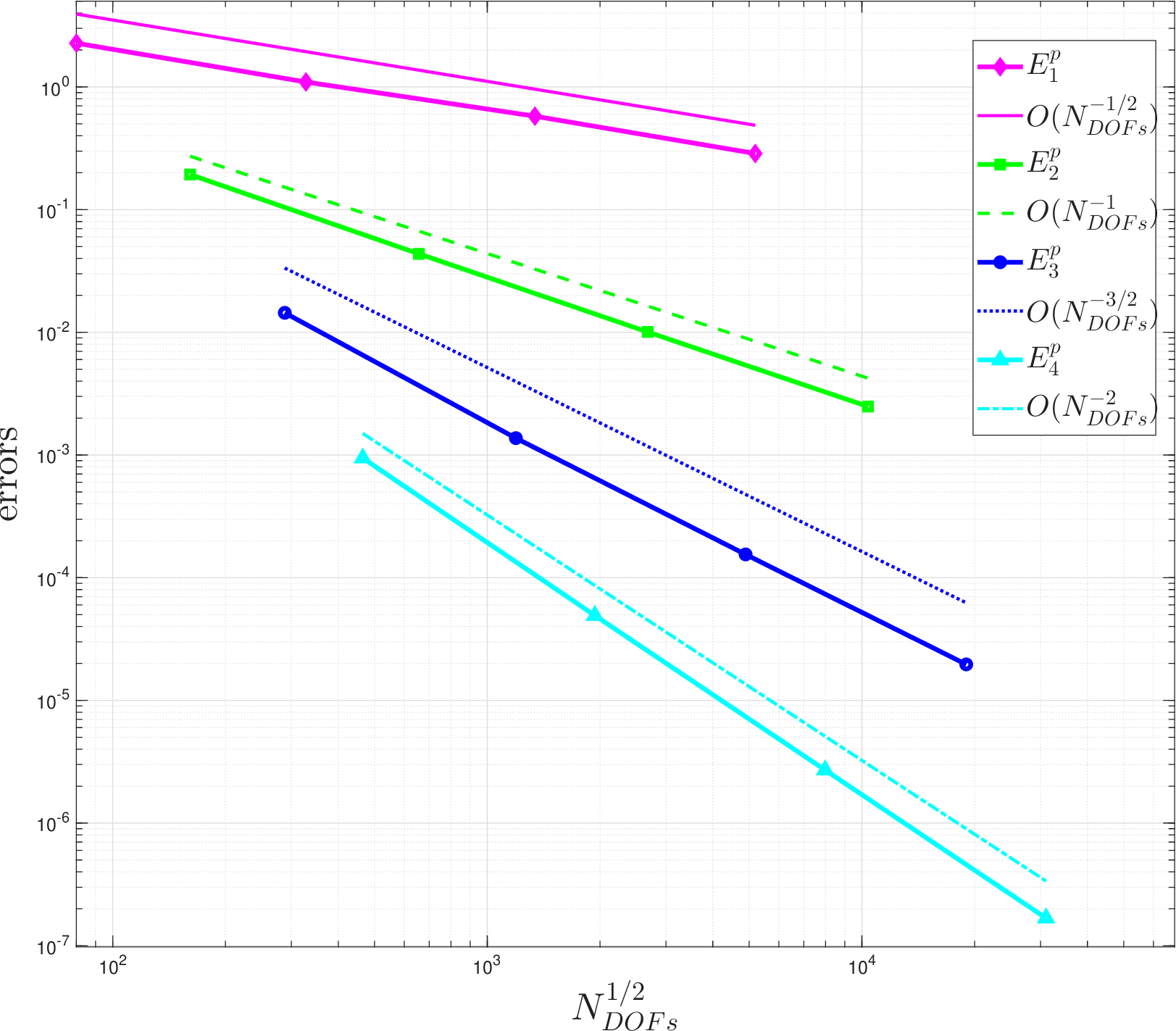}
\caption{Exact solution in~\eqref{u1};
$\h$-convergence of the flux (\emph{left-panel})
and the potential (\emph{right-panel})
for $k$ in $\{1,2,3,4\}$, $\alpha = 3$, and~$\beta=10$.
Standard fixed point
iterative scheme~\eqref{algorithm:Darcy-Forchheimer-step}.}
\label{fig:sincos-k}
\end{figure}
\begin{table}[htbp]
\centering
\begin{tabular}{lcccc}
\hline
 & h=0.5 & h=0.3 & h=0.15 & h=0.08 \\
\hline
$k=1$   & 85 & 122 & 145 & 160  \\
$k=2$   & 202 & 170 & 162 & 160 \\
$k=3$  & 162 & 167  & 166 & 166 \\
$k=4$   & 162 & 167  & 166 & 166  \\
\hline
\end{tabular}
\caption{Exact solution in~\eqref{u1}: iterations
needed to meet the stopping criterion~\eqref{stopping-criterion}
for different values of~$k$, $\alpha=3$, and~$\beta=10$.
Standard fixed point
iterative scheme~\eqref{algorithm:Darcy-Forchheimer-step}.}
\label{table:sincos-k}
\end{table}
The number of iterations seems stable
with respect to the order of accuracy~$k$
(with the exception of the lowest order case)
and to the mesh size~$\h$.

\paragraph*{Numerical results:
varying the coefficient~$\beta$ of the nonlinear term; fixed point scheme.}
In Figure~\ref{fig:sincos-beta}, we display
the errors in~\eqref{error-measures}
under mesh refinements
for the iterative scheme~\eqref{algorithm:Darcy-Forchheimer-step}.
We consider $\beta$ in $\{1,10,50,100\}$,
$k=2$, and~$\alpha = 3$.
In Table~\ref{table:sincos-beta},
we report the number of iterations
needed to meet the stopping criterion~\eqref{stopping-criterion}.
\begin{figure}[htbp]
\centering
\includegraphics[width=0.49\linewidth]{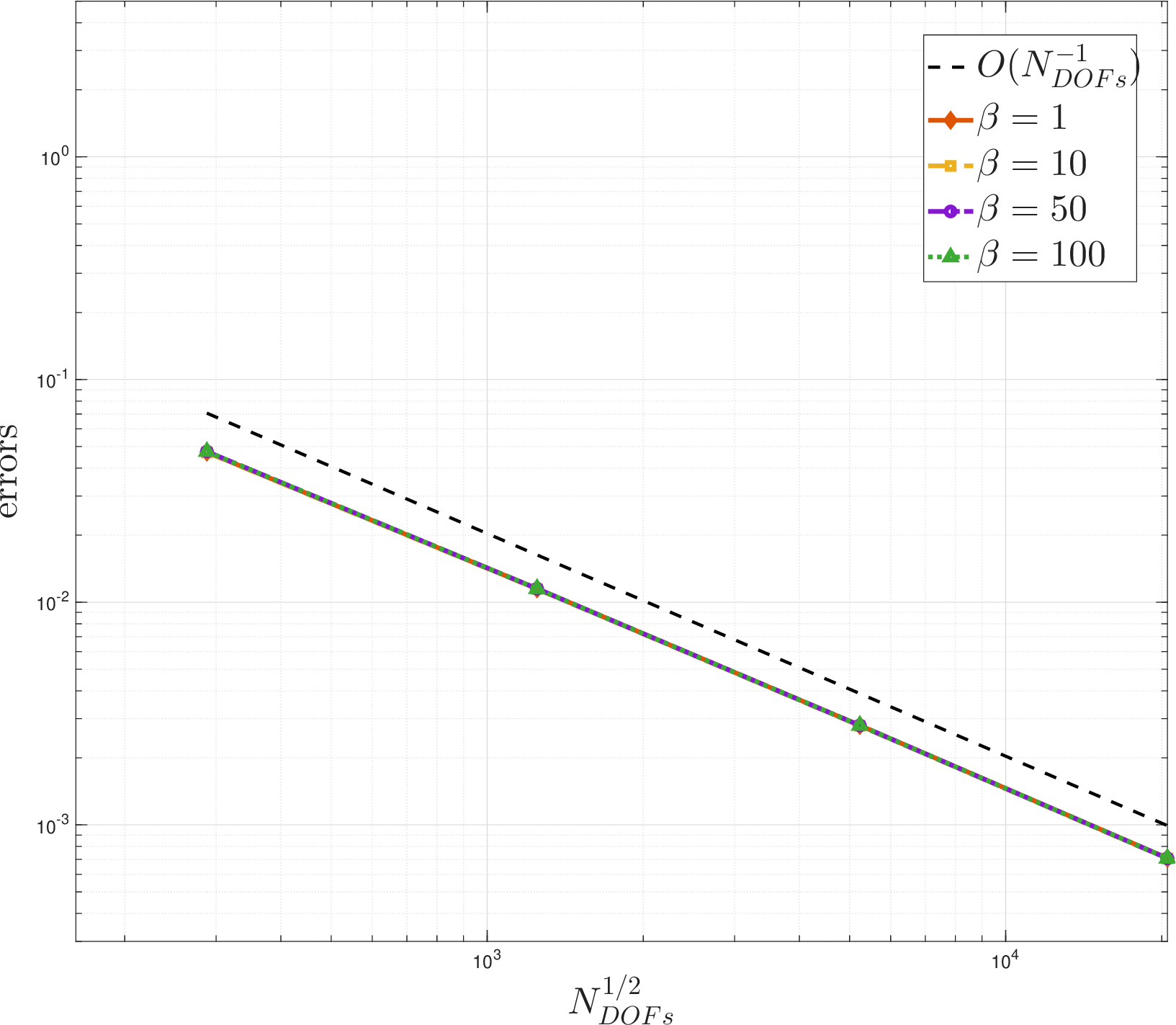}
\includegraphics[width=0.49\linewidth]{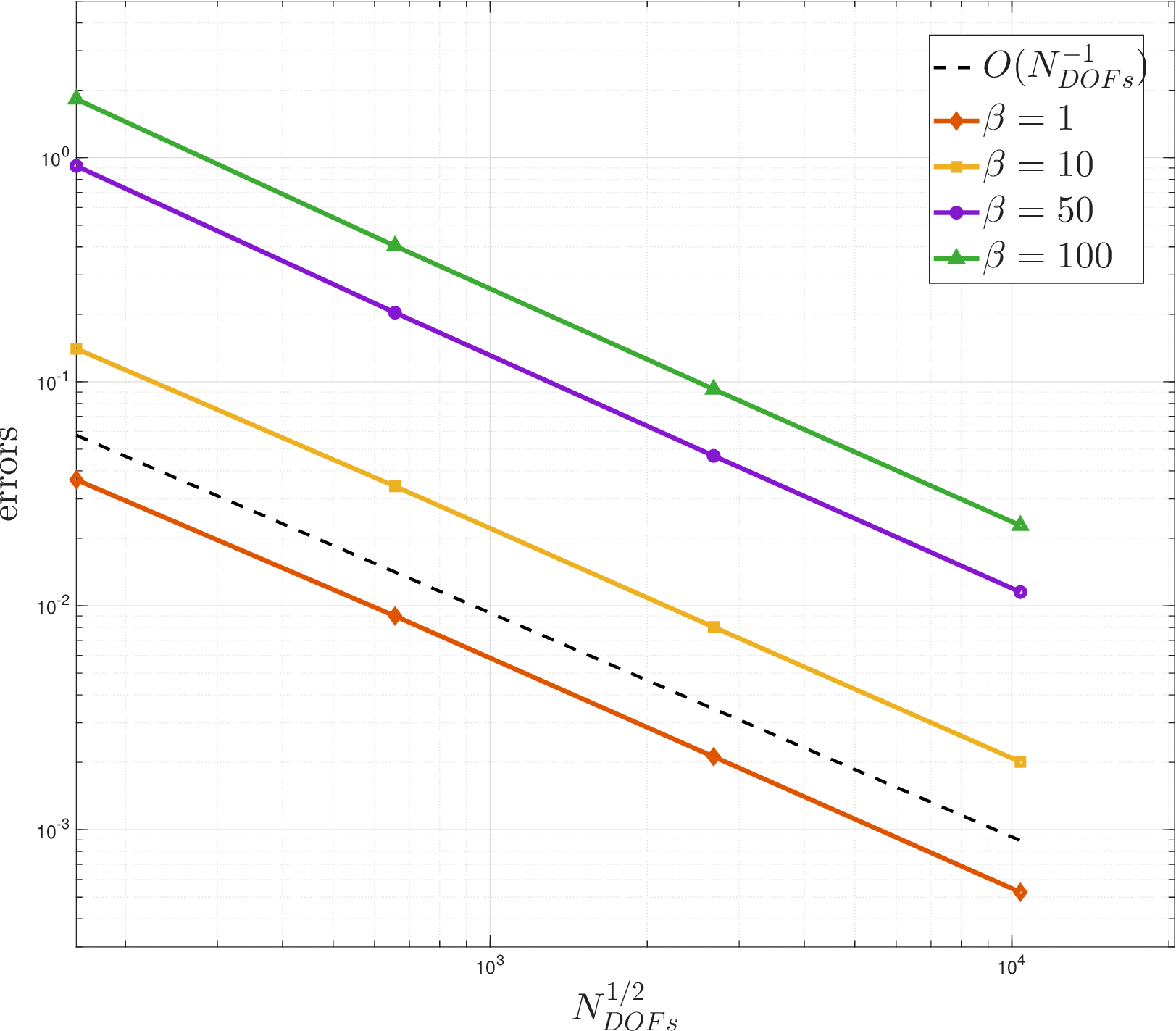}
\caption{Exact solution in~\eqref{u1};
$\h$-convergence of the flux (\emph{left-panel})
and the potential (\emph{right-panel})
for $\beta$ in $\{1,10,50,100\}$,
$k = 2$, and~$\alpha = 3$.
Standard fixed point iterative scheme~\eqref{algorithm:Darcy-Forchheimer-step}.}
\label{fig:sincos-beta}
\end{figure}
\begin{table}[htbp]
\centering
\begin{tabular}{lcccc}
\hline
 & h=0.5 & h=0.3 & h=0.15 & h=0.08 \\
\hline
$\beta=1$   & 23 & 22 & 22 & 21  \\
$\beta=10$  & 202 & 170 & 162 & 160  \\
$\beta=50$   & 913 & 753 & 770 & 745 \\
$\beta=100$   & 1604 & 1382 & 1464 & 1450 \\
\hline
\end{tabular}
\caption{Exact solution in~\eqref{u1}:
iterations
needed to meet the stopping criterion~\eqref{stopping-criterion}
for different values of~$\beta$, $k=2$, and~$\alpha=3$.
Standard fixed point
iterative scheme~\eqref{algorithm:Darcy-Forchheimer-step}.}
\label{table:sincos-beta}
\end{table}
The number of iterations grows linearly with the value of~$\beta$.

\paragraph*{Numerical results: varying the exponent~$\alpha$ of the nonlinear term;
$2<\alpha<3$; fixed point scheme.}
In Figure~\ref{fig:sincos-alpha-standardsmall}, we display
the errors in~\eqref{error-measures}
under mesh refinements
for the iterative scheme~\eqref{algorithm:Darcy-Forchheimer-step}.
We consider $\alpha$ in $\{ 2.2,2.4,2.6,2.8\}$,
$k=2$, and~$\beta=10$.
In Table~\ref{table:sincos-alpha-standardsmall},
we report the number of iterations
needed to meet the stopping criterion~\eqref{stopping-criterion}.
\begin{figure}[htbp]
\centering
\includegraphics[width=0.49\linewidth]{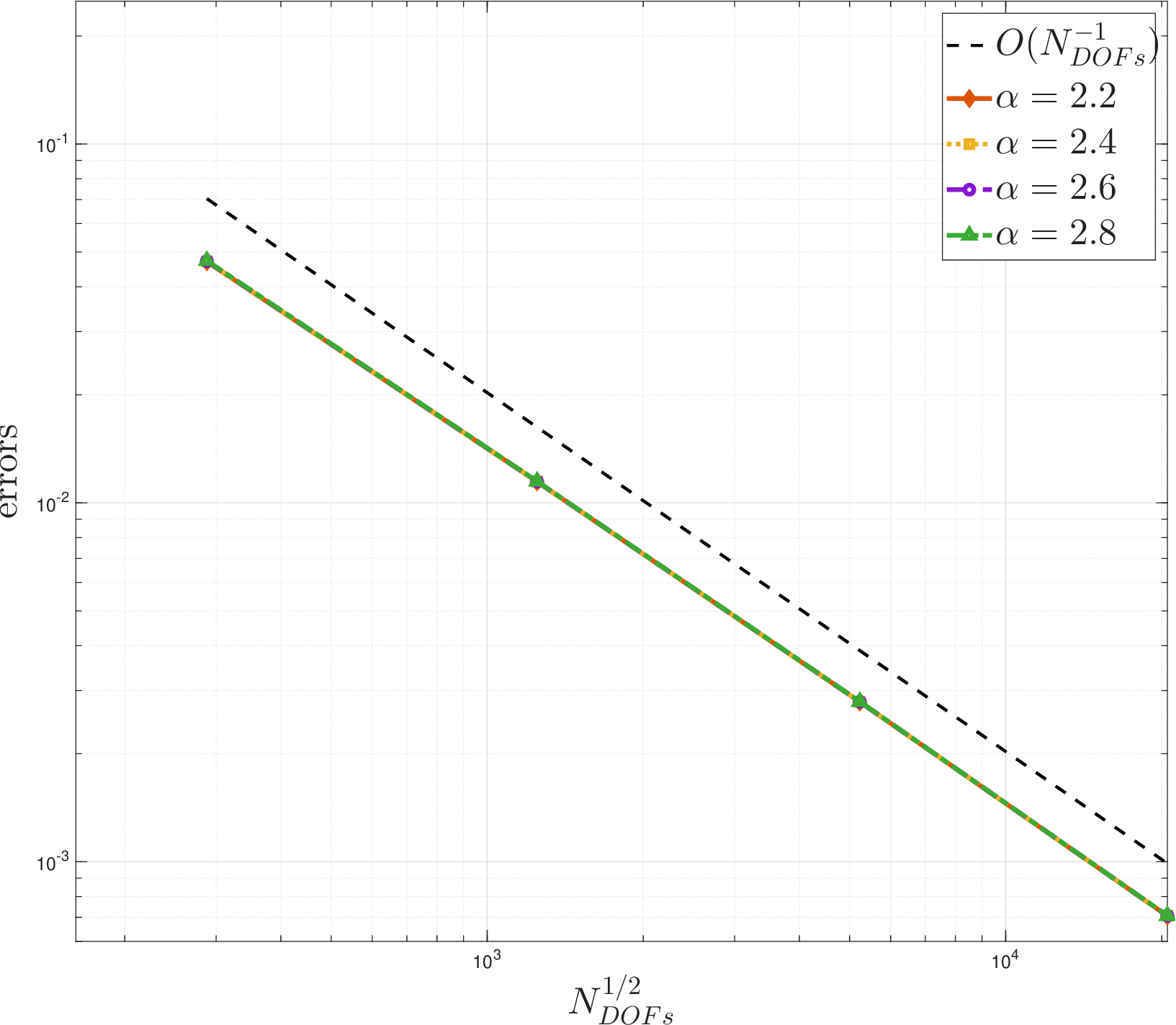}
\includegraphics[width=0.49\linewidth]{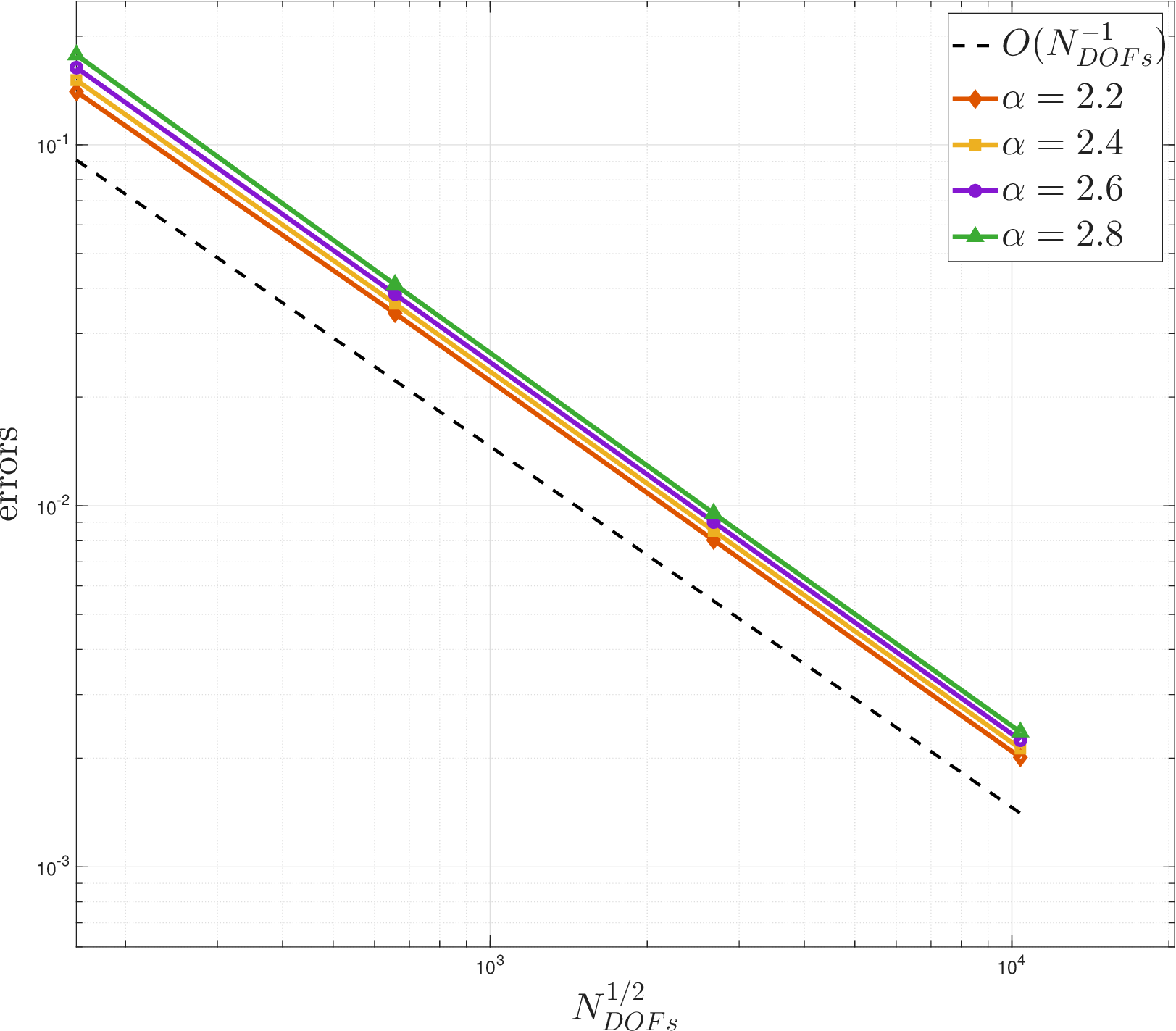}
\caption{Exact solution in~\eqref{u1};
$\h$-convergence of the flux (\emph{left-panel})
and the potential (\emph{right-panel})
for $\alpha$ in $\{2.2,2.4,2.6,2.8\}$,
$k = 2$, and~$\beta=10$.
Standard fixed point iterative scheme~\eqref{algorithm:Darcy-Forchheimer-step}.}
\label{fig:sincos-alpha-standardsmall}
\end{figure}
\begin{table}[htbp]
\centering
\begin{tabular}{lcccc}
\hline
 & h=0.5 & h=0.3 & h=0.15 & h=0.08 \\
\hline
$\alpha=2.2$  & 11 & 11 & 11 & 11  \\
$\alpha=2.4$ & 17 & 18 & 18 & 19   \\
$\alpha=2.6$  & 28 & 28  & 28 & 28 \\
$\alpha=2.8$ & 52 & 52  & 52 & 51 \\
\hline
\end{tabular}
\caption{Exact solution in~\eqref{u1}:
iterations
needed to meet the stopping criterion~\eqref{stopping-criterion}
for values of~$\alpha$ smaller than~$3$,
$k=2$, and~$\beta=10$.
Standard fixed point iterative scheme~\eqref{algorithm:Darcy-Forchheimer-step}.}
\label{table:sincos-alpha-standardsmall}
\end{table}
The number of iterations grows with~$\alpha$.
Similar results were obtained using higher order methods.

For~$\alpha$ larger than~$3$,
the standard fixed point algorithm~\eqref{algorithm:Darcy-Forchheimer-step}
fails to reach the prescribed tolerance
within the prescribed maximum number of iterations;
we do not report the results here for the sake of conciseness.
Thus, in the next paragraphs, we employ
the relaxed iterative scheme~\eqref{damping}
for such values of~$\alpha$.

\paragraph*{Numerical results: varying the exponent~$\alpha$ of the nonlinear term; $2 < \alpha \leq 4$; relaxed scheme.}
In Figure~\ref{fig:sincos-alpha-damplarge},
we display the errors in~\eqref{error-measures}
under mesh refinements for the relaxed scheme~\eqref{damping}.
We consider $\alpha$ in $\{2.2,2.4,2.6,2.8,3.2,3.4,3.6,3.8,4\}$,
$k=2$, $\beta=10$, and relaxation parameter~$\omega = 0.5$.
The corresponding number of iterations
is reported in Table~\ref{table:sincos-alpha-damplarge}.
\begin{figure}[htbp]
\centering
\includegraphics[width=0.49\linewidth]{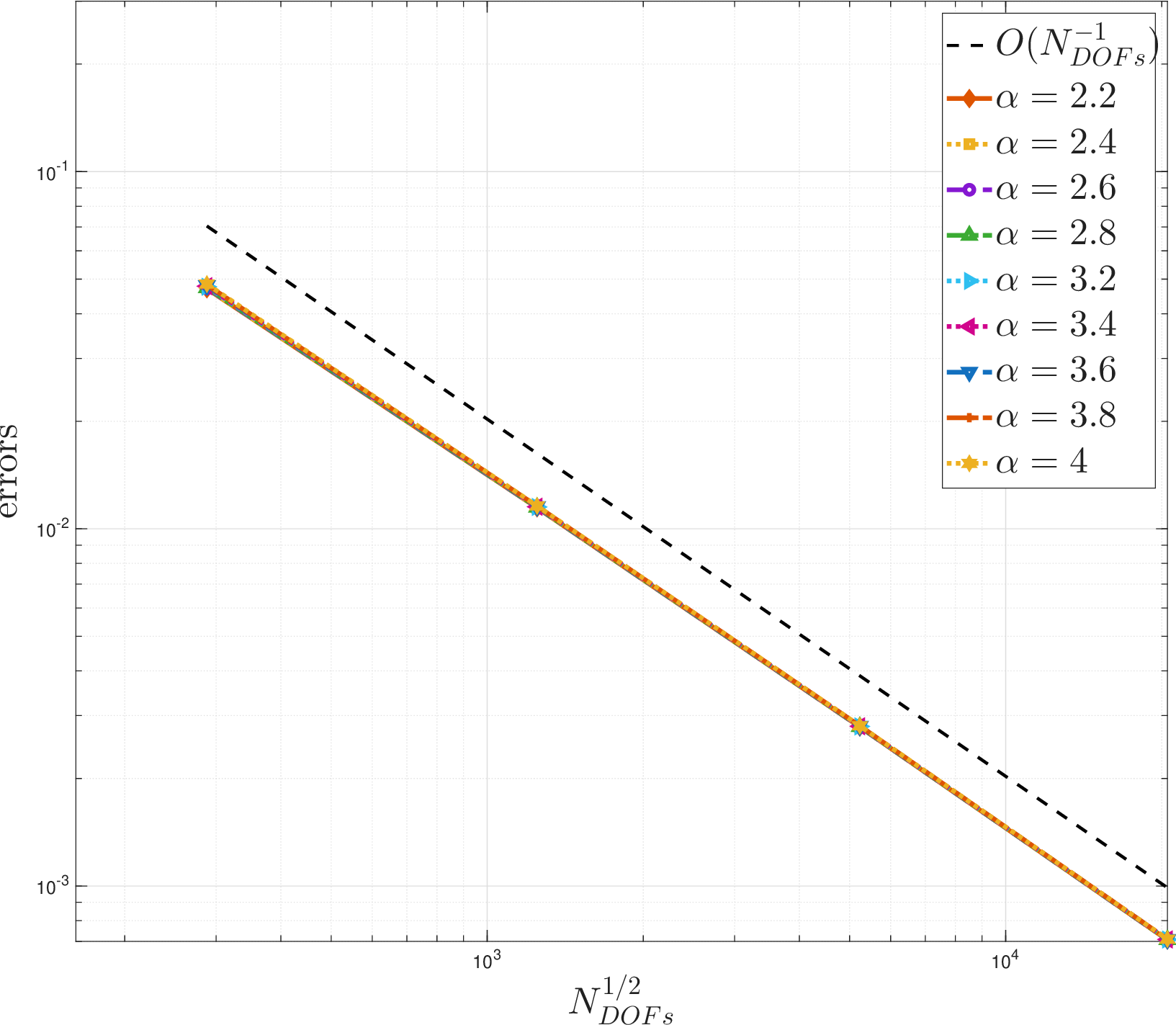}
\includegraphics[width=0.49\linewidth]{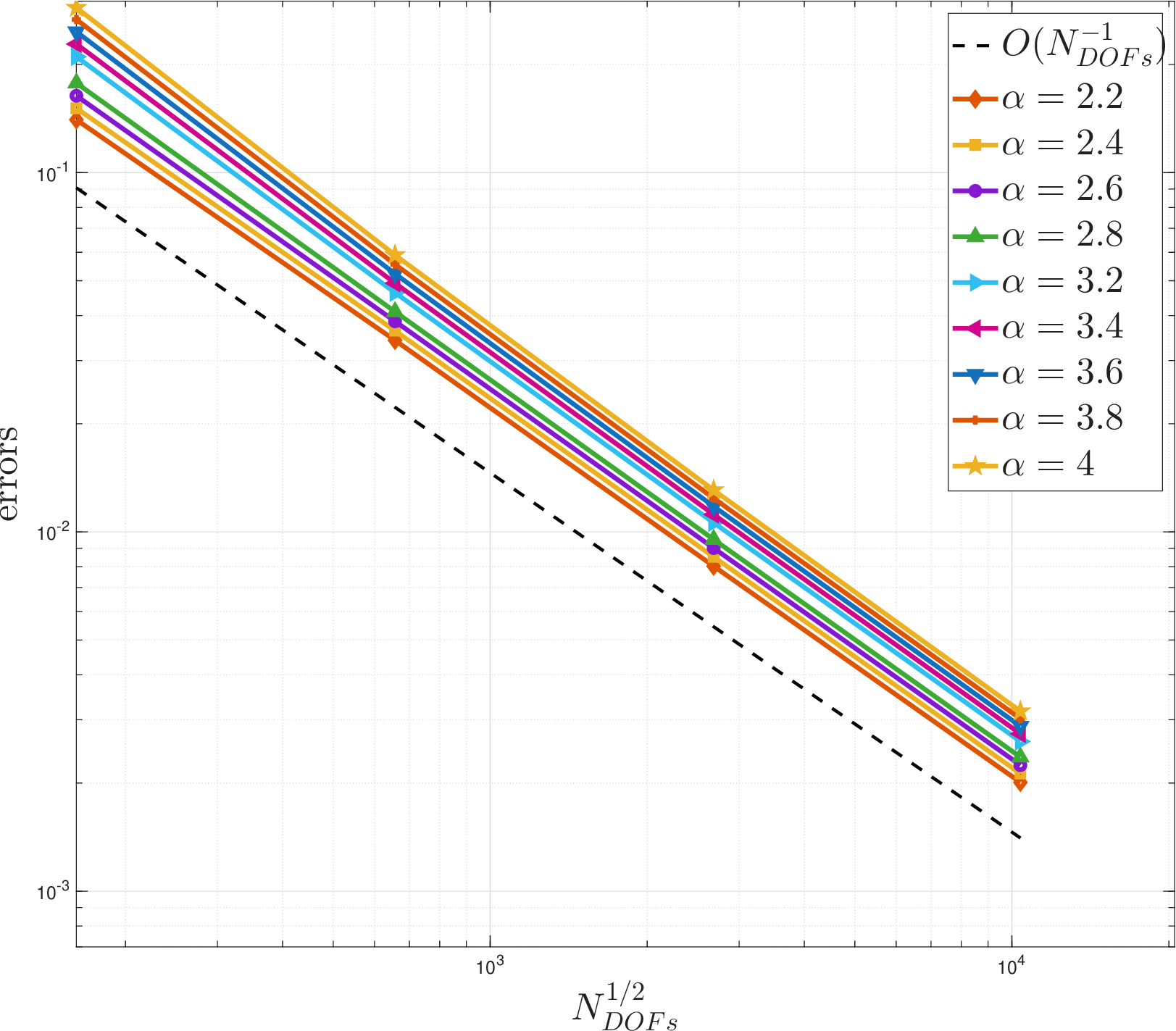}
\caption{Exact solution in~\eqref{u1};
$\h$-convergence of the flux (\emph{left-panel})
and the potential (\emph{right-panel})
for $\alpha$ in $\{2.2,2.4,2.6,2.8,3.2,3.4,3.6,3.8,4\}$,
$k = 2$, and~$\beta=10$.
Relaxed fixed point iterative scheme~\eqref{damping} with~$\omega = 0.5$.}
\label{fig:sincos-alpha-damplarge}
\end{figure}
\begin{table}[htbp]
\centering
\begin{tabular}{lcccc}
\hline
 & h=0.5 & h=0.3 & h=0.15 & h=0.08 \\
\hline
$\alpha=2.2$ & 23 & 22 & 22 & 22 \\
$\alpha=2.4$ & 22 & 22 & 22 & 22 \\
$\alpha=2.6$ & 22 & 22 & 22 & 22  \\
$\alpha=2.8$ & 22 & 22 & 22 & 22 \\
$\alpha=3.2$ & 23 & 22 & 22 & 22 \\
$\alpha=3.4$ & 22 & 22 & 22 & 22 \\
$\alpha=3.6$ & 22 & 22 & 22 & 22  \\
$\alpha=3.8$ & 23 & 23 & 23 & 23 \\
$\alpha=4$ & 23 & 23 & 23 & 23 \\
\hline
\end{tabular}
\caption{Exact solution in~\eqref{u1}:
iterations
needed to meet the stopping criterion~\eqref{stopping-criterion}
for values of~$\alpha$ between~$2$ and~$4$,
$k=2$, and~$\beta=10$.
Relaxed fixed point iterative scheme~\eqref{damping} with~$\omega = 0.5$.}
\label{table:sincos-alpha-damplarge}
\end{table}
The number of iterations remains essentially stable
as~$\alpha$ increases;
for values of~$\alpha$ smaller than~$3$,
this was instead not the case for
scheme~\eqref{algorithm:Darcy-Forchheimer-step};
cf. Table~\ref{table:sincos-alpha-standardsmall}.
For values of~$\alpha$ close to~$2$,
the relaxed version requires more iterations.

\paragraph*{Numerical results: varying the exponent~$\alpha$ of the nonlinear term;
$\alpha>4$; relaxed scheme.}
In Figure~\ref{fig:sincos-alphagrt4-damplarge}, we display
the errors in~\eqref{error-measures}
under mesh refinements for the relaxed scheme~\eqref{damping}.
We consider $\alpha$ in $\{4.2,4.4,4.6,4.8,5,5.1\}$,
$k=2$, $\beta=10$, and relaxation parameter~$\omega = 0.5$.
The corresponding number of iterations
is reported in Table~\ref{table:sincos-alphagrt4-damplarge}.
\begin{figure}[htbp]
\centering
\includegraphics[width=0.49\linewidth]{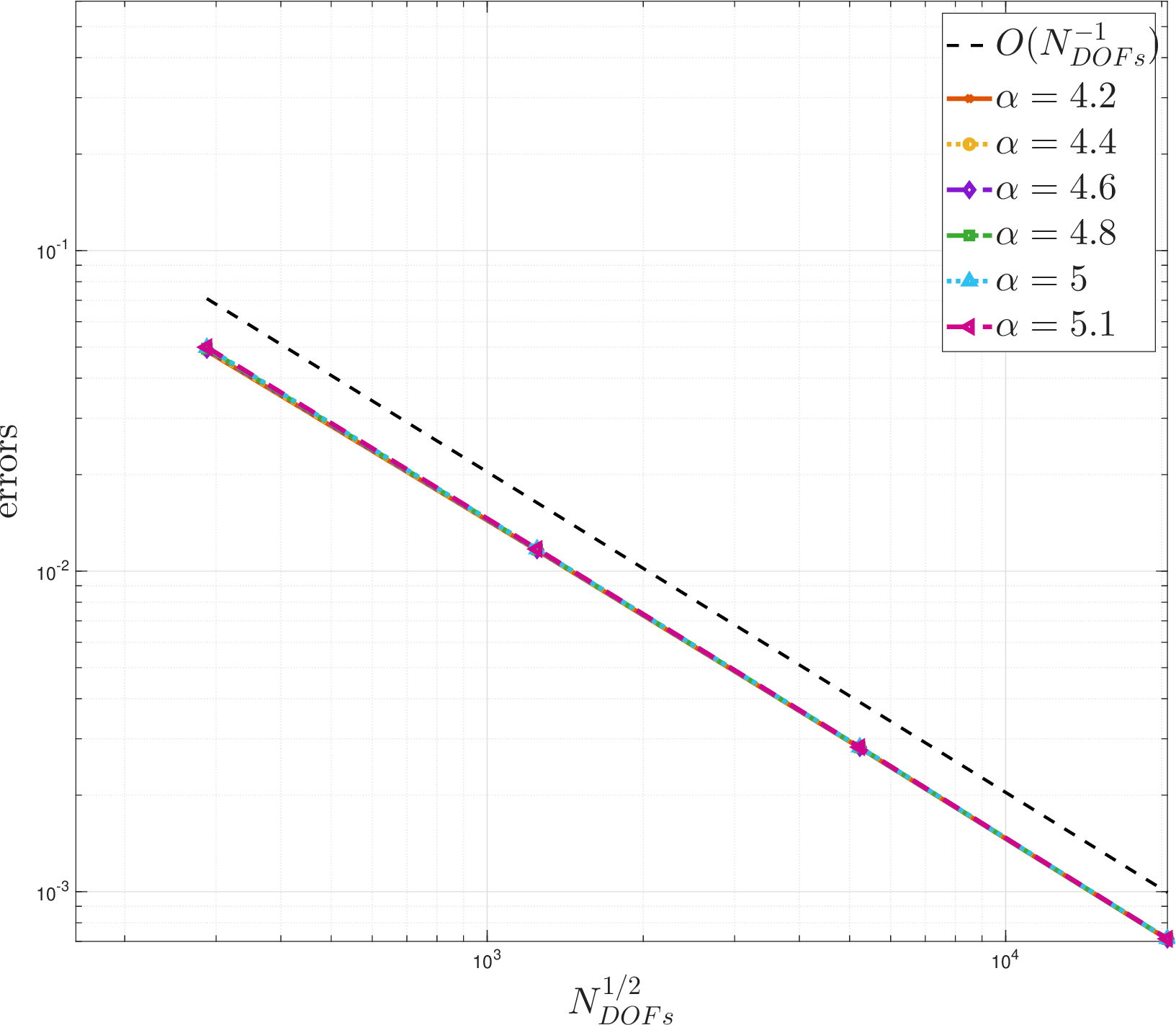}
\includegraphics[width=0.49\linewidth]{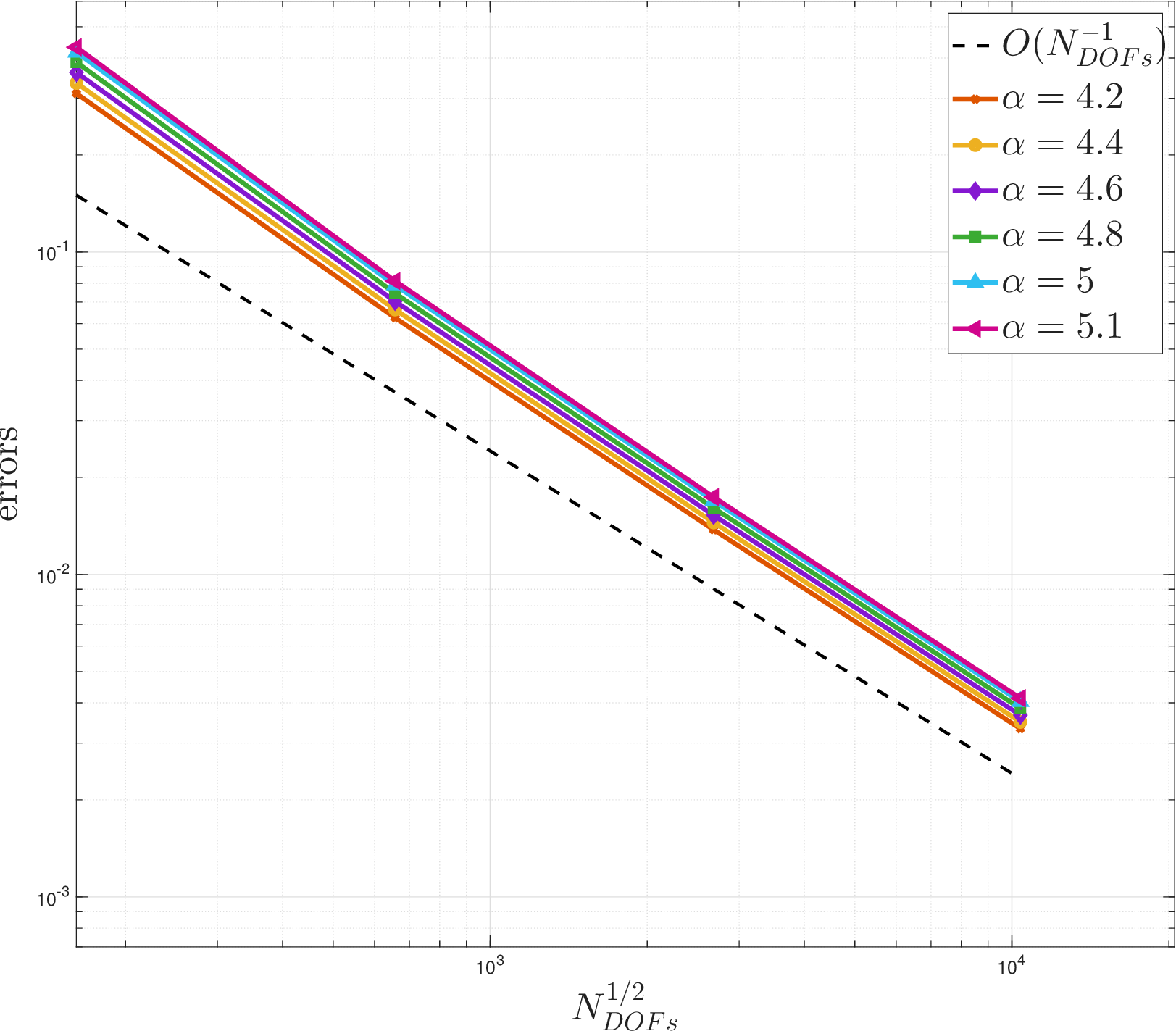}
\caption{Exact solution in~\eqref{u1};
$\h$-convergence of the flux (\emph{left-panel})
and the potential (\emph{right-panel})
for $\alpha$ in $\{4.2,4.4,4.6,4.8,5,5.1\}$,
$k = 2$, and $\beta=10$.
Relaxed fixed point iterative scheme~\eqref{damping} with~$\omega = 0.5$.}
\label{fig:sincos-alphagrt4-damplarge}
\end{figure}
\begin{table}[htbp]
\centering
\begin{tabular}{lcccc}
\hline
 & h=0.5 & h=0.3 & h=0.15 & h=0.08 \\
\hline
$\alpha=4.2$ & 25 & 25 & 25 & 25 \\
$\alpha=4.4$ & 33 & 33 & 33 & 33 \\
$\alpha=4.6$ & 50 & 45  & 47 & 47  \\
$\alpha=4.8$ & 85 & 73  & 77 & 77 \\
$\alpha=5$ & 224 & 172  & 187 & 186 \\
$\alpha=5.1$ & 934 & 715  & 774 & 2051 \\
\hline
\end{tabular}
\caption{Exact solution in~\eqref{u1}:
iterations
needed to meet the stopping criterion~\eqref{stopping-criterion}
for values of~$\alpha$ larger than~$4$,
$k=2$, and~$\beta=10$.
Relaxed fixed point iterative scheme~\eqref{damping} with~$\omega = 0.5$.}
\label{table:sincos-alphagrt4-damplarge}
\end{table}
The number of iterations grows with~$\alpha$,
differently from what was reported in Table~\ref{table:sincos-alpha-damplarge};
optimal convergence rates are observed.
Convergence is not achieved for~$\alpha$
larger than or equal to~$5.2$
within the prescribed maximum number of iterations.

\paragraph*{Numerical results: varying the relaxation parameter~$\omega$;
$\alpha>4$; relaxed scheme.}
As a first attempt to address the lack of convergence
for large~$\alpha$ observed in the previous paragraph,
we consider different choices
of the relaxation parameter~$\omega$.
We consider $\omega$
in $\{0.1,0.2,0.3,0.4,0.5,0.6,0.8\}$,
$\alpha$ in $\{4.6,5.1,6\}$,
$k=2$, and~$\beta=10$.
In Table~\ref{table:sincos-omega},
we report the number of iterations 
under mesh refinements
for the relaxed scheme~\eqref{damping}.
\begin{table}[htbp]
\centering
\small
\begin{subtable}{0.49\textwidth}
\centering
\begin{tabular}{lcccc}
\hline
 & h=0.5 & h=0.3 & h=0.15 & h=0.08 \\
\hline
$\omega=0.1$ & 145 & 143 & 140 & 139 \\
$\omega=0.2$ & 69 & 68 & 67 & 67 \\
$\omega=0.3$ & 44 & 43 & 43 & 43 \\
$\omega=0.4$ & 31 & 31 & 31 & 32 \\
$\omega=0.5$ & 50 & 45 & 47 & 47 \\
$\omega=0.6$ & $\nmax$ & $\nmax$ & $\nmax$ & $\nmax$ \\
$\omega=0.8$ & $\nmax$ & $\nmax$ & $\nmax$ & $\nmax$ \\
\hline
\end{tabular}
\caption{$\alpha=4.6$}
\end{subtable}
\hfill
\begin{subtable}{0.49\textwidth}
\centering
\begin{tabular}{lcccc}
\hline
 & h=0.5 & h=0.3 & h=0.15 & h=0.08 \\
\hline
$\omega=0.1$ & 148 & 146 & 142 & 141 \\
$\omega=0.2$ & 71 & 70 & 68 & 68 \\
$\omega=0.3$ & 45 & 44 & 44 & 44 \\
$\omega=0.4$ & 33 & 32 & 32 & 33 \\
$\omega=0.5$ & $\nmax$ & $\nmax$ & $\nmax$ & $\nmax$ \\
$\omega=0.6$ & $\nmax$ & $\nmax$ & $\nmax$ & $\nmax$ \\
$\omega=0.8$ & $\nmax$ & $\nmax$ & $\nmax$ & $\nmax$ \\
\hline
\end{tabular}
\caption{$\alpha=5.1$}
\end{subtable}

\vspace{0.5cm}

\begin{subtable}{0.7\textwidth}
\centering
\begin{tabular}{lcccc}
\hline
 & h=0.5 & h=0.3 & h=0.15 & h=0.08 \\
\hline
$\omega=0.1$ & 154 & 154 & 146 & 146 \\
$\omega=0.2$ & 73 & 73 & 70 & 70 \\
$\omega=0.3$ & 47 & 47 & 45 & 46 \\
$\omega=0.4$ & 274 & 227 & 245 & 243 \\
$\omega=0.5$ & $\nmax$ & $\nmax$ & $\nmax$ & $\nmax$ \\
$\omega=0.6$ & $\nmax$ & $\nmax$ & $\nmax$ & $\nmax$ \\
$\omega=0.8$ & $\nmax$ & $\nmax$ & $\nmax$ & $\nmax$ \\
\hline
\end{tabular}
\caption{$\alpha=6$}
\end{subtable}
\caption{Exact solution in~\eqref{u1}:
iterations for different values of~$\omega$,
$k=2$, $\beta=10$,
and~$\alpha=4.6$ (\emph{top-left panel}),
$\alpha=5.1$ (\emph{top-right panel}),
and~$\alpha=6$ (\emph{bottom panel}).
Relaxed fixed point iterative scheme~\eqref{damping}.
The entry $\nmax$ stands for ``maximum number of iterations is reached''.}
\label{table:sincos-omega}
\normalsize
\end{table}
The results suggest that a careful choice of the relaxation parameter
may affect the performance of the method,
not only by possibly optimizing the number of iterations,
but also by determining whether the method converges or not;
the optimal value of~$\omega$
may depend on the exponent of the nonlinear term~$\alpha$.
More sophisticated iterative schemes should be used
in order to avoid such a dependence
and to cope with larger~$\alpha$.

\paragraph*{Lack of consistency for higher order methods.} 
In Figure~\ref{fig:constant-cubic},
we display the errors in~\eqref{error-measures}
under mesh refinements
for the iterative scheme~\eqref{algorithm:Darcy-Forchheimer-step}.
We consider~$k=1,2$, $\alpha = 3$, and~$\beta=10$.
\begin{figure}[htbp]
\centering
\includegraphics[width=0.49\linewidth]{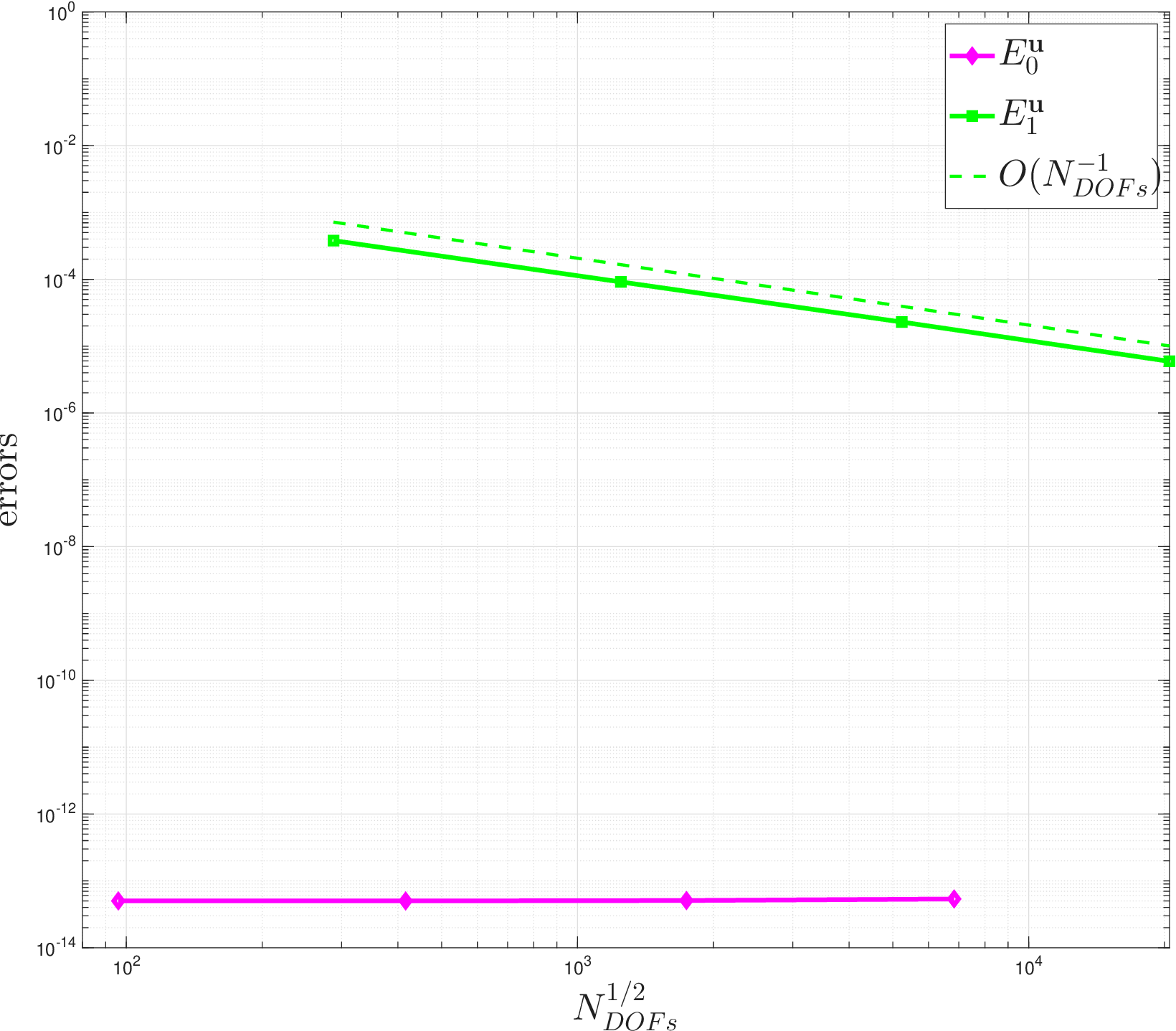}
\includegraphics[width=0.49\linewidth]{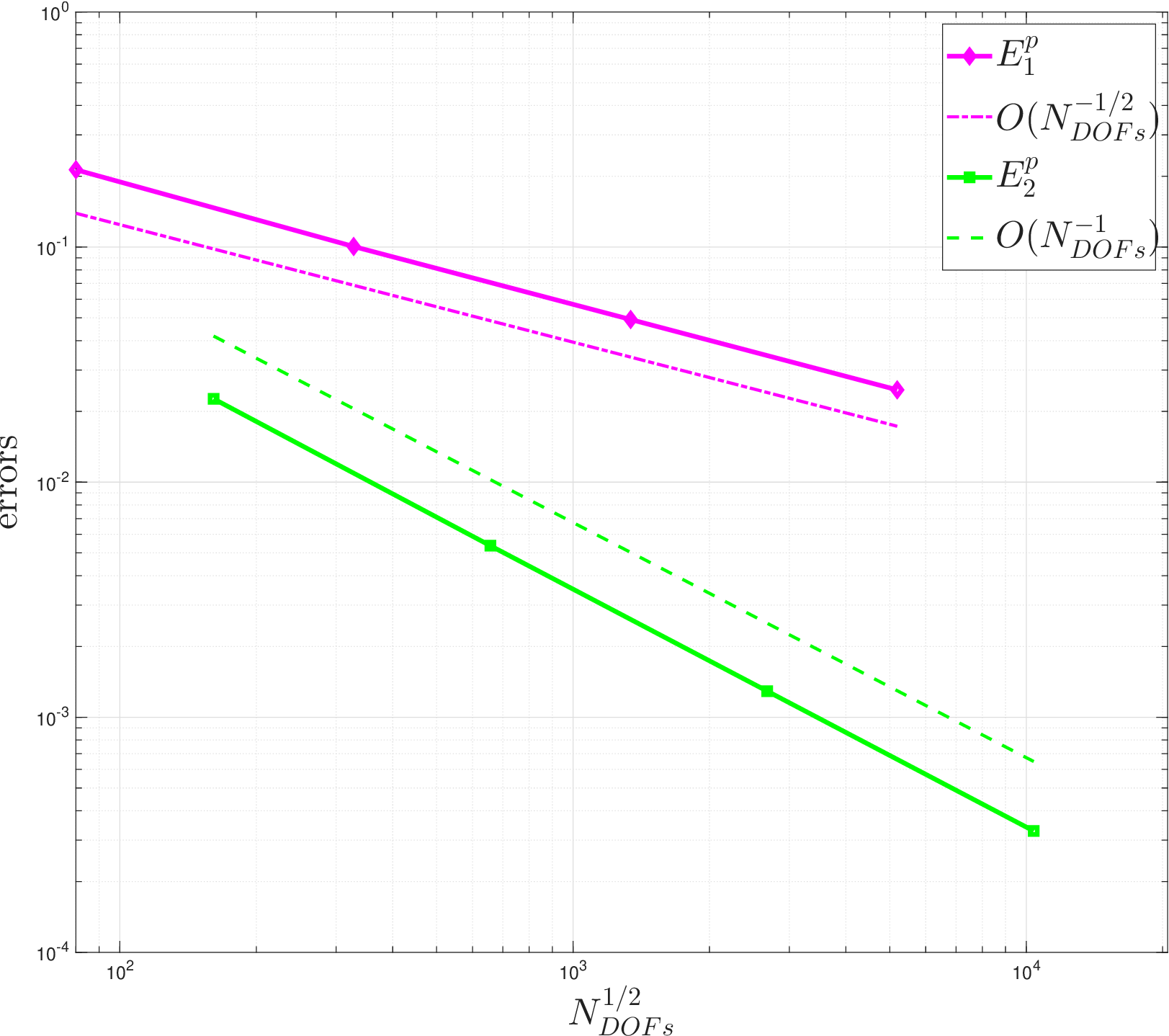}
\caption{Exact solution $u_2$ in~\eqref{u2};
$\h$-convergence of the flux (\emph{left-panel})
and the potential (\emph{right-panel})
for $k =1,2$, $\alpha=3$, and~$\beta=10$.
Standard fixed point iterative scheme~\eqref{algorithm:Darcy-Forchheimer-step}.}
    \label{fig:constant-cubic}
\end{figure}
As detailed in Remark~\ref{remark:galerkin-lowest-order},
the error estimates for the lowest order discretization
of~\eqref{discrete-formulation} are decoupled
whence the error of the fluxes is zero up to machine precision.
This is instead not the case for higher polynomial degrees;
cf. also Proposition~\ref{prop:error-estimate-velocity}.

\section{Conclusions} \label{section:conclusions}
In this work, we presented and analyzed a general-order
dual mixed nonconforming discretization
for a generalized Darcy--Forchheimer problem.
Our approach extends the framework in~\cite{Girault-Wheeler:2008},
along different directions:
\begin{itemize}
\item  ($\alpha$-2)-type Forchheimer nonlinearities
are allowed,
moving beyond the standard quadratic model;
\item the proposed method accommodates mixed,
inhomogeneous boundary conditions,
and permits permeability tensors with lower Lebesgue regularity;
\item we developed general-order schemes
based on piecewise-polynomial spaces of order $k-1$ for the fluxes
paired with Crouzeix-Raviart elements of order $k$ for the potentials;
\item based on novel Sobolev-trace inequalities
for broken spaces with constants
independent of the polynomial degree,
we established convergence to the exact solution
under low-regularity assumptions;
\item we further derived general-order error estimates
assuming additional regularity of the exact solution and data.
\end{itemize}
The theoretical findings were supported by numerical experiments
that assessed the performance of two iterative schemes
for different nonlinearity parameters $\alpha$,
orders~$k$, and parameters of the model.
While a standard fixed point scheme does not converge
for large nonlinearity parameters~$\alpha$,
a relaxed version may reach convergence
by suitably tuning the relaxation parameter.

The framework introduced in this work
lays the essential mathematical groundwork for our future objective:
the prospective development of $hp$-adaptive strategies
to efficiently discretize complex generalized Darcy--Forchheimer models.

\paragraph*{Acknowledgments.}
MB and LM have been partially funded by the
European Union (ERC, NEMESIS, project number 101115663);
views and opinions expressed are however those
of the authors only and do not necessarily reflect
those of the EU or the ERC Executive Agency.
MB and LM have been partially supported by the INdAM-GNCS project CUP E53C25002010001.
The authors are members of the Gruppo Nazionale Calcolo Scientifico-Istituto Nazionale di Alta Matematica (GNCS-INdAM).


{\footnotesize
\bibliography{bibliogr.bib}}

@article {Barrett-Liu:1993,
    AUTHOR = {Barrett, J. W. and Liu, W. B.},
     TITLE = {Finite element approximation of the {$p$}-{L}aplacian},
   JOURNAL = {Math. Comp.},
    VOLUME = {61},
      YEAR = {1993},
    NUMBER = {204},
     PAGES = {523--537}
}

@article {Burman-Hansbo:2005,
    AUTHOR = {Burman, E. and Hansbo, P.},
     TITLE = {Stabilized {C}rouzeix-{R}aviart element for the
              {D}arcy-{S}tokes problem},
   JOURNAL = {Numer. Methods Partial Differential Equations},
    VOLUME = {21},
      YEAR = {2005},
    NUMBER = {5},
     PAGES = {986--997}
}

@book {Brenner-Scott:2008,
    AUTHOR = {Brenner, S. C. and Scott, L. R.},
     TITLE = {The {M}athematical {T}heory of {F}inite {E}lement {M}ethods},
    SERIES = {Texts in Applied Mathematics},
    VOLUME = {15},
   EDITION = {Third},
 PUBLISHER = {Springer, New York},
      YEAR = {2008},
     PAGES = {xviii+397}
}

@misc{Antonietti-Bonetti-Botti:2025,
  author        = {Bonetti, S. and Botti, M. and Antonietti, P. F.},
  title         = {Conforming and discontinuous discretizations of non-isothermal {D}arcy-{F}orchheimer flows},
  year          = {2025},
  journal       = {ArXiv},
  howpublished  = {\url{http://arxiv.org/abs/2508.21630}}
}

@misc{Botti-Mascotto:2025-C,
author = {Botti, M. and Mascotto, L.},
title = {Trace inequalities for piecewise ${W}^{1,p}$ functions over general polytopic meshes},
journal={ArXiv},
year={2025},
howpublished={\url{http://arxiv.org/abs/2512.09752}}
}

@article{Botti-Mascotto:2025-B,
author = {Botti, M. and Mascotto, L.},
title = {Sobolev-{P}oincar\'e inequalities for piecewise ${W}^{1,p}$ functions over general polytopic meshes},
journal={Accepted on SIAM J. Numer. Anal.},
year={2026}
}

@book{Ern-Guermond:2021,
  title={Finite {E}lements {I}: {A}pproximation and {I}nterpolation},
  author={Ern, A. and Guermond, J.-L.},
  volume={72},
  year={2021},
  publisher={Springer Nature}
}

@article {Girault-Wheeler:2008,
    AUTHOR = {Girault, V. and Wheeler, M. F.},
     TITLE = {Numerical discretization of a {D}arcy-{F}orchheimer model},
   JOURNAL = {Numer. Math.},
    VOLUME = {110},
      YEAR = {2008},
    NUMBER = {2},
     PAGES = {161--198}
}

@incollection {Douglas-PaesLeme-Giorgi:1993,
    AUTHOR = {Douglas, Jr., J. and Paes-Leme, P. J. and Giorgi, T.},
     TITLE = {Generalized {F}orchheimer flow in porous media},
 BOOKTITLE = {Boundary value problems for partial differential equations and
              applications},
    SERIES = {RMA Res. Notes Appl. Math.},
    VOLUME = {29},
     PAGES = {99--111},
 PUBLISHER = {Masson, Paris},
      YEAR = {1993}
}

@article {Kim-Park:1999,
    AUTHOR = {Kim, M.-Y. and Park, E.-J.},
     TITLE = {Fully discrete mixed finite element approximations for non-{D}arcy flows in porous media},
   JOURNAL = {Comput. Math. Appl.},
    VOLUME = {38},
      YEAR = {1999},
    NUMBER = {11-12},
     PAGES = {113--129}
}

@article {Zhao-Chung-Park-Zhou:2021,
    AUTHOR = {Zhao, L. and Chung, E. T. and Park, E.-J. and Zhou, G.},
     TITLE = {Staggered {DG} method for coupling of the {S}tokes and {D}arcy-{F}orchheimer problems},
   JOURNAL = {SIAM J. Numer. Anal.},
    VOLUME = {59},
      YEAR = {2021},
    NUMBER = {1},
     PAGES = {1--31}
}

@article {Almonacid-Diaz-Gatica-Marquez:2020,
    AUTHOR = {Almonacid, J. A. and D\'iaz, H. S. and Gatica, G. N. and M\'arquez, A.},
     TITLE = {A fully mixed finite element method for the coupling of the {S}tokes and {D}arcy-{F}orchheimer problems},
   JOURNAL = {IMA J. Numer. Anal.},
    VOLUME = {40},
      YEAR = {2020},
    NUMBER = {2},
     PAGES = {1454--1502}
}

@article {Caucao-Gatica-Sandoval:2021,
    AUTHOR = {Caucao, S. and Gatica, G. N. and Sandoval, F.},
     TITLE = {A fully-mixed finite element method for the coupling of the {N}avier-{S}tokes and {D}arcy-{F}orchheimer equations},
   JOURNAL = {Numer. Methods Partial Differential Equations},
    VOLUME = {37},
      YEAR = {2021},
    NUMBER = {3},
     PAGES = {2550--2587}
}

@article {Rui-Pan:2017,
    AUTHOR = {Rui, H. and Pan, H.},
     TITLE = {A block-centered finite difference method for slightly compressible {D}arcy-{F}orchheimer flow in porous media},
   JOURNAL = {J. Sci. Comput.},
    VOLUME = {73},
      YEAR = {2017},
    NUMBER = {1},
     PAGES = {70--92}
}

@article {Rui-Pan:2012,
    AUTHOR = {Rui, H. and Pan, H.},
     TITLE = {A block-centered finite difference method for the {D}arcy-{F}orchheimer model},
   JOURNAL = {SIAM J. Numer. Anal.},
    VOLUME = {50},
      YEAR = {2012},
    NUMBER = {5},
     PAGES = {2612--2631}
}

@article {Zhang-Rui:2017,
    AUTHOR = {Zhang, J. and Rui, H.},
     TITLE = {A stabilized {C}rouzeix-{R}aviart element method for coupling {S}tokes and {D}arcy-{F}orchheimer flows},
   JOURNAL = {Numer. Methods Partial Differential Equations},
    VOLUME = {33},
      YEAR = {2017},
    NUMBER = {4},
     PAGES = {1070--1094}
}

@article {Wang-Rui:2015,
    AUTHOR = {Wang, Y. and Rui, H.},
     TITLE = {Stabilized {C}rouzeix-{R}aviart element for {D}arcy-{F}orchheimer model},
   JOURNAL = {Numer. Methods Partial Differential Equations},
    VOLUME = {31},
      YEAR = {2015},
    NUMBER = {5},
     PAGES = {1568--1588}
}

@article {Sayah-Semaan-Triki:2021,
    AUTHOR = {Sayah, T. and Semaan, G. and Triki, F.},
     TITLE = {Finite element methods for the {D}arcy-{F}orchheimer problem coupled with the convection-diffusion-reaction problem},
   JOURNAL = {ESAIM Math. Model. Numer. Anal.},
    VOLUME = {55},
      YEAR = {2021},
    NUMBER = {6},
     PAGES = {2643--2678}
}

@article {Amigo-Lepe-Otarola-Rivera:2025,
    AUTHOR = {Amigo, D. and Lepe, F. and Ot\'arola, E. and
              Rivera, G.},
     TITLE = {A virtual element method for a convective {B}rinkman-{F}orchheimer problem coupled with a heat equation},
   JOURNAL = {Comput. Math. Appl.},
    VOLUME = {191},
      YEAR = {2025},
     PAGES = {1--23}
}

@article{Badia-Carstensen-Martin-RuizBaier-VillaFuentes:2025,
    AUTHOR = {Badia, S. and Carstensen, C. and Martin, A. F. and Ruiz-Baier, R. and Villa-Fuentes, S.},
     TITLE = {A velocity-vorticity-pressure formulation for the steady
              {N}avier-{S}tokes-{B}rinkman-{F}orchheimer problem},
   JOURNAL = {Comput. Methods Appl. Mech. Engrg.},
  VOLUME = {447},
      YEAR = {2025},
     PAGES = {Paper No. 118343, 27}
}

@article {Triki-Sayah-Semaan:2024,
    AUTHOR = {Triki, F. and Sayah, T. and Semaan, G.},
     TITLE = {A posteriori error estimates for {D}arcy-{F}orchheimer's problem coupled with the convection-diffusion-reaction equation},
   JOURNAL = {Int. J. Numer. Anal. Model.},
    VOLUME = {21},
      YEAR = {2024},
    NUMBER = {1},
     PAGES = {65--103}
}

@article {Caucao-Gatica-Gatica:2023,
    AUTHOR = {Caucao, S. and Gatica, G. N. and Gatica, L. F.},
     TITLE = {A {B}anach spaces-based mixed finite element method for the stationary convective {B}rinkman-{F}orchheimer problem},
   JOURNAL = {Calcolo},
    VOLUME = {60},
      YEAR = {2023},
    NUMBER = {4},
     PAGES = {Paper No. 51, 32}
}

@article {Salas-Lopez-Molina:2013,
    AUTHOR = {Salas, J. J. and L\'opez, H. and Molina, B.},
     TITLE = {An analysis of a mixed finite element method for a {D}arcy-{F}orchheimer model},
   JOURNAL = {Math. Comput. Modelling},
    VOLUME = {57},
      YEAR = {2013},
    NUMBER = {9-10},
     PAGES = {2325--2338}
}

@article {Lopez-Molina-Salas:2008-2009,
    AUTHOR = {L\'opez, H. and Molina, B. and Salas, J. J.},
     TITLE = {Comparison between different numerical discretizations for a {D}arcy-{F}orchheimer model},
   JOURNAL = {Electron. Trans. Numer. Anal.},
    VOLUME = {34},
      YEAR = {2008/09},
     PAGES = {187--203}
}

@article {Park:2005,
    AUTHOR = {Park, E.-J.},
     TITLE = {Mixed finite element methods for generalized {F}orchheimer flow in porous media},
   JOURNAL = {Numer. Methods Partial Differential Equations},
    VOLUME = {21},
      YEAR = {2005},
    NUMBER = {2},
     PAGES = {213--228}
}

@article {Balci-Ortner-Storn,
    AUTHOR = {Balci, A. Kh. and Ortner, Ch. and Storn, J.},
     TITLE = {Crouzeix-{R}aviart finite element method for non-autonomous variational problems with {L}avrentiev gap},
   JOURNAL = {Numer. Math.},
    VOLUME = {151},
      YEAR = {2022},
    NUMBER = {4},
     PAGES = {779--805}
}

@article{Forchheimer:1901,
  title={Wasserbewegung durch {B}oden},
  author={Forchheimer, P. H.},
  journal={Z. Ver. Deutsch. Ing.},
  volume={45},
  pages={1782--1788},
  year={1901}
}

@article {Crouzeix-Raviart:1973,
    AUTHOR = {Crouzeix, M. and Raviart, P.-A.},
     TITLE = {Conforming and nonconforming finite element methods for solving the stationary {S}tokes equations. {I}},
   JOURNAL = {Rev. Fran\c. Automat. Informat. Rech.
              Op\'er. S\'er. Rouge},
    VOLUME = {7},
      YEAR = {1973},
     PAGES = {33--75}
}

@article{Raviart-Thomas:1977,
    AUTHOR = {Raviart, P.-A. and Thomas, J. M.},
     TITLE = {Primal hybrid finite element methods for 2nd order elliptic equations},
   JOURNAL = {Math. Comp.},
    VOLUME = {31},
      YEAR = {1977},  
    NUMBER = {138},
     PAGES = {391--413} 
}

@article {Ainsworth-Rankin:2008,
    AUTHOR = {Ainsworth, M. and Rankin, R.},
     TITLE = {Fully computable bounds for the error in nonconforming finite element approximations of arbitrary order on triangular elements},
   JOURNAL = {SIAM J. Numer. Anal.},
    VOLUME = {46},
      YEAR = {2008},
    NUMBER = {6},
     PAGES = {3207--3232}
}

@book {Showalter:1997,
    AUTHOR = {Showalter, R. E.},
     TITLE = {{M}onotone {O}perators in {B}anach {S}pace and {N}onlinear {P}artial
              {D}ifferential {E}quations},
    SERIES = {Mathematical Surveys and Monographs},
    VOLUME = {49},
 PUBLISHER = {American Mathematical Society, Providence, RI},
      YEAR = {1997}
}

@article {Baran-Stoyan:2007,
    AUTHOR = {Baran, \'A. and Stoyan, G.},
     TITLE = {Gauss-{L}egendre elements: a stable, higher order
              non-conforming finite element family},
   JOURNAL = {Computing},
    VOLUME = {79},
      YEAR = {2007},
    NUMBER = {1},
     PAGES = {1--21}
}

@article {Stoyan-Baran:2006,
    AUTHOR = {Stoyan, G. and Baran, \'A.},
     TITLE = {Crouzeix-{V}elte decompositions for higher-order finite
              elements},
   JOURNAL = {Comput. Math. Appl.},
  FJOURNAL = {Computers \& Mathematics with Applications. An International
              Journal},
    VOLUME = {51},
      YEAR = {2006},
    NUMBER = {6-7},
     PAGES = {967--986}
}

@article {Fortin-Soulie:1983,
    AUTHOR = {Fortin, M. and Soulie, M.},
     TITLE = {A nonconforming piecewise quadratic finite element on
              triangles},
   JOURNAL = {Internat. J. Numer. Methods Engrg.},
  FJOURNAL = {International Journal for Numerical Methods in Engineering},
    VOLUME = {19},
      YEAR = {1983},
    NUMBER = {4},
     PAGES = {505--520},
}

@article {Carstensen-Sauter:2022,
    AUTHOR = {Carstensen, C. and Sauter, S. A.},
     TITLE = {Critical functions and inf-sup stability of {C}rouzeix-{R}aviart elements},
   JOURNAL = {Comput. Math. Appl.},
    VOLUME = {108},
      YEAR = {2022},
     PAGES = {12--23}
}

@article{Bressan-Mascotto-Mosconi:2025,
author = {Bressan, A. and Mascotto, L. and Mosconi, M.},
title = {New {C}rouzeix-{R}aviart elements of even degree: theoretical aspects, numerical performance, and applications to the {S}tokes' equations},
journal={IMA J. Numer. Anal},
year={2026},
note={\url{https://doi.org/10.1093/imanum/draf091}}
}

@book {Berinde:2007,
    AUTHOR = {Berinde, V.},
     TITLE = {Iterative {A}pproximation of {F}ixed {P}oints},
    SERIES = {Lecture Notes in Mathematics},
    VOLUME = {1912},
   EDITION = {Second},
 PUBLISHER = {Springer, Berlin},
      YEAR = {2007},
}
\bibliographystyle{plain}

\appendix

\section{Properties of the nonlinear operator}  \label{appendix:properties-Acal}

This section is devoted with showing three properties of $\Acal(\cdot)$:
\begin{enumerate}
    \item monotonicity, in Section \ref{subsection:monotonicity-Acal};
    \item coercivity, in Section \ref{subsection:coercivity-Acal};
    \item hemi-continuity, in Section \ref{subsection:continuity-Acal}.
\end{enumerate}

\subsection{Monotonicity} \label{subsection:monotonicity-Acal}

Here,
we show that $\Acal(\cdot)$ is monotone.
\begin{proposition} \label{prop:monotonicity}
Let $\ybf$ be in $\Vbfcal$.
Then,
recalling that $\lambdamin$ is the smallest eigenvalue of~$\Kbb^{-1}$,
the following estimate holds true:
\begin{equation*} 
    \frac{\mu}{\rho}\lambdamin\Norm{\wbf-\vbf}_{\LtwoOmegad}
        \leq \int_{\Omega}(\Acal(\wbf+\ybf)-\Acal(\vbf+\ybf))\cdot(\wbf-\vbf)\dxbf
        \qquad\qquad
        \forall\ \wbf,\vbf \in \Vbfcal.
\end{equation*}
\end{proposition}
\begin{proof}
We define $\Jcal: \LalphaOmegad \to \Rbb$ as
\begin{equation*}
    \Jcal(\vbf)
        := \frac12 \frac{\mu}{\rho}\int_{\Omega}(\Kbbmo)\vbf\cdot\vbf \dxbf
            + \frac{1}{\alpha}\frac{\beta}{\rho}\int_{\Omega}|\vbf|^{\alpha} \dxbf
        \qquad\qquad\forall\ \vbf \in \LalphaOmegad.
\end{equation*}
The first Gateaux derivative $\Jcal'(\vbf): \LalphaOmegad \to \LalphaprimeOmegad$ reads 
\begin{equation} \label{Jcal-prime}
\langle \Jcal'(\vbf),\wbf \rangle = 
     \frac{\mu}{\rho}\int_{\Omega} (\Kbbmo)\vbf\cdot\wbf \dxbf
    + \frac{\beta}{\rho}\int_{\Omega} |\vbf|^{\alphamt} \vbf \cdot \wbf \dxbf
    \qquad\qquad
    \forall\ \vbf,\,\wbf \in \LalphaOmegad.
\end{equation}
The second Gateaux derivative $\Jcal''(\vbf)\wbf: \LalphaOmegad \to \LalphaprimeOmegad$ reads
\small\begin{equation} \label{Jcaldue}
\begin{split}
\langle \Jcal''(\vbf)\zbf, \wbf \rangle 
    &= \frac{\mu}{\rho}\int_{\Omega}(\Kbbmo)\zbf\cdot\wbf\dxbf \\
    & \quad    + \frac{\beta}{\rho}\int_{\Omega}
            (\alpha-2)|\vbf|^{\alpha-4}(\vbf\cdot\zbf)(\vbf\cdot\wbf)
            + |\vbf|^\alphamt\zbf\cdot\wbf\dxbf
\quad \forall\ \vbf \neq \zerobf,\wbf,\zbf \in \LalphaOmegad.
\end{split}
\end{equation}\normalsize
with
\begin{equation*}
    \langle \Jcal''(\zerobf)\wbf, \zbf \rangle
        = \frac{\mu}{\rho}\int_{\Omega}(\Kbbmo)\zbf\cdot\wbf\dxbf
    \qquad \forall\ \wbf,\zbf \in \LalphaOmegad.
\end{equation*}
Let $\ybf$ be fixed in  $\LalphaOmegad$,
and $\vbf$ and $\wbf$ be in $\LalphaOmegad$.
Define $\vbftilde := \vbf + \ybf$
and $\wbftilde := \wbf + \ybf$;
$\wbftilde-\vbftilde$ coincides with $\wbf - \vbf$.
Hence,
we use the fundamental theorem of calculus and get
\begin{equation*}
\begin{split}
\int_{\Omega} (\Acal(\wbftilde)-\Acal(\vbftilde))& \cdot (\wbf-\vbf) \dxbf 
    =\int_{\Omega} (\Acal(\wbftilde)-\Acal(\vbftilde) \cdot (\wbftilde-\vbftilde)) \dxbf \\
    &\overset{\eqref{Acal},\eqref{Jcal-prime}}{=} \langle \Jcal'(\wbftilde)-\Jcal'(\vbftilde),\wbftilde-\vbftilde \rangle
    = \int_{0}^1 \langle \Jcal''(\vbftilde+\theta(\wbftilde-\vbftilde))(\wbftilde-\vbftilde),\wbftilde-\vbftilde\rangle \dint{\theta}.
\end{split}
\end{equation*}
The integral on the right-hand side
is well-defined for all $\alpha > 2$ since
\begin{equation*}
|\vbf|^{\alpha-4}(\vbf\cdot\zbf)(\vbf\cdot\wbf) 
    \leq |\vbf|^{\alpha-2} |\zbf||\wbf|.
\end{equation*}
For all $\vbf$ and $\wbf$ in $\LalphaOmegad$
\begin{equation*}
 \langle \Jcal''(\vbf)\wbf,\wbf \rangle
    = \frac{\mu}{\rho}\int_{\Omega}(\Kbbmo\wbf)\cdot\wbf\dxbf
        + \frac{\beta}{\rho}\int_{\Omega}
            (\alphamt)|\vbf|^{\alpha-4}|\vbf\cdot\wbf|^2
            + |\vbf|^\alphamt|\wbf|\dxbf
    \geq \frac{\mu}{\rho} \lambdamin\Norm{\wbf}^2_{\L^2(\Omega)}.
\end{equation*}
In particular,
we deduce
\begin{equation*}
\int_{\Omega} (\Acal(\wbftilde)-\Acal(\vbftilde)) \cdot (\wbf-\vbf) \dxbf 
    \geq \frac{\mu}{\rho}\lambdamin\Norm{\wbftilde-\vbftilde}_{\L^2(\Omega)}
    =  \frac{\mu}{\rho}\lambdamin\Norm{\wbf-\vbf}_{\L^2(\Omega)}.
\end{equation*}
\end{proof}

\subsection{Coercivity} \label{subsection:coercivity-Acal}

Here,
we show that $\Acal(\cdot)$ is coercive.

\begin{proposition} \label{prop:coercivity}
Let $\ybf$ be in $\Vbfcal$.
We have
\begin{equation*}
\lim_{\Norm{\vbf}_{\LalphaOmegad}\to +\infty}
    \Big( \frac{1}{\Norm{\vbf}_{\LalphaOmegad}} \int_{\Omega} \Acal(\vbf + \ybf) \cdot\vbf \Big)
    = + \infty.
\end{equation*}
\end{proposition}
\begin{proof}
Let $\vbf$ and $\wbf$ be in $\Vbfcal$;
define $\vbftilde := \vbf + \ybf$.
The definition of $\vbftilde$
yields
\begin{equation} \label{int-Acalutilde-u}
\int_{\Omega} \Acal(\vbftilde) \cdot  \vbf \dxbf
    = \int_{\Omega} \Acal(\vbftilde)\cdot \vbftilde \dxbf
        - \int_{\Omega}\Acal(\vbftilde) \cdot \ybf \dxbf
    \overset{\eqref{Acal},\eqref{Jcal-prime}}{=} \langle \Jcal'(\vbftilde), \vbftilde \rangle
        - \int_{\Omega} \Acal(\vbftilde)\cdot \ybf \dxbf.
\end{equation}
We manipulate the two terms
on the right-hand side separately.
As for the former,
we have
\begin{equation}  \label{Jcalprime-utilde}
\langle \Jcal'(\vbftilde), \vbftilde \rangle
    = \frac{\mu}{\rho}\int_{\Omega}(\Kbbmo\vbftilde)\cdot \vbftilde\dxbf
        + \frac{\beta}{\rho}\int_{\Omega}|\vbftilde|^{\alphapt}\dxbf
    \geq \frac{\mu}{\rho} \lambdamin\Norm{\vbftilde}^2_{\LtwoOmegad}
        + \frac{\beta}{\rho}\Norm{\vbftilde}^{\alpha}_{\LalphaOmegad}.
\end{equation}
As for the latter,
H\"older's inequality twice
entails
\begin{equation} \label{Acaltilde-uell}
\int_\Omega \Acal(\vbftilde)\cdot \ybf \dxbf
    \leq \frac{\mu}{\rho}   \Norm{\Kbbmo}_{\Lbu^{\alphafracalphamt}(\Omega)}
                            \Norm{\vbftilde}_{\L^\alpha(\Omega)}
                            \Norm{\ybf}_{\L^\alpha(\Omega)} 
        + \frac{\beta}{\rho}\Norm{\vbftilde}^{\alphamo}_{\LalphaOmegad}
                            \Norm{\ybf}_{\LalphaOmegad}.
\end{equation}
Inserting \eqref{Jcalprime-utilde} and \eqref{Acaltilde-uell}
in \eqref{int-Acalutilde-u}
we have
\begin{equation} \label{Acalutilde-u}
\begin{split}
\int_{\Omega} \Acal(\vbftilde)  \cdot  \ybf \dxbf
    &\geq \frac{\mu}{\rho}\lambdamin \Norm{\vbftilde}_{\LtwoOmegad}^2
            \\
    &\quad + \frac{\beta}{\rho}\Norm{\vbftilde}^{\alphamo}_{\LalphaOmegad}
            \Big(\Norm{\vbftilde}_{\LalphaOmegad}
                     - \Norm{\ybf}_{\LalphaOmegad}
                    - \frac{\mu}{\beta}\Norm{\Kbbmo}_{\Lbu^{\alphafracalphamt}(\Omega)}\frac{\Norm{\ybf}_{\LalphaOmegad}}{\Norm{\vbftilde}^{\alphamt}_{\LalphaOmegad}}\Big) \\
    &\geq \frac{\beta}{\rho}\Norm{\vbftilde}^{\alphamo}_{\LalphaOmegad}
            \Big(\Norm{\vbftilde}_{\LalphaOmegad}
                     - \Norm{\ybf}_{\LalphaOmegad}
                    - \frac{\mu}{\beta}\Norm{\Kbbmo}_{\Lbu^{\alphafracalphamt}(\Omega)}\frac{\Norm{\ybf}_{\LalphaOmegad}}{\Norm{\vbftilde}^{\alphamt}_{\LalphaOmegad}}\Big) .
\end{split}
\end{equation}
Since $\frac{\Norm{\vbftilde}_{\LalphaOmegad}}{\Norm{\vbf}_{\LalphaOmegad}} \to 1$
when $\Norm{\vbf}_{\LalphaOmegad} \to \infty$,
we have, for $\Norm{\vbftilde}^{\alphamt}_{\LalphaOmegad}$
sufficiently large,
\begin{equation*}
\frac{1}{\Norm{\vbf}_{\LalphaOmegad}}
\int_{\Omega} \Acal(\vbftilde)  \cdot  \vbf \dxbf
    \gtrsim \frac{\beta}{\rho}\Norm{\vbftilde}^{\alphamt}_{\LalphaOmegad}
        \Big(\Norm{\vbftilde}_{\LalphaOmegad}
                - \Norm{\ybf}_{\LalphaOmegad}\Big( 1-\frac{\mu}{\rho}\frac{\Norm{\Kbbmo}_{\Lbu^{\alphafracalphamt}}}{\Norm{\vbftilde}^{\alphamt}_{\LalphaOmegad}}\Big) \Big)
    \gtrsim \Norm{\vbftilde}^{\alphamt}_{\LalphaOmegad}.
\end{equation*}
\end{proof}

\subsection{Hemi-continuity} \label{subsection:continuity-Acal}

Here,
we show that the mapping $\Acal(\cdot)$ is hemi-continuous.
\begin{proposition} \label{prop:hemi-continuity}
Given $\ybf$, $\vbf$, and $\wbf$ in $\Vbfcal$,
the mapping
\begin{equation*}
	t \to \int_\Omega \Acal(\ybf + \vbf + t\,\wbf) \cdot \wbf \dxbf
\end{equation*}
is continuous from $\Rbb$ into $\Rbb$.
\end{proposition}
\begin{proof}
Given $\ybf$, $\vbf$,
and~$\wbf$ in $\Lbf^{\alpha}(\Omega)$,
and defining $\vbftilde:= \vbf + \ybf$,
for all $t$ and $\tzero$ in $\Rbb$,
the fundamental theorem of calculus gives
\begin{equation*}
\begin{split}
\int_\Omega(\Acal(\vbftilde + \tzero\vbf) -  \Acal(\vbftilde+t\vbf))
\cdot \vbf \dxbf
&\overset{\eqref{Acal},\eqref{Jcal-prime}}{=}
\langle\Jcal'(\vbftilde+\tzero\vbf)-\Jcal'(\vbftilde+t\vbf), \vbf\rangle \\
&= -(t-\tzero)\int_0^1
\langle\Jcal''(\vbftilde - t\vbf-\theta(t-\tzero)\vbf)
    \vbf, \vbf \rangle \dint{\theta}.
\end{split}
\end{equation*}
H\"older's inequality in~\eqref{Jcaldue},
the fact that
$\vbftilde - t\vbf-\theta(t-\tzero)\vbf$
belongs to~$\Lbf^\alpha(\Omega)$,
and the fact that
$\Jcal''(\vbf)\wbf: \LalphaOmegad \to \LalphaprimeOmegad$
yield that the right-hand side of the above display
tends to zero as~$t$ goes to~$\tzero$.
The assertion follows.
\end{proof}

\end{document}